\documentclass[11pt,reqno]{amsart}
\usepackage{url}
\usepackage{bbm}
\usepackage[T1]{fontenc}
\usepackage{enumitem}
\usepackage{hyperref}
\hypersetup{
colorlinks   = true,
citecolor    = blue,
linkcolor    = blue
}

\usepackage{amsmath,amsfonts,amsthm,amssymb,color,tikz, comment, xcolor}
\usepackage{mathtools}
\usepackage{algorithm} 
\usepackage{algpseudocode} 

\makeatletter
\newenvironment{breakablealgorithm}
  {
   \begin{center}
     \refstepcounter{algorithm}
     \hrule height.8pt depth0pt \kern2pt
     \renewcommand{\caption}[2][\relax]{
       {\raggedright\textbf{\ALG@name~\thealgorithm} ##2\par}%
       \ifx\relax##1\relax 
         \addcontentsline{loa}{algorithm}{\protect\numberline{\thealgorithm}##2}%
       \else 
         \addcontentsline{loa}{algorithm}{\protect\numberline{\thealgorithm}##1}%
       \fi
       \kern2pt\hrule\kern2pt
     }
  }{
     \kern2pt\hrule\relax
   \end{center}
  }
\makeatother

\usepackage{cite}

\usepackage{mathrsfs}   
\usepackage[normalem]{ulem}
\usepackage{xfrac}
\usepackage{pdfsync}
\usepackage[font={scriptsize}]{caption}
\usepackage{subcaption}
\usepackage{environ}
\usepackage{tikz}
\usepackage[margin = 1in]{geometry}
\usepackage{graphicx}

\newcommand{\der}{\delta}

\newcommand{\1}{{\mathbbm{1}}}

\usepackage{accents}

\newcommand{\E}{\mathbb E}
\newcommand{\R}{\mathbb R}
\newcommand{\N}{\mathbb N}

\newcommand{\cb}{\mathcal B}

\newcommand{\ce}{\mathcal E}
\newcommand{\cf}{\mathcal F}

\newcommand{\al}{\alpha}

\newcommand{\ep}{\varepsilon}

\newcommand{\la}{\lambda}

\newcommand{\om}{\omega}
\newcommand{\oom}{\Omega}

\newcommand{\si}{\sigma}

\newcommand{\lp}{\left(}
\newcommand{\rp}{\right)}

\newcommand{\lln}{\left|}
\newcommand{\rrn}{\right|}

\numberwithin{equation}{section}

\newtheorem{theorem}{Theorem}[section]

\newtheorem{corollary}[theorem]{Corollary}

\newtheorem{definition}[theorem]{Definition}

\newtheorem{assumption}{Assumption}[section]
\newtheorem{lemma}[theorem]{Lemma}

\newtheorem{proposition}[theorem]{Proposition}

\theoremstyle{remark}
\newtheorem{remark}[theorem]{Remark}

\theoremstyle{remark}

\theoremstyle{definition}

\newcommand{\bean}{\begin{eqnarray*}}
\newcommand{\eean}{\end{eqnarray*}}
\newcommand{\ben}{\begin{enumerate}}
\newcommand{\een}{\end{enumerate}}
\newcommand{\beq}{\begin{equation}}
\newcommand{\eeq}{\end{equation}}

\title[Fluid limit for time-varying many server loss systems]{
Fluid Limits for Time-Varying Many-Server Loss Systems via Discontinuous Volterra Equations
}

\author[MINGRUI WANG]{MINGRUI WANG}
\address{(M. Wang) Harold and Inge Marcus Department of Industrial and Manufacturing
Engineering, The Pennsylvania State University, University Park, PA 16802 United States\\
}
\email{mvw5822@psu.edu}

\author[PRAKASH CHAKRABORTY]{PRAKASH CHAKRABORTY}
\address{(P. Chakraborty) Harold and Inge Marcus Department of Industrial and Manufacturing
Engineering, The Pennsylvania State University, University Park, PA 16802 United States\\
}
\email{prakashc@psu.edu}
\thanks{P. Chakraborty acknowledges support from the National Science Foundation under grant DMS-2153915, DARPA HR0011-25-3-E012 and Korb Early Career Professorship.}

\date{\today}

\begin{document}
\begin{abstract} 
This paper develops fluid limits for nonstationary many-server loss systems with general service-time distributions. For the zero-buffer $M_t/G/n/n$ queueing model, we prove a functional strong law of large numbers for the fraction of busy servers and characterize the limit by a nonlinear Volterra integral equation with discontinuous coefficients induced by instantaneous blocking. Well-posedness is established through an appropriate solution concept, yielding the time-varying acceptance probability without heuristic approximations. We then treat the finite-buffer $M_t/G/n/(n+b_n)$ regime, proving a functional strong law of large numbers for the triplet of fractions of busy servers, occupied buffers, and cumulative departures, whose limit satisfies a coupled system of three discontinuous Volterra equations capturing the interaction of service completions, buffer occupancy, and admission control at the capacity boundary. We establish well-posedness and convergence of the time-varying acceptance probability weakly in $L^1$ weighted by the arrival rate, and locally uniformly on the set of times at which the system is not full. Our theoretical results are supported by numerical simulations for both zero and finite-buffer regimes, illustrating the convergence of transient acceptance probabilities guaranteed by our theory. Finally, we use the fluid limits to derive optimal staffing and buffer-capacity for both time-varying loss systems.
\end{abstract}


\maketitle

\section{Introduction}

Many modern service systems operate with limited capacity, meaning customers are turned away or lost when the system is full. Classic examples include telephone networks with a fixed number of trunk lines \cite{Kelly1986,Eklundh1986,DaehyoungHong1986}, hospital or emergency units with limited beds \cite{deBruin2009,Bekker2016,Andersen2017}, wireless and optical networks with bandwidth and channel constraints \cite{Li1996,Vakilinia2015,Wang2014,Zalesky2007}, emergency services like ambulances, and self-driving cars \cite{Hampshire2020}. These loss models, sometimes called Erlang loss systems, have been studied extensively under steady-state conditions. In fact, the famous Erlang-B formula \cite{erlang1917} developed over a century ago for telephone traffic gives the steady-state blocking probability for an $M/M/n/n$ queue and remains a cornerstone result in stationary loss models. Yet real-world systems are rarely stationary: arrival rates and service demands fluctuate over time, and service durations are not necessarily memoryless. As a result, steady-state measures often fail to capture short-term dynamics, leading to inefficient or unstable operational decisions. Nonstationary, non-Markovian loss systems such as $M_t/G/n/n$ queues are significantly more challenging to analyze, and closed-form transient performance formulas are virtually impossible to obtain. This difficulty motivates the use of stochastic-process approximations for performance analysis, especially in many-server regimes where the number of servers $n$ is large.

Fluid limits or functional strong laws of large numbers (FSLLN) provide deterministic approximations to many-server queueing systems by tracking the scaled system state as $n\to\infty$. These limits reveal the macroscopic law of motion of complex stochastic systems. Foundational work such as \cite{halfin1981heavy, mandelbaum1998} introduced asymptotic techniques for many-server systems and Markovian service networks.  Subsequent research established fluid and diffusion limits under increasingly general conditions, including time-varying arrivals and non-exponential service times \cite{reed2009g, kaspi2011law, liu2012many, liu2014many, zhang2013fluid}.
In contrast to these limit theorems, an extensive applied literature has developed practical approximations and staffing heuristics for time-varying service systems. Related work, including \cite{jennings1996, green1991, green2007, whitt2019, whitt2018time}, proposes pointwise-stationary (PSA), modified-offered-load (MOL), and other transient approximations aimed at dynamic staffing, capacity planning, and transient performance evaluation. These studies underscore the need for rigorous transient characterizations that connect operational heuristics with asymptotic theory.

\subsection{Overview of Approach and Key Insights.}
This paper develops a rigorous fluid-limit framework for analyzing time-varying many-server \emph{loss systems}. Specifically, we study a sequence of systems with nonhomogeneous Poisson process (NHPP) arrivals and general service-time distributions, where both the number of servers and the arrival rate scale linearly with system size. The resulting limit is characterized by a nonlinear Volterra integral equation (VIE) that captures the transient evolution of the system’s occupancy and, crucially, yields the time-dependent blocking and acceptance probabilities in the large-scale regime. Our work builds upon \cite{chakraborty2021many}, which established a fluid limit for the nonstationary many-server $M_t/G/n/n$ loss system using a semimartingale representation of the instantaneous acceptance mechanism. We enhance that framework by introducing a refined convergence proof based on the discontinuous Volterra equation methodology of \cite{kiffe1979discontinuous}, ensuring well-posedness of the limit even under nonsmooth boundary dynamics.

A distinguishing feature of our analysis is the emergence of nonlinear Volterra equations with \emph{discontinuous coefficients}, induced by the instantaneous blocking constraint at full capacity. This structure departs sharply from classical Markovian formulations and provides a new analytic mechanism to capture threshold-type, transient blocking phenomena in nonstationary systems. The discontinuity is not merely a technical complication. It serves as the deterministic counterpart of the system’s stochastic acceptance barrier and encodes the operational behavior of loss systems under time-varying load.

From a methodological standpoint, our results bridge three traditions in the study of nonstationary queues: (i) steady-state or quasi-stationary approximations such as the Erlang-B, PSA, and MOL methods \cite{green1991, massey1994analysis, green2007, whitt2017many}; (ii) computational and moment-based approximation methods, including cumulant and truncated-ODE approaches for time-varying loss and many-server systems \cite{pender2015, pender2017, jennings1996, whitt2019, whitt2018time}; and (iii) rigorous asymptotic limit theorems \cite{halfin1981heavy, mandelbaum1998, reed2009g}. The fluid model derived here serves simultaneously as a limit theorem and a computational engine. It is a deterministic equation directly solvable by numerical methods, providing transient blocking probabilities without Monte Carlo simulations. This connection between rigorous scaling limits and practical performance computation strengthens the link between applied probability and operational analysis, particularly in time-dependent service environments such as healthcare scheduling, cloud service provisioning, and mobility-on-demand platforms. Accurate transient blocking or acceptance probabilities support dynamic staffing and admission-control decisions  under fluctuating demand. Whereas traditional time-varying approximations assume local equilibrium, our limit provides a theoretically consistent foundation for approximating time-varying acceptance probabilities under nonstationary demand, which is central to time-dependent operations management. 

\subsection{Contributions.}
We establish functional strong laws of large numbers for nonstationary many-server loss systems under general service-time distributions. Both the zero-buffer ($M_t/G/n/n$) and finite-buffer ($M_t/G/n/(n+b_n)$) systems are analyzed in a common framework that scales the number of servers and the arrival rate proportionally with system size \footnote{we believe our analysis can be readily extended to time-varying piecewise constant service and buffer capacities. However, for simplicity, we consider the case where both are constant.}. The resulting limits are deterministic trajectories described by nonlinear Volterra integral equations (VIEs) that capture the transient evolution of the system occupancy and yield the associated time-dependent blocking and acceptance probabilities.

\noindent    
\emph{(i) Zero-buffer systems.}  
For the $M_t/G/n/n$ model with nonhomogeneous Poisson arrivals of rate $\lambda(\cdot)$ and i.i.d.\ service times with distribution $G$, let $N_t^n$ denote the number in system and $\bar{N}_t^n = N_t^n/n$ its scaled occupancy. We prove that when the system starts empty $\bar{N}^n$ converges almost surely to a deterministic function $\rho$ (see Theorem~\ref{thm:fluidlimitini0uni} for a more general and precise formulation) such that $\rho$ solves the following discontinuous VIE:
\beq\label{eq:vie-1}
\rho_t = \int_0^t \mathbbm{1}_{\{\rho_{u-}<1\}}\bar{G}(t-u)\la(u) du,
\eeq
where $\bar{G}$ is the service-time survival function.  The integral equation above has discontinuous coefficients due to the indicator $\mathbbm{1}_{\{\rho_{u-}<1\}}$, reflecting instantaneous blocking at capacity.  We refine the convergence analysis of \cite{chakraborty2021many} by introducing the discontinuous Volterra solution concept of \cite{kiffe1979discontinuous}. Specifically, $\rho$ solves \eqref{eq:vie-1} if there exists an auxiliary acceptance function $w(\cdot)$ 
with $\mathbbm{1}_{\{\rho_{u-}<1\}}\leq w(u)\leq1$ for a.e. $u$
such that
$$
\rho_t = \int_0^t w(u) \bar{G}(t-u) \la(u) du,
$$
which ensures well-posedness even under nonsmooth boundary dynamics. As a corollary (see Corollary~\ref{cor:acceptprob} for a precise formulation), we obtain that the acceptance probability 
\beq\label{eq:acceptance-limit-1}
\mathbb{P}(\bar{N}_{t-}^n < 1) \la(t) \rightharpoonup w(t) \la(t) \quad \text{in }L^1[0,T].
\eeq
In addition, we identify $w(t)=1$ when $\rho_t<1$ and $w(t)\la(t)=d(t)$ when $\rho_t=1$, where $d(t)$ is the instantaneous departure rate, which agrees with heuristic expectations. 
Thus our analysis yields a rigorous FSLLN that provides a direct functional relationship between the time-varying acceptance (or blocking) probability and the system primitives through a deterministic limit equation~\eqref{eq:acceptance-limit-1}.

\noindent
\emph{(ii) Finite-buffer systems.}  
We extend the analysis to the $M_t/G/n/(n+b_n)$ model, where the buffer size $b_n$ may scale with $n$ so that $b_n/n \to \beta \in (0,\infty)$. Denote by $\bar{S}_t^n$, $\bar{Q}_t^n$, and $\bar{D}_t^n$ the scaled numbers of busy servers, queued customers, and cumulative departures, respectively. We prove (see Theorem~\ref{thm:fluidlimitbuf} for a more precise formulation) that, for a system starting empty, the joint limit of these processes $(\bar{S}^n,\bar{Q}^n,\bar{D}^n)$ is given by the tuple $(\rho,\eta,D)$ that satisfies a system of three coupled nonlinear VIEs:
\begin{align*}
&\rho_{t}=\int_{0}^{t} \mathbbm{1}_{\{\rho_{u-}<1\}} \bar{G}(t-u) \lambda(u) d u  +\int_{0}^{t} \mathbbm{1}_{\{\eta_{u-}>0\}} \bar{G}(t-u) d(u) d u, \nonumber\\
& \eta_{t}=\int_{0}^{t} \lp \mathbbm{1}_{\{\eta_{u-}<\beta\}} -  \mathbbm{1}_{\{\rho_{u-}<1\}}\rp \lambda(u) d u  -\int_{0}^{t} \mathbbm{1}_{\{\eta_{u-}>0\}} d(u) d u, \nonumber\\
& D_t=\int_{0}^{t} \mathbbm{1}_{\{\rho_{u-}<1\}} G(t-u) \lambda(u) d u  +\int_{0}^{t} \mathbbm{1}_{\{\eta_{u-}>0\}} G(t-u) d(u) d u,
\end{align*}
where $d=D'$ denotes the fluid departure rate. These equations jointly describe the evolution of service completions, queue occupancy, and admission control at the boundary. As in the zero-buffer case, they are interpreted through auxiliary acceptance functions $(w^1,w^2,w^3)$. 
The limit $(\rho, \eta, D)$ is unique, as are $w^1\la+w^2 d$ and $w^3\la$.
ensuring existence and uniqueness of the limit. The resulting acceptance probability satisfies in the sense of Corollary~\ref{cor:acceptprobn}
$$
\mathbb{P}\lp \bar{Q}_{t-}^n < \frac{b_n}{n}\rp \la(t) \rightharpoonup w^3(t)\la(t) \quad \text{in }L^1[0,T], 
$$
with $\mathbb{P}(\bar{Q}_{t-}^n < {b_n}/{n}) \to 1$ locally uniformly on $\{t \in [0,T]: \eta_t < \beta\}$, where 
$w^3(t) = 1$ when $\eta_t < \beta $ and $w^3(t)\la(t) = d(t)$ when $\eta_t = \beta$.
This extension introduces significant technical challenges beyond the zero-buffer case, requiring new arguments to handle the emerging coupled nonlinear Volterra systems.

\noindent
\emph{(iii) Analytical and operational significance.}  
The discontinuous Volterra framework developed here characterizes  transient blocking and acceptance probabilities for time-varying hard-loss service systems with general service, with and without a buffer. It yields a numerically tractable representation: the limit equations can be solved efficiently via numerical methods, enabling direct computation of transient performance measures without simulation. Beyond analytical clarity, the framework serves as a practical foundation for operational decision-making. We demonstrate its use for optimal staffing and buffer capacity design, showing how the deterministic fluid model can approximate system-level performance with high accuracy. 
Overall, these results unify the transient analysis of Erlang loss and delay systems and offer a theoretically grounded computational tool for performance evaluation and capacity design in applications such as call centers, hospitals, and cloud-service platforms.

Together, the zero and finite-buffer results form an integrated theory of time-varying many-server systems. The discontinuous Volterra formulation opens the door to higher-order diffusion refinements and control-theoretic extensions.

\subsection{Paper Organization.} 
The remainder of the paper is organized as follows. Section~\ref{sec:prelim} presents the preliminaries, including notations, key probability results, and the analytical framework for discontinuous Volterra integral equations (VIEs). Section~\ref{sec:zero-buffer} focuses on the zero-buffer $M_t/G/n/n$ system. We derive the fluid limit, prove the functional strong law of large numbers, and establish convergence of the time-varying acceptance and blocking probabilities. Section~\ref{sec:finite-buffer} extends the analysis to the finite-buffer $M_t/G/n/(n+b_n)$ model. Here, we characterize the joint fluid limit of the fractions of busy servers, occupied buffers, and departures as the solution to a system of coupled Volterra integral equations, and we prove convergence of the corresponding acceptance and blocking probabilities. Section~\ref{sec:numerics} provides numerical experiments that illustrate the accuracy and interpretability of the fluid-limit approximation across both zero- and finite-buffer regimes, in addition to optimal server and capacity applications. A brief concluding Section~\ref{sec:conclusion} summarizes the findings and outlines potential extensions, including diffusion refinements and control applications.

\section{Preliminaries and Notations}\label{sec:prelim}

In this section we present some preliminary results that will be useful later on.

\subsection{Counting Measure.} 
Let $(\oom, \cf, \mathbb{F}= ( \cf_t )_{t \geq 0}, \mathbb{P})$ be a filtered probability space. Let $(N_t)_{t \geq 0}$ be a point process given by a sequence $(T_n)_{n \geq 1}$ of jump times, that is
$$
N_t :=N\left((0,t]\right)= \sum_{i=1}^{\infty} \mathbbm{1}_{\{T_i \leq t\}},
$$
where $N(\cdot)=\sum_{n\geq 1}\delta_{T_n}(\cdot)$ is the corresponding counting measure and $\delta_y$ stands for the Dirac measure at $y$. Suppose in addition the $n^{\textnormal{th}}$ jump time or arrival $T_n$ has a corresponding {mark or} random variable $Z_n$ taking values in some measurable space $(E, \ce)$. Then $(T_n, Z_n)_{n \geq 1}$ is called an $E$-marked point process. Let $\mathcal{M}^N(\cdot \times \cdot)$ be the counting measure of the marked point process, that is, for each $C \in\cb(\R),L \in \ce$
$$
\mathcal{M}^N(C\times L) =\sum_{i=1}^{\infty}\mathbbm{1}_{\{T_i \in C\}} \mathbbm{1}_{\{Z_i \in L\}} .
$$
This implies for measurable functions $\varphi:(\mathbb{R}, \mathcal{B}(\mathbb{R})) \times (E,\ce)\rightarrow(\overline{\mathbb{R}}, \mathcal{B}(\overline{\mathbb{R}}))$
$$
\int_0^t \int_{E} \varphi(u, z) \mathcal{M}^N(du \times dz) = \sum_{i=1}^{\infty} \varphi(T_i, Z_i) \mathbbm{1}_{\{T_i \leq t\}}.
$$
We recall the notions of intensity measure and intensity function following \cite{bremaud2024introduction}.
\begin{definition}{\cite[Def.\ 10.2.13]{bremaud2024introduction}}\label{subsec:intensity}
The intensity measure $\nu$ of a locally finite point process $N$ on $\R^m$ is defined by 
\begin{equation*}
	C \mapsto \nu(C):=\E[N(C)] \quad\left(C \in \mathcal{B}\left(\mathbb{R}^m\right)\right)	.
\end{equation*}
In addition, if $\nu$ is of the form $\nu(C)=\int_C \zeta(x)dx$ for some non-negative measurable function $\zeta:\R^m \rightarrow\R$, the point process $N$ is said to admit the intensity function $\zeta(x)$.
\end{definition} 
For a point process $N$ with intensity measure $\nu$ and intensity function $\zeta$, we introduce the Campbell's formula from \cite[Thm.~10.2.15]{bremaud2024introduction}:
\begin{theorem}\label{thm:campbell}
For all measurable functions $\varphi: \mathbb{R}^m \rightarrow \mathbb{R}$ which are non-negative or $\nu$-integrable, the integral $\int_{\R^m}\varphi(x)N(dx)$ is well defined and
$$
\mathbb{E}\left[\int_{\R^m}\varphi(x)N(dx)\right]=\int_{\R^m}\varphi(x)\nu(dx)=\int_{\R^m}\varphi(x)\zeta(x)dx .
$$
In particular, $\int_{\R^m}\varphi(x)N(dx)$ is a.s. finite if $\varphi$ is $\nu$-integrable.
\end{theorem}
\subsection{Discontinuous Volterra Integral Equation.} We recall the notion of solution for discontinuous Volterra integral equations, as presented in \cite{kiffe1979discontinuous}. First, we introduce some related notations. For any $p \in L_{\text{loc}}^{\infty}(-\infty, \infty)$ and any $\epsilon > 0$, define:
		\begin{equation*}
			\underline{p}_\epsilon(t)=\underset{|t-s|<\epsilon}{\operatorname{ess} \inf } ~p(s), \quad \bar{p}_\epsilon(t)=\underset{|t-s|<\epsilon}{\operatorname{ess} \sup } ~p(s) .
		\end{equation*}
In addition, for $t \in [0,T]$ define:
\begin{equation}\label{eq:g-bar}
	\underline{p}(t)=\lim _{\epsilon \rightarrow 0} \underline{p}_\epsilon(t), \quad \bar{p}(t)=\lim _{\epsilon \rightarrow 0} \bar{p}_\epsilon(t) .
\end{equation}
\begin{definition}\label{def:sol_Vol}
Let $p :[0,\infty)\rightarrow\R$ and $q :[0,T]\rightarrow\R$ be bounded functions. Furthermore, let $a \in L^1[0, T]$. A pair of functions $x:[0,T]\rightarrow\R$ and $z:[0,T]\rightarrow\R$ is said to be a solution of the Volterra integral equation
\begin{equation*}
	x(t)+\int_0^t a(t-s) p(x(s)) d s=q(t), \quad 0 \leq t \leq T
\end{equation*} 
if $x$ and $z$ are bounded and
	\begin{equation*}
		\underline{p}(x(t)) \leq z(t) \leq \bar{p}(x(t)) \quad \text { a.e., } \quad 0 \leq t \leq {T}, 
	\end{equation*}
such that
$$
x(t)+\int_0^t a(t-s) z(s) d s=q(t), \quad 0 \leq t \leq T.
$$
\end{definition}
\begin{remark}
We point out to the reader that the assumption on $p,q$ and $a$ can be relaxed or modified as done in \cite{kiffe1978existence,kiffe19792}. Our exposition here is chosen for simplicity and the specific processes we encounter later.
\end{remark}
\subsection{Notations and additional lemma.} We use $\rightharpoonup$ and $\stackrel{*}{\rightharpoonup}$ respectively to denote weak and weak-star convergences in general function spaces.

The following lemma helps upgrade pointwise convergence to uniform convergence, and is used in the sequel.
\begin{lemma}\label{lem:to-be-written}
Let $F,H:[0,T] \to \R$ be non-decreasing with $F(0)=H(0)$, and let $0=s_0 < s_1 < \cdots < s_m=T$. Then
$$
\sup_{t \in [0,T]} |F(t)-H(t)| \leq \max_{0\leq j \leq m} |F(s_j)-H(s_j)| + \max_{1 \leq j \leq m} (H(s_j)-H(s_{j-1})). 
$$
\end{lemma}
\begin{proof}
Let $t \in [s_{j-1}, s_j]$. By monotonicity of $F$ and $H$,
$$
F(t) - H(t) \leq F(s_j) - H(s_{j-1}) = (F(s_j) - H(s_j)) + (H(s_j)-H(s_{j-1})).
$$
Similarly $F(t)-H(t) \geq (F(s_{j-1}) - H(s_{j-1})) - (H(s_j)-H(s_{j-1}))$.
\end{proof}

\section{Fluid limit for zero-buffer loss system}\label{sec:zero-buffer}
\subsection{Setup.}
In this section, we introduce the zero-buffer loss queueing model. We consider a sequence of queueing systems indexed by $n$, subject to the following assumptions.
\begin{assumption}\label{asm:queue0}
	Consider a $M_t / G / n / n$ loss queueing system; namely, a queueing system with
	\begin{enumerate}[label=\roman*.]
		\item a nonhomogeneous Poisson arrival process $A^n$ with rate or intensity function {$ {n}\lambda(\cdot)$}, where {$\lambda$ is} locally integrable;
		\item general customer service times sampled {independently} from a distribution $G$ with density $g$;
		\item the system has $n$ servers and zero buffer or waiting space. That is, when all $n$ servers are busy, incoming customer arrivals are lost. Equivalently, the customers can be thought to have $0$ patience. 
	\end{enumerate}
\end{assumption}

\begin{remark}
{Note that the intensity function corresponding to the arrival process $A^n$ could be extended to a more general $\la_n$ for all $n$, such that $\la_n/n\rightarrow\la$ under some topology. This generalization should be an easy extension and not considered in this article to keep considerations simpler.} 
\end{remark}

\tikzset{every picture/.style={line width=0.75pt}} 
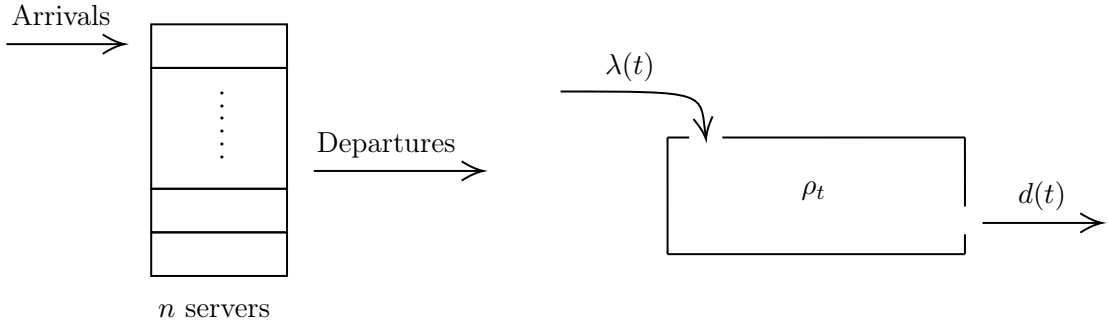
\begin{figure}[h!]
\centering
{%
\begin{tikzpicture}[x=0.75pt,y=0.75pt,yscale=-1,xscale=1]

\draw   (110,62) -- (178.5,62) -- (178.5,84) -- (110,84) -- cycle ;
\draw   (110,84) -- (178.5,84) -- (178.5,145) -- (110,145) -- cycle ;
\draw   (110,167) -- (178.5,167) -- (178.5,189) -- (110,189) -- cycle ;
\draw   (110,145) -- (178.5,145) -- (178.5,167) -- (110,167) -- cycle ;
\draw    (37,72) -- (94.5,72) ;
\draw [shift={(96.5,72)}, rotate = 180] [color={rgb, 255:red, 0; green, 0; blue, 0 }  ][line width=0.75]    (10.93,-4.9) .. controls (6.95,-2.3) and (3.31,-0.67) .. (0,0) .. controls (3.31,0.67) and (6.95,2.3) .. (10.93,4.9)   ;
\draw    (370.43,178) -- (520.4,178) ;
\draw    (520.41,153.92) -- (520.39,119.1) ;
\draw    (370.43,178) -- (370.41,119.12) ;
\draw    (520.39,119.1) -- (398.06,119.12) ;
\draw    (520.4,167.98) -- (520.4,178) ;
\draw    (370.41,119.12) -- (381.31,119.12) ;
\draw    (316.5,95.9) .. controls (392.08,95.9) and (386.76,95.9) .. (389.57,118.19) ;
\draw [shift={(389.81,119.95)}, rotate = 262.02] [color={rgb, 255:red, 0; green, 0; blue, 0 }  ][line width=0.75]    (10.93,-4.9) .. controls (6.95,-2.3) and (3.31,-0.67) .. (0,0) .. controls (3.31,0.67) and (6.95,2.3) .. (10.93,4.9)   ;
\draw    (529.25,162.24) -- (587.5,162.24) ;
\draw [shift={(589.5,162.24)}, rotate = 180] [color={rgb, 255:red, 0; green, 0; blue, 0 }  ][line width=0.75]    (10.93,-4.9) .. controls (6.95,-2.3) and (3.31,-0.67) .. (0,0) .. controls (3.31,0.67) and (6.95,2.3) .. (10.93,4.9)   ;
\draw    (192,136) -- (275.5,136) ;
\draw [shift={(277.5,136)}, rotate = 180] [color={rgb, 255:red, 0; green, 0; blue, 0 }  ][line width=0.75]    (10.93,-4.9) .. controls (6.95,-2.3) and (3.31,-0.67) .. (0,0) .. controls (3.31,0.67) and (6.95,2.3) .. (10.93,4.9)   ;

\draw (150,93) node [anchor=north west][inner sep=0.75pt]  [rotate=-90]  {$\cdots \cdots $};
\draw (112,202) node [anchor=north west][inner sep=0.75pt]   [align=left] {$ n$ servers};
\draw (38,51) node [anchor=north west][inner sep=0.75pt]   [align=left] {Arrivals};
\draw (337.39,74.69) node [anchor=north west][inner sep=0.75pt]    {$\lambda ( t)$};
\draw (436.5,140) node [anchor=north west][inner sep=0.75pt]    {$\rho _{t}$};
\draw (545.67,139.72) node [anchor=north west][inner sep=0.75pt]    {$d( t)$};
\draw (192,115) node [anchor=north west][inner sep=0.75pt]   [align=left] {Departures};

\end{tikzpicture}
}
\caption{{Zero-buffer loss system and its fluid model}}
\end{figure}

\subsection{Fraction of Occupied Servers or Scaled Number in System.}
Consider the $M_t/G/n/n$ loss system as in Assumption~\ref{asm:queue0}. 
Let $T_i$ and $V_i$ represent the arrival and service times, respectively, of the $i$-th customer. 
Let $N_t^n$ denote the number of occupied servers, or equivalently the number of customers in the system at time $t$. Denote $\bar{N}^n_t:= \frac{N_t^n}{n}$ to be the fraction of occupied servers or the $n-$scaled number of customers in the system at time $t$. Also, let $\mathcal{F}_t^{n}$ be the filtration generated by the marked point process, $\cf_t^n := \si(\{(T_i, V_i): T_i \leq t\})$.

For the sake of simplicity, we first assume that the system starts empty, that is, the number of customers in the system at time $0$ is zero. In the sequel, we will relax this assumption. 

Observe that the number of busy servers at time $t$ consists of all arrivals to the system such that all of the following conditions are met: 
\begin{enumerate}[label = (\roman*), wide, labelwidth=!, labelindent=0pt]
\item the customer arrival occurs {at or} prior to time $t$, 
\item the number of occupied servers upon the customer's arrival is less than $n$, and
\item the remaining service time of this customer at time $t$ is positive, that is, the customer is yet to depart the system. 
\end{enumerate} 
For the $i-$th customer arriving to the system, these conditions correspond to $\{T_i \leq t\}$, $\{N_{T_i-}^{n} < n\}$ or $\{\bar{N}_{T_i-}^{n} < 1\}$, and $\{V_i > t-T_i\}$ respectively. Consequently, the number of customers at time $t$ satisfies
\begin{equation}\label{eq:N^n}
	N_{t}^{n}=\sum_{i=1}^{\infty} \mathbbm{1}_{\{T_{i} \leq t\} } \mathbbm{1}_{\{N_{T_{i}-}^{n}<n\}} \mathbbm{1}_{\{V_{i} >t-T_{i}\}} .
\end{equation}
On scaling \eqref{eq:N^n} by $n$, we obtain that the fraction of occupied servers satisfies
\begin{equation}\label{eq:rho^n1}
	\bar{N}_t^n=\frac{1}{n}\sum_{i=1}^{\infty}  \mathbbm{1}_{\{T_{i} \leq t\} } \mathbbm{1}_{\{\bar{N}_{T_{i}-}^{n}<1\}} \mathbbm{1}_{\{V_{i} >t-T_{i}\}} .
\end{equation}
Crucially, observe that $N^n$ or $\bar{N}^n$ given by \eqref{eq:N^n}-\eqref{eq:rho^n1} are given by integral equations whose evolution depends, in general, on its history. As such, these processes are non-Markovian and in this work we provide a way of obtaining scaling limits of such processes arising out of loss queueing systems. To that effect, we work with the scaled process $\bar{N}^n$ and obtain a representation using random measures. Denote 
\beq\label{eq:W^n}
W_n(t,u,x)=\frac{1}{n} \mathbbm{1}_{\{\bar{N}_{u-}^{n}<1\}} \mathbbm{1}_{\{x >t-u\}} \mathbbm{1}_{\{u \leq t\} }.
\eeq
Then relation~\eqref{eq:rho^n1} can be represented as
\begin{equation}\label{eq:rho^n0}
\bar{N}_t^n=\int_0^t\int_\R W_n(t,u,x) \mathcal{M}^n (du,dx),
\end{equation}
where $\mathcal{M}^n$ is the counting measure associated with the marked point process of the arrival and service time pair $(T_i,V_i)$. Taking expectation, we have by Theorem~\ref{thm:campbell} that 
\begin{equation*}
	\E\left[\int_0^t\int_\R W_n(t,u,x) \mathcal{M}^n (du,dx)\right]=\E\left[\int_0^t\int_\R W_n(t,u,x) n\lambda(u)g(x)dxdu\right].
\end{equation*}
Denote $\mathcal{M}^n_*$ to be the compensated random measure:
\begin{equation}\label{eq:M-comp}
	\mathcal{M}^n_*=\mathcal{M}^n-\mathcal{M}^n_c,
\end{equation}
where $\mathcal{M}^n_c(du, dx):=\E\left[\mathcal{M}^n(du,dx)\right]=n\lambda(u)g(x)dudx$.

Having obtained an integral representation for the fraction of occupied servers in \eqref{eq:rho^n0}, our goal is to exploit this relation to obtain the limit of the stochastic process $\{ \bar{N}_t^n, t \geq 0 \}$ as $n$ goes to infinity. 
 We begin with a result proving convergence along a subsequence.
\begin{proposition}\label{prop:weaklimits0}
	Let Assumption~\ref{asm:queue0} hold. Assume that the system starts empty, that is $\rho_0^n = 0$ for all $n$. Then 
\begin{enumerate}[label = (\roman*), wide, labelwidth=!, labelindent=0pt]
\item For any $T > 0$, almost all $\om \in \oom$ and any subsequence, there exists a further subsequence $(r_k)$, possibly depending on $\om$, and a continuous{, possibly stochastic} process $\rho$ such that 
	\begin{equation}\label{eq:limitsub0}
\bar{N}^{r_k}\rightarrow\rho,
	\end{equation} 
in the uniform topology.
\item {Moreover, given $(r_k)$, there exists a bounded, possibly stochastic process $w$ such that}
\begin{equation}\label{eq:rhoweak*w}
\mathbbm{1}_{\{\bar{N}^{r_k}_{t-}<1\}}\stackrel{*}{\rightharpoonup}w(t) \quad \text{in } L^\infty[0,T].
\end{equation}
\item {Furthermore, almost surely, $\rho$ and $w$ defined in \eqref{eq:limitsub0}-\eqref{eq:rhoweak*w} satisfy}  
	\begin{align}\label{eq:rho-w}
		&\rho_{t}=\int_{0}^{t} w(u) \bar{G}(t-u) \lambda(u) d u, \quad t \in [0,T],\quad\text{and}\\
&\mathbbm{1}_{\left\{\rho_{u-}<1\right\}} \leq w(u) \leq 1,\,\text{a.e. in }[0,T].\nonumber
	\end{align}
That is, {for almost all $\omega\in\Omega$} $(\rho{(\om)}, w{(\om)})$ as in \eqref{eq:rho-w} is a solution, interpreted according to Definition~\ref{def:sol_Vol}, to the following non-linear discontinuous Volterra integral equation
	\begin{equation}\label{eq:rhoVol}
		\rho_{t}=\int_{0}^{t} \mathbbm{1}_{\{\rho_{u-}< 1 \}} \bar{G}(t-u) \lambda(u) d u .
	\end{equation}

\end{enumerate}
\end{proposition}
\begin{proof}
For simplicity we will consider the initial subsequence to be $(n)$, but the arguments below go through for any initial subsequence.

\vspace{0.1in}
\noindent
\emph{Part (i).}
Applying the decomposition \eqref{eq:M-comp} in \eqref{eq:rho^n0} we have
\beq\label{eq:rhon}
\bar{N}_t^n = X_t^n + Y_t^n,
\eeq
where 
\beq\label{eq:X^n+ga^n}
X_t^n := \int_0^t\int_\R W_n(t,u,x) \mathcal{M}^n_* (du,dx), 
\text{ and }
Y_t^n := \int_{0}^{t} \mathbbm{1}_{\{\bar{N}^n_{u-}<1\}} \bar{G}(t-u) \lambda(u) du.
\eeq
We will analyze $X^n$ and $Y^n$ separately, starting with the term $Y^n$. 

\noindent
\emph{Convergence of $Y^n$.}
By the local integrability of \( \lambda \) from Assumption~\ref{asm:queue0}, we have from \eqref{eq:X^n+ga^n} that almost surely
	\begin{equation} \label{eq:Ybd}
\sup_n \sup_{t \in [0,T]} Y_t^n \leq \int_0^T \lambda(u) \, du < \infty.  
	\end{equation}  
 Meanwhile $Y^n$ satisfies
\begin{align}\label{eq:ga^n-diff}
Y_t^n - Y_s^n &= \int_0^t \mathbbm{1}_{\{\bar{N}_{u-}^n < 1\}} \bar{G}(t-u) \lambda(u) du - \int_0^s \mathbbm{1}_{\{\bar{N}_{u-}^n < 1\}} \bar{G}(s-u) \la(u) du\nonumber\\
&= \int_s^t \mathbbm{1}_{\{ \bar{N}_{u-}^n < 1 \}} \bar{G}(t-u) \la(u) du + \int_0^s \mathbbm{1}_{\{\bar{N}_{u-}^n < 1\}} \lp \bar{G}(t-u) - \bar{G}(s-u) \rp \la(u) du.
\end{align}
Given that $\bar{G}$ is non-increasing and bounded above by $1$, we can derive from \eqref{eq:ga^n-diff} that almost surely
\begin{equation*}
    \sup_n\left| Y_t^n - Y_s^n \right| \leq \int_s^t \la(u) du.
\end{equation*}
Since the function $\Lambda(t)=\int_0^t\la(u)du$ is uniformly continuous on $[0,T]$, it follows that $Y^{n}$ are {equicontinuous} almost surely. 
Therefore, by Arzel\`a-Ascoli \cite[Thm.\ 3.12]{bressan2012lecture}, we have that for almost all $\om \in \oom$, there exists a subsequence $l_k(\om)$ such that 
\begin{equation}\label{eq:Ytorho}
    Y^{l_k} {\rightarrow}\rho, \quad \text{uniformly in }[0,T],
\end{equation}
for some continuous function $\rho(\om)$ on $[0,T]$. Moreover, $L^1[0, T]$ is a separable Banach space with dual $L^{\infty}[0, T]$ and $\mathbbm{1}_{\{\bar{N}^n_{u-}<1\}}\in L^{\infty}[0,T]$. Therefore by applying \cite[Thm.\ 2.34]{bressan2012lecture} for each fixed $\om$, we have that for almost all $\om \in \oom$, there is a subsubsequence $(r_k(\om))\subset(l_k(\om))$ and $w(\om) \in L^\infty[0,T]$, such that for any $\phi \in L^1[0,T]$ 
\begin{equation}\label{eq:rho0weakstar}
\lim _{k \rightarrow \infty} \int_0^t \phi(u) \mathbbm{1}_{\{\bar{N}^{r_k}_{u-} < 1\}} d u=\int_0^t \phi(u) w(u) d u, \quad \text{for all } t\in[0,T].
\end{equation}
Note that both $w$ and $r_k$ may still depend on $\om$. In particular, choosing $\phi(\cdot) = \bar{G}(t - \cdot) \la(\cdot)$ we have for all $t\in[0,T]$, and almost all $\om \in \oom$ 
\begin{equation}\label{eq:idrho}
\lim_{k \to \infty} Y_t^{r_k} = \lim _{k \rightarrow \infty} \int_0^t \mathbbm{1}_{\{\bar{N}^{r_k}_{u-}< 1\}} \bar{G}(t-u) \lambda(u) du=\int_0^t w(u) \bar{G}(t-u) \lambda(u) d u.
\end{equation}
From \eqref{eq:idrho} we identify $\rho$ in \eqref{eq:Ytorho}, that is:
\begin{equation}\label{eq:rho=int}
    \rho_t=\int_0^t w(u) \bar{G}(t-u) \lambda(u) d u.
\end{equation}

\noindent
\emph{Convergence of $X^n$.}
Let us now analyze the term $X^n$. By \eqref{eq:rhon}-\eqref{eq:Ybd} we have that almost surely
	\begin{equation} \label{eq:Xbd}
	\sup_n \sup_{t \in [0,T]} \left| X_t^{n} \right| \leq\sup_n \sup_{t \in [0,T]}\bar{N}_t^n + \sup_n \sup_{t \in [0,T]} Y_t^n \leq1 + \int_0^T \lambda(u) \, du < \infty.  
	\end{equation}  
Furthermore, consider the process 
\beq\label{eq:Z_t_1}
Z_t^{n}=\int_0^t\int_\R \mathcal{M}_*^{n}(du,dx)=\int_0^t\int_\R \mathcal{M}^{n}(du,dx)-\int_0^t\int_\R \mathcal{M}_c^{n}(du,dx).
\eeq
Since for any fixed $n,\,\int_0^t\int_\R \mathcal{M}_c^{n}(du,dx)=\int_0^tn\lambda(u)du$ is a continuous function with bounded variation, by \cite[Chap.\ 2, Thm.\ 26]{protter2004stochastic} it has $0$ quadratic variation. It follows that the quadratic variation of $Z^{{n}}$ coincides with the quadratic variation of the pure jump process $\mathcal{M}^{n}([0, \cdot] \times \R)$, i.e.
	\begin{equation}\label{eq:Zquad_1}
[Z^{{n}},Z^{{n}}]_t={\sum_{i=1}^\infty}\left(\mathbbm{1}_{\{T_i\leq t\}}\right)^2=A^n(t).
	\end{equation}
$Z^n$ are martingales. However, the integrand $W_n(t,u,x)$ of \eqref{eq:W^n} depends on $t$ through $\mathbbm{1}_{\{x>t-u\}}$, and hence $t\mapsto X^n_t$ is not a martingale. Instead, we split $X^n$ into two pieces, each of which \emph{is} a martingale when $t$ is frozen, and recover the supremum from Lemma~\ref{lem:to-be-written}. Denote
\beq\label{eq:BCsplit}
\bar{B}^n_t:=\frac{1}{n}\sum_{i=1}^\infty\mathbbm{1}_{\{T_i\leq t\}}\mathbbm{1}_{\{\bar{N}^n_{T_i-}<1\}},\qquad
\bar{C}^n_t:=\frac{1}{n}\sum_{i=1}^\infty\mathbbm{1}_{\{T_i\leq t\}}\mathbbm{1}_{\{\bar{N}^n_{T_i-}<1\}}\mathbbm{1}_{\{V_i\leq t-T_i\}}
\eeq
to be the scaled numbers of admitted arrivals and of completed services in $[0,t]$, together with
\beq\label{eq:BCsplit-tilde}
\tilde{B}^n_t:=\int_0^t\mathbbm{1}_{\{\bar{N}^n_{u-}<1\}}\la(u)du,\qquad
\tilde{C}^n_t:=\int_0^t\mathbbm{1}_{\{\bar{N}^n_{u-}<1\}}G(t-u)\la(u)du .
\eeq
By \eqref{eq:rho^n1} and \eqref{eq:X^n+ga^n},
\beq\label{eq:XBC}
\bar{N}^n_t=\bar{B}^n_t-\bar{C}^n_t,\qquad Y^n_t=\tilde{B}^n_t-\tilde{C}^n_t,\qquad
X^n_t=\lp\bar{B}^n_t-\tilde{B}^n_t\rp-\lp\bar{C}^n_t-\tilde{C}^n_t\rp .
\eeq
It is readily checked that all four quantities in \eqref{eq:BCsplit}-\eqref{eq:BCsplit-tilde} are nondecreasing in $t$, and for every $n$ and every $\om$, the compensators satisfy
\beq\label{eq:tildemod}
\tilde{B}^n_t-\tilde{B}^n_s\leq\int_s^t\la(u)du,\qquad
\tilde{C}^n_t-\tilde{C}^n_s\leq\int_s^t\la(u)du+\om_G(t-s)\int_0^T\la(u)du,
\eeq
where $\om_G(\der):=\sup_{r\geq0}\lp G(r+\der)-G(r)\rp\downarrow0$ as $\der\downarrow0$ because $G$ is absolutely continuous with density $g\in L^1$.

Fix $t\in[0,T]$ and consider the two processes
\beq\label{eq:fix-t}
u\mapsto\frac{1}{n}\int_0^u\int_\R\mathbbm{1}_{\{\bar{N}^n_{v-}<1\}}\mathcal{M}^n_*(dv,dx),\qquad
u\mapsto\frac{1}{n}\int_0^u\int_\R\mathbbm{1}_{\{\bar{N}^n_{v-}<1\}}\mathbbm{1}_{\{x\leq t-v\}}\mathcal{M}^n_*(dv,dx),
\eeq
whose values at $u=t$ are $\bar{B}^n_t-\tilde{B}^n_t$ and $\bar{C}^n_t-\tilde{C}^n_t$ respectively. Since $t$ appears as a parameter, both integrands are predictable in $v$ and bounded by $1/n$ and hence both processes in \eqref{eq:fix-t} are square integrable martingales with jumps bounded by $1/n$. Consequently,  as in \eqref{eq:Zquad_1}, their quadratic variation is at most $n^{-2}A^n_u$. Writing $M^n$ for either one, the Burkholder--Davis--Gundy inequality implies
\beq\label{eq:BDG-M_n}
\E\lln M^n_t\rrn^4\leq C\,\E\lp[M^n,M^n]_T\rp^2\leq\frac{C}{n^4}\E\lp A^n_T\rp^2 ,
\eeq
where $C$ is a constant independent of $n$ and of $t$.
Since $A^n_T$ is Poisson with mean $n\int_0^T\la(u)du$, \eqref{eq:BDG-M_n} and \eqref{eq:XBC} provide
\begin{equation}\label{eq:sup-X-4th}
\E\lln X^n_t\rrn^4\leq8\,\E\lln\bar{B}^n_t-\tilde{B}^n_t\rrn^4+8\,\E\lln\bar{C}^n_t-\tilde{C}^n_t\rrn^4\leq\frac{C}{n^4}\left(n^2\lp\int_0^T\lambda(u)du\rp^2+n\int_0^T\lambda(u)du\right)\leq \frac{C^{'}}{n^2},
\end{equation}
where $C'>0$ is independent of $n$ and of $t$. For any $\ep>0$, \eqref{eq:sup-X-4th} and Markov's inequality imply $\mathbb{P}(|M^n_t|>\ep)\leq\frac{C'}{\ep^4n^2}$. Thus, $\sum_n\mathbb{P}(|M^n_t|>\ep)<\infty$ and, by the Borel--Cantelli lemma, 
$$
\bar{B}^n_t-\tilde{B}^n_t\rightarrow 0\quad\text{and}\quad\bar{C}^n_t-\tilde{C}^n_t\rightarrow0 \text{ almost surely},
$$
for each fixed $t$.
By \eqref{eq:XBC}, $X^n_t\rightarrow0$ almost surely as well.

It remains to show the uniform convergence. To that effect, we fix $m\in\N$ and using \eqref{eq:tildemod} choose a grid $0=s_0<\cdots<s_{k_m}=T$ fine enough that the increments of $\tilde{B}^n$ and of $\tilde{C}^n$ across consecutive grid points are at most $1/(2m)$, simultaneously for every $n$ and every $\om$. Applying Lemma~\ref{lem:to-be-written} to the pairs $(\bar{B}^n,\tilde{B}^n)$ and $(\bar{C}^n,\tilde{C}^n)$ and using \eqref{eq:XBC},
$$
\sup_{t\in[0,T]}\lln X^n_t\rrn\leq\max_{0\leq j\leq k_m}\lln\bar{B}^n_{s_j}-\tilde{B}^n_{s_j}\rrn+\max_{0\leq j\leq k_m}\lln\bar{C}^n_{s_j}-\tilde{C}^n_{s_j}\rrn+\frac{1}{m} .
$$
Each maximum is over finitely many points, at each of which the difference tends to $0$ almost surely by the previous paragraph. Excluding the union of the corresponding null sets over $j$ and over $m\in\N$ we obtain, for the entire sequence,
\begin{equation}\label{eq:Xto0U}
    \lim_{n\rightarrow\infty}\sup_{t\in[0,T]}\lln X^n_t\rrn=0 ,\quad \text{almost surely.}
\end{equation}
Combining the above arguments together, we thus obtain from \eqref{eq:rhon}, \eqref{eq:Ytorho} and \eqref{eq:Xto0U} that for almost all $\om\in\Omega$ and any subsequence, there exists a further subsequence $(r_k(\om))$ such that
\begin{equation}\label{eq:rho}
\bar{N}^{r_k}=X^{r_k}+Y^{r_k}\rightarrow\rho,
	\end{equation}
in the uniform topology, where $\rho$ is identified by \eqref{eq:rho=int}. This completes the proof of \emph{Part (i)}.

\vspace{0.1in}
\noindent
\emph{Part (ii).} This has already been shown above in \eqref{eq:rho0weakstar}.

\vspace{0.1in}
\noindent
\emph{Part (iii).}
The function $\rho$ has been identified by \eqref{eq:rho=int}. It remains to show constraint for the function $w$. We now observe that for sequence $(\bar{N}^n)$ the set $\left\{\bar{N}_{u-}^n<1\right\}$ is identical to the set $\left\{\bar{N}_{u-}^n \leq 1-\frac{1}{n}\right\}$. This is because $\bar{N}^n$ only takes values in $\left\{\frac{i}{n}: i=0, 1, \ldots, n\right\}$. Therefore, we can rewrite \eqref{eq:rhon} as
	\begin{equation*}
		\bar{N}^{n}_t=X_t^{n}+\int_{0}^{t} \mathbbm{1}_{\{\bar{N}^n_{u-}\leq 1-1/n\}} \bar{G}(t-u) \lambda(u) d u.
	\end{equation*}
Since $\bar{N}^n \leq 1$ for every $n$, by the uniform limit in \eqref{eq:rho}, $\rho\leq1$. Our next objective is to discover the function $w$ in \eqref{eq:rho-w}. Since $\rho$ is continuous, fix $\varepsilon>0$ and choose $N$ large enough such that for all $k>N$ we have $r_k>\frac{3}{\varepsilon}$, and $\left\|\bar{N}^{r_k}-\rho\right\|_T<\frac{\varepsilon}{3}$ {almost surely}. Then it is readily checked that
	\begin{equation*}
		\mathbbm{1}_{\{\rho_{u-}\leq 1-\varepsilon\}} \leq \mathbbm{1}_{\{\bar{N}^{r_k}_{u-}\leq 1-1/r_k\}} \leq \mathbbm{1}_{\{\rho_{u-}<1+\varepsilon\}}.
	\end{equation*} 
	Therefore for any $\phi \geq 0$ such that $\phi\in L^1[0, T]$ we have {almost surely}
	\begin{equation*}
		\int_0^t \phi(u) \mathbbm{1}_{\{\rho_{u-}\leq 1-\varepsilon\}} d u \leq \int_0^t \phi(u) \mathbbm{1}_{\{\bar{N}^{r_k}_{u-}\leq 1-1/r_k\}} d u \leq \int_0^t \phi(u) \mathbbm{1}_{\left\{\rho_{u-}<1+\varepsilon\right\}} d u.
	\end{equation*}
	Note that $\lim _{\varepsilon \downarrow 0} \mathbbm{1}_{\{\rho_{u-}<1-\varepsilon\}}=\mathbbm{1}_{\{\rho_{u-}<1\}}$ and $\lim _{\varepsilon \downarrow 0} \mathbbm{1}_{\{\rho_{u-}<1+\varepsilon\}}=\mathbbm{1}_{\{\rho_{u-}\leq 1\}}=1$. Consequently taking $k \rightarrow \infty$ and then $\varepsilon \downarrow 0$ we have by the dominated convergence theorem and \eqref{eq:rhoweak*w} that {almost surely}:
	\begin{equation*}
		\int_0^t \phi(u) \mathbbm{1}_{\left\{\rho_{u-}<1\right\}} d u \leq \int_0^t w(u) \phi(u) d u \leq \int_0^t \phi(u)  d u.
	\end{equation*}
	Since $\phi$ is arbitrary in $L^1[0,T]$, we have {almost surely}
	\begin{equation*}
		\mathbbm{1}_{\left\{\rho_{u-}<1\right\}} \leq w(u) \leq 1,\,\text{a.e. in }[0,T].
	\end{equation*}
Recall the notations defined in \eqref{eq:g-bar}. It is easily checked that $\underline{\mathbbm{1}}_{\{\rho_{u-}<1\}}=\mathbbm{1}_{\{\rho_{u-}<1\}}$ and $\bar{\mathbbm{1}}_{\{\rho_{u-}<1\}}=\mathbbm{1}_{\{\rho_{u-}\leq1\}}=1$. Therefore, by \eqref{eq:rho-w} and Definition~\ref{def:sol_Vol} we conclude that $(\rho,w)$ is a solution to the discontinuous Volterra equation \eqref{eq:rhoVol}.

\end{proof}

We have established a fluid limit for $\bar{N}_t^n$ {along a subsequence} when the system starts empty. Now, we extend our considerations to a more general case.

\begin{assumption}\label{asm:initial0}
Let the conditions under Assumption~\ref{asm:queue0} hold. In addition let the number of customers in the system at time $0$: $N_0^n$, satisfy
	
\begin{equation*}
		\lim _{n \rightarrow \infty}\dfrac{N_0^n}{n} = \rho_0, \quad \text{ almost surely, }
\end{equation*}
	where $\rho_0 \in{[}0,1]$. Moreover, assume that the remaining service times of each of the initially occupied servers are independently and identically distributed with the distribution $F^n$, independent of $N_0^n$,  satisfying
	\begin{equation*}
		\lim_{n \rightarrow \infty} \sup_{t}\left|F^n(t)-F(t)\right|=0,
	\end{equation*}
	for some limiting distribution $F$ with density $f$.
\end{assumption}

\begin{proposition}\label{thm:fluidlimitini0}
	Let Assumption~\ref{asm:initial0} hold. Then 
\begin{enumerate}[label = (\roman*), wide, labelwidth=!, labelindent=0pt]
\item For any $T > 0$, almost all $\om \in \oom$ and any subsequence of $\bar{N}^n$, there exists a further subsequence $\bar{N}^{r_k}$, possibly depending on $\om$ and a real-valued continuous, possibly stochastic process $\rho$ such that 
	\begin{equation}\label{eq:limitsub0rho0}
\bar{N}^{r_k}\rightarrow\rho,
	\end{equation} 
in the uniform topology.
\item {Moreover, given $(r_k)$, there exists a bounded, possibly stochastic process $w$ such that }
\begin{equation}\label{eq:rhoweak*wrho0}
\mathbbm{1}_{\{\bar{N}^{r_k}_{t-}<1\}}\stackrel{*}{\rightharpoonup}w(t) \quad \text{in } L^\infty[0,T].
\end{equation}
\item {Furthermore, almost surely, $\rho$ and $w$ defined in \eqref{eq:limitsub0rho0}-\eqref{eq:rhoweak*wrho0} satisfy}
	\begin{align}\label{eq:rho-wrho0}
		&\rho_{t}=\rho_0\bar{F}(t)+\int_{0}^{t} w(u) \bar{G}(t-u) \lambda(u) d u, \quad t \in [0,T],\quad\text{and}\\
&\mathbbm{1}_{\left\{\rho_{u-}<1\right\}} \leq w(u) \leq 1,\quad \text{a.e. in }[0,T].\nonumber
	\end{align}
That is, {for almost all $\om\in\Omega$} $(\rho{(\om)}, w{(\om)})$ as in \eqref{eq:rho-wrho0} is a solution, interpreted according to Definition~\ref{def:sol_Vol}, to the following non-linear discontinuous Volterra integral equation
	\begin{equation}\label{eq:rhoVolini}
		\rho_{t}=\rho_0\bar{F}(t)+\int_{0}^{t} \mathbbm{1}_{\{\rho_{u-}< 1 \}} \bar{G}(t-u) \lambda(u) d u .
	\end{equation}
\end{enumerate}
\end{proposition}
\begin{proof}
At time $0$, the number of customers in service is $N_0^n$. Let the remaining service times for the customers in service be $(V_i^{0,n})_{1 \leq i \leq N_0^n}$. 
We will first prove that
\begin{equation}\label{eq:iniD0}
	\lim_{n \rightarrow \infty} \sup_t \left|\frac{1}{n}\sum_{i=1}^{N_0^n}\mathbbm{1}_{\{V_i^{0,n} > t\}}- \rho_0 \bar{F}(t)\right|=0, \quad \text{almost surely}.
\end{equation}
When $\rho_0=0$, this is immediate, since $\sup_t \frac{1}{n} \sum_{i \leq N_0^n} \mathbbm{1}_{\{V_i^{0,n} > t\}} \leq \frac{N_0^n}{n} \to 0$ almost surely, which is \eqref{eq:iniD0} for $\rho_0=0$. Hence we consider the case when $\rho_0>0$. Then, similar to \eqref{eq:rho^n0} we have:
\begin{equation}\label{eq:rhon-init}
	\bar{N}_t^n=\frac{1}{n}\sum_{i=1}^{N_0^n}\mathbbm{1}_{\{V_i^{0,n}>t\}}+\int_0^t\int_\R W_n(t,u,x) \mathcal{M}^n (du,dx).
\end{equation}	
Observe that 
\begin{equation}\label{eq:initial-decomp}
	\frac{1}{n}\sum_{i=1}^{N_0^n}\mathbbm{1}_{\{V_i^{0,n}>t\}}=\frac{N_0^n}{n} \frac{1}{N_0^n}\sum_{i=1}^{N_0^n}{\mathbbm{1}_{\{V_i^{0,n}>t\}}}.
\end{equation}
Denote
$$E^n=\sup_t \left|\frac{1}{N_0^n}\sum_{i=1}^{N_0^n}\mathbbm{1}_{\{V_i^{0,n}>t\}}- \bar{F}^n(t)\right|,\, \text{ and } H^n=\left\{N_0^n>\frac{\rho_0 n}{2}\right\}.$$
When $N_0^n=0$ define $E^n=0$. By Assumption~\ref{asm:initial0}, thanks to Dvoretzky–-Kiefer-–Wolfowitz inequality \cite{massart1990tight}, for any $\ep>0$
\begin{equation}\label{eq:DKW}
	\mathbb{P}\left(E^n>\ep \,\middle|\, N_0^n\right)\leq 2e^{-2N_0^n \ep^2}.
\end{equation}
Therefore, using \eqref{eq:DKW}
\begin{equation*}
\mathbb{P}\left(\{E^n>\ep\}\cap H^n\right)=\E\left[\mathbbm{1}_{H^n} \mathbb{P}\left(E^n>\ep \,\middle|\, N_0^n\right)\right]\leq2\E\left[\mathbbm{1}_{H^n}  e^{-2N_0^n \ep^2}\right]\leq 2e^{-\rho_0n \ep^2}.
\end{equation*}
Consequently, 
$$\sum_{n=1}^{\infty} \mathbb{P}\left(\left\{E^n>\varepsilon\right\} \cap H^n\right)<\infty .$$
By Borel--Cantelli Lemma, $\{E^n>\varepsilon\} \cap H^n$ occurs only finitely often almost surely. Since $N_0^n/n\rightarrow \rho_0$ almost surely, $H^n$ occurs eventually almost surely. Therefore, $\{E^n>\varepsilon\}$ occurs finitely often almost surely. Thus, we can conclude that almost surely
\begin{equation}\label{eq:emp-Fn}
    \lim_{n \to \infty}E^n=0.
\end{equation}
Therefore, by triangle inequality, Assumption~\ref{asm:initial0} and \eqref{eq:emp-Fn} we have
almost surely that
\begin{equation*}
	\lim_{n \rightarrow \infty} \sup_t \left|\frac{1}{N_0^n}\sum_{i=1}^{N_0^n}\mathbbm{1}_{\{V_i^{0,n}>t\}}- \bar{F}(t)\right|\leq\lim_{n \rightarrow \infty} E^n+\lim_{n \rightarrow \infty} \sup_t|\bar{F}^n(t)-\bar{F}(t)|=0.
\end{equation*}
Finally, from the decomposition \eqref{eq:initial-decomp} we have \eqref{eq:iniD0}. Since we have already analyzed the second term in \eqref{eq:rhon-init} involving integration with respect to $\mathcal{M}^n$ in Proposition~\ref{prop:weaklimits0}, we obtain our desired result from \eqref{eq:iniD0}.
\end{proof}

Now, we establish the existence of a unique $\rho$ that satisfies \eqref{eq:rhoVolini} in the sense of Definition~\ref{def:sol_Vol}. Consequently, we obtain a unique fluid limit of the fraction of occupied servers $\bar{N}_t^n$.

\begin{theorem}\label{thm:solVol}
Let Assumption~\ref{asm:initial0} hold. Then \eqref{eq:rhoVolini} is uniquely solvable, in the following sense: there exists a unique function $\rho$ such that for all $t\in [0,T]$
\begin{equation}\label{eq:existlemrhou}
\rho_{t}=\rho_0\bar{F}(t)+\int_{0}^{t} z(u) \bar{G}(t-u) \lambda(u) du, \quad \text{such that } \,0\leq\rho_t\leq 1, 
\end{equation}
for some $z(t)$ that satisfies
\begin{equation}\label{eq:wtbdu}
\mathbbm{1}_{\{\rho_{t}<1\}} \leq z(t) \leq 1 \quad \,\text{a.e. in }[0,T].
\end{equation}
In addition, the function $z\la$ is unique almost everywhere.
\end{theorem}
\begin{proof}
The existence of the solution follows directly from Proposition~\ref{thm:fluidlimitini0}. It remains to show the uniqueness. Suppose there are two solutions $(\rho^1_t,z_1(t))$ and $(\rho^2_t,z_2(t))$ satisfying \eqref{eq:existlemrhou}. Let us denote 
\begin{equation*}
    \Delta(t):=\rho^1_t-\rho^2_t, \quad k(t):=(z_1(t)-z_2(t))\la(t).
\end{equation*}
If $\Delta(t)>0$, then
$$
\rho^2_t<\rho^1_t \leq 1,
$$
that is,  $\rho^2_t<1$. From \eqref{eq:wtbdu} this forces $z_2(t)=1$ for almost every such $t \in [0,T]$. Since $z_1(t) \leq 1$, we have
$$
k(t)=(z_1(t)-z_2(t))\la(t) \leq 0 .
$$
Similarly, $\Delta(t)<0$ implies $k(t) \geq 0$. Therefore,
\begin{equation}\label{eq:Del*a}
\Delta(t) k(t) \leq 0 \quad \text {a.e.}
\end{equation}
Since $\lambda$ is integrable and $z_1,z_2$ are bounded, $k$ is also Lebesgue integrable. Denote $K(t):=\int_0^t k(u)du$. 
As an indefinite integral of an $L^1$ function, $K$ is absolutely continuous, and by the Lebesgue differentiation theorem (see \cite[Chap.\ 3, Thm.\ 3.11]{Stein2005})  $K^{\prime}(t)=k(t)$ almost everywhere. Since $\bar{G}(0)=1$, from \eqref{eq:existlemrhou} and integration by parts we have
\begin{equation}\label{eq:Del}
    \Delta(t)=\int_0^t k(u)\bar{G}(t-u)du=K(t)-\int_0^t K(u)g(t-u)du.
\end{equation}
Plugging \eqref{eq:Del} into \eqref{eq:Del*a} we get
\begin{equation}\label{eq:A-A'}
    \left[K(t)-\int_0^t K(u)g(t-u)du\right]K^{\prime}(t)\leq 0 \quad \text {a.e.}
\end{equation}
Since $g$ is integrable on $[0,T]$, we can choose $\delta>0$ such that
$$
c_\delta:=\int_0^\delta g(s)ds<1.
$$
Denote $M:=\sup_{t\in[0,\delta]}|K(t)|$. Then for $t\in[0,\delta]$
\begin{equation}\label{eq:c-del-M}
\left|\int_0^t K(u)g(t-u)du\right|\leq M\int_0^\delta g(s)ds=c_\delta M.
\end{equation}
We claim that $|K(t)|\leq c_\delta M$ for $t\in[0,\delta]$. Indeed, if $K(t)>c_\delta M$, by \eqref{eq:c-del-M} we have 
\begin{equation}\label{eq:A-int}
    K(t)-\int_0^t K(u)g(t-u)du>0,
\end{equation}
and thus by \eqref{eq:A-A'} and \eqref{eq:A-int} we can conclude that $K^{\prime}(t)\leq0$ for a.e. $t\in\{K(t)>c_\delta M\}\cap\,[0,\delta]$. Since $\phi(x)=(x-c_\delta M)^+$ is Lipschitz continuous and $K$ is absolutely continuous, $\phi\circ K$ is also absolutely continuous and hence differentiable almost everywhere. Consequently,
\begin{equation}\label{eq:A>cM}
    \frac{d}{dt} (K(t)-c_\delta M)^+\leq0, \quad \text{for a.e. } t\in\{K(t)>c_\delta M\}\cap\,[0,\delta].
\end{equation}
Since $K(t)$ is continuous, $\{K(t)<c_\delta M\}$ is an open set. $(K(t)-c_\delta M)^+=0$ is a constant for all interior points of $\{K(t)<c_\delta M\}$, we then have
\begin{equation*}
    \frac{d}{dt} (K(t)-c_\delta M)^+=0, \quad \text{for a.e. } t\in\{K(t)<c_\delta M\}\cap\,[0,\delta].
\end{equation*}
For $t\in E:=\{K(t)=c_\delta M\}\cap\,[0,\delta]$, since $E$ is Lebesgue measurable, by Lebesgue's density theorem (see for example \cite[Chap.\ 3, Cor.\ 1.5]{Stein2005}), almost every point in $E$ is a density point. Consequently, we can choose a sequence $(t_n)$ such that $t_n\in E\backslash \{t\}$ and $t_n\to t$. Hence
\begin{equation}\label{eq:A=cM}
    K^{\prime}(t)=\lim_{n\to \infty}\frac{K(t_n)-K(t)}{t_n-t}=0 \quad \text{for a.e. } t\in E.
\end{equation}
Moreover,
\begin{equation}\label{eq:K=cM+}
   \frac{d}{dt} (K(t)-c_\delta M)^+= \lim_{n\to \infty}\frac{(K(t_n)-c_\delta M)^+-(K(t)-c_\delta M)^+}{t_n-t}=0\quad \text{for a.e. } t\in E.
\end{equation}
Combining \eqref{eq:A>cM}-\eqref{eq:K=cM+}, we have
$$
\frac{d}{dt} (K(t)-c_\delta M)^+\leq0, \quad \text{for a.e. } t\in[0,\delta].
$$
Since $K(0)=0 \leq c_\delta M$ and hence $(K(0)-c_{\der}M)^{+} = 0$, it follows that 
$$(K(t)-c_\delta M)^+=0+\int_0^t\frac{d}{dt}(K(u)-c_\delta M)^+du\leq0 \quad \text{for all } t\in[0,\delta].$$ 
Consequently, $K(t)\leq c_\delta M$. If $K(t)<-c_\delta M$ then $K(t)-\int_0^t K(u)g(t-u)du<0$ and $K^{\prime}(t)\geq0$ for $t\in\{K(t)<-c_\delta M\}\cap\,[0,\delta]$. A similar argument shows that $K(t)\geq -c_\delta M$. Therefore, for $t\in[0,\delta]$,
\begin{equation}\label{eq:K-leq-cM}
    |K(t)|\leq c_\delta M,
\end{equation}
and thus $M\leq c_\delta M$. Since $c_\delta<1$ we can conclude that $M=0$, and hence $K=0$ on $[0,\delta]$. 

For $t\in[\delta,2\delta]$ we denote $M_1:=\sup_{t\in[\delta,2\delta]}|K(t)|$. Since $K(t)=0$ on $[0,\delta]$, we have
\begin{equation*}
\left|\int_0^t K(u)g(t-u)du\right|=\left|\int_\delta^t K(u)g(t-u)du\right|\leq M_1\int_\delta^t g(t-u)du\leq M_1\int_0^\delta g(s)ds =c_\delta M_1.
\end{equation*}
Similar to the argument on $[0,\delta]$, we can conclude that $M_1=0$, and hence $K=0$ on $[\delta,2\delta]$. Iterating this argument we can conclude that $K=0$ on $[0,T]$. By \eqref{eq:Del} we then have $\Delta\equiv 0$ on $[0,T]$, which shows the uniqueness of $\rho$. Since $K\equiv0$, we also have $K^{\prime}=0$ almost everywhere. Recalling $K^{\prime}=k=(z_1-z_2)\la$ almost everywhere, we conclude $z_1\la=z_2\la$ almost everywhere.
\end{proof}
In Theorem~\ref{thm:solVol} we established the unique solvability of \eqref{eq:existlemrhou}-\eqref{eq:wtbdu}. Therefore, by Proposition~\ref{thm:fluidlimitini0} the fraction of occupied servers $\bar{N}^n_t$ converge to this unique $\rho_t$. By \eqref{eq:wtbdu} $z(t)=1$ when $\rho_t<1$. However, $z(t)$ remains unspecified when $\rho_t=1$. It would be beneficial to specify the value of $z(t)$ in this regime, specifically for the purpose of numerical experimentations. The following theorem provides a solution of $z$.

\begin{theorem}\label{thm:solVolw}
Under the setting of Theorem~\ref{thm:solVol}, the solution pair $(\rho,z)$ to \eqref{eq:existlemrhou}-\eqref{eq:wtbdu} must satisfy
\begin{align}
	&\rho_{t}=\rho_0\bar{F}(t)+\int_{0}^{t} z(u)\lambda(u) \bar{G}(t-u)  d u,\nonumber \\
	&z(t)\la(t)=\left\{\begin{array}{ll}
		\la(t), & \rho_t<1,\\
		\rho_0f(t)+\int_0^t z(u)\lambda(u)g(t-u)du, \, &\rho_t=1,
	\end{array}
	\right. \quad\text{a.e. in }[0,T].\label{eq:wt}
\end{align}
and
\begin{equation}\label{eq:ztbdu}
\mathbbm{1}_{\{\rho_{t}<1\}} \leq z(t) \leq 1 \quad \,\text{a.e. in }[0,T].
\end{equation}
\end{theorem}
\begin{proof}
By Theorem~\ref{thm:solVol} there exist a unique pair $(\rho,z\lambda)$ that solves \eqref{eq:existlemrhou}-\eqref{eq:wtbdu}. We will find out $z\lambda$. Since $z(t)$ is bounded and $\lambda(t),g(t) \in L^1[0,T]$, by Young's convolution inequality we have 
$$\frac{\partial}{\partial t} z(u)\bar{G}(t-u)\lambda(u)=-z(u)\lambda(u)g(t-u)\in L^1([0,T]\times[0,T]).$$
Therefore, by \cite[Thm.\ 2.7]{vsremr2012differentiation} we have for $t\in[0,T]$ 
	\begin{equation}\label{eq:leibniz}
		\rho_t^{\prime}=-\rho_0f(t)+z(t)\lambda(t)-\int_0^t z(u)g(t-u)\lambda(u)du, \quad \text{a.e. in }[0,T].
	\end{equation}
Since $\rho$ is differentiable almost everywhere, by a similar argument leading to \eqref{eq:A=cM} we can conclude that
\begin{equation}\label{eq:rho=1}
    \rho^{\prime}_t=0 \quad \text{for a.e. } t\in\{\rho_t=1\}.
\end{equation}
Therefore, by \eqref{eq:leibniz}-\eqref{eq:rho=1} we have
\begin{equation}\label{eq:z-la}
    z(t)\lambda(t)=\rho_0f(t)+\int_0^t z(u)\lambda(u)g(t-u)du \quad \text{for a.e. } t\in\{\rho_t=1\}.
\end{equation}
On the other hand, \eqref{eq:ztbdu} implies that
$$
z(t)\lambda(t)=\lambda(t) \quad \text{for a.e. } t\in\{\rho_t<1\}.
$$
This completes the proof.
\end{proof}

\begin{remark}\label{rk:exit-time}
The quantity
$$
d(t):=\rho_0 f(t)+\int_0^t z(u)\lambda(u)g(t-u)\,du
$$
represents (heuristically) the departure rate at time $t$. The term $\rho_0 f(t)$ accounts for departures from the customers initially in service, while the integral accounts for departures from customers admitted after time $0$. Moreover, $z(t)\lambda(t)$ represents the service acceptance rate. Thus \eqref{eq:leibniz} becomes
$$
\rho_t'=z(t)\lambda(t)-d(t),\qquad \text{for a.e. }t\in[0,T],
$$
and can be viewed as a flow-balance relation: the rate of change of the fluid in service equals the admission rate minus the departure rate.

In this perspective, the exiting time of the set $\{\rho_t=1\}$ has a natural rate interpretation.  When $\rho_t<1$, all arrivals are admitted, so that
$$
z(t)=1,\quad \rho'_t=\lambda(t)-d(t).
$$
When $\rho_t=1$, the system can remain at capacity only if departures can be replaced by arrivals. Thus,
$$
z(t)\la(t)=d(t)\leq\lambda(t).
$$
Consequently, the system leaves capacity when the arrivals are insufficient to replace the departures. More precisely, call $t \in \{\rho_t=1\}$ a right exit point if and only if, in every right neighborhood of $t$, $d(s)>\lambda(s)$ on
a set of positive Lebesgue measure.
\end{remark}
\begin{remark}
Suppose that, in addition to Assumption~\ref{asm:initial0}  
\begin{equation}\label{eq:f,g>0}
    f(t)>0\ \text{ for a.e. }t \text{ with } \bar{F}(t)>0,\quad \text{ and } \quad g(t)>0  \text{ for a.e. }t \text{ with }\bar{G}(t)>0.
\end{equation}
Then $\la(t)>0$ a.e. on $\{\rho_t=1\}$. 

Indeed, for $t\in\{\rho_t=1\}$ we have 
$$
1=\rho_{t}=\rho_0\bar{F}(t)+\int_{0}^{t} z(u)\lambda(u) \bar{G}(t-u)  d u.
$$
Since $\rho_0,z,\lambda$ are non-negative, at least one of the two terms on the right-hand side of the above equation is positive. Consequently, by \eqref{eq:f,g>0} we have at least one of the two terms 
$$
\rho_0f(t)\ \text{ and }\ \int_{0}^{t} z(u)\lambda(u) g(t-u) d u
$$
is positive. Therefore by \eqref{eq:z-la} we have
$$
\lambda(t)\geq z(t)\lambda(t)=\rho_0f(t)+\int_{0}^{t} z(u)\lambda(u) g(t-u) d u>0. \quad \text{for a.e. } t\in\{\rho_t=1\}.
$$
\end{remark}

Since we have obtained the unique solvability of $\rho$, we can establish the fluid limit result for the entire sequence $\bar{N}^{n}$.
\begin{theorem}\label{thm:fluidlimitini0uni}
	Let Assumption~\ref{asm:initial0} hold. Then 
\begin{enumerate}[label = (\roman*), wide, labelwidth=!, labelindent=0pt]
\item For any $T > 0$, there exist a real-valued continuous deterministic process $\rho$ such that almost surely,
	\begin{equation}\label{eq:limitsub0rho0n}
		\lim_{n\rightarrow\infty}\sup_{t\in[0,T]}|\bar{N}^{n}_t-\rho_t|=0. 
	\end{equation} 
\item {Moreover, there exists a bounded deterministic function $w$ 
such that almost surely}
\begin{equation}\label{eq:rhoweak*wrho0n}
\mathbbm{1}_{\{\bar{N}^{n}_{t-}<1\}} \la(t)\rightharpoonup w(t) \la(t) \quad \text{in } L^1[0,T], 
\end{equation}
where $w$ solves \eqref{eq:wt}-\eqref{eq:ztbdu}.
\item {Furthermore, $\rho$ and $w$ defined in \eqref{eq:limitsub0rho0n}-\eqref{eq:rhoweak*wrho0n} satisfy}
	\begin{align}\label{eq:rho-wrho0n}
		&\rho_{t}=\rho_0\bar{F}(t)+\int_{0}^{t} w(u) \bar{G}(t-u) \lambda(u) d u, \quad t \in [0,T],\quad\text{and}\\
&\mathbbm{1}_{\left\{\rho_{u-}<1\right\}} \leq w(u) \leq 1,\quad\text{a.e. in }[0,T].\nonumber
	\end{align}
That is, $(\rho, w)$ as in \eqref{eq:rho-wrho0n} is a solution, interpreted according to Definition~\ref{def:sol_Vol}, to the following non-linear discontinuous Volterra integral equation
	\begin{equation}\label{eq:rhoVolinin}
		\rho_{t}=\rho_0\bar{F}(t)+\int_{0}^{t} \mathbbm{1}_{\{\rho_{u-}< 1 \}} \bar{G}(t-u) \lambda(u) d u .
	\end{equation}

\end{enumerate}
\end{theorem}
\begin{proof}
\emph{Part (i).}
Define the sequence  $a_n(\om) := \sup_{t \in [0,T]} |\bar{N}_t^n(\om) - \rho_t|$. By Theorem~\ref{thm:solVol}, $\rho$ is unique. Consequently $\rho$ is a deterministic function, which is the same for all $\om$. From Proposition~\ref{thm:fluidlimitini0}, for almost all $\om$ and every subsequence of $a_n(\om)$ there exists a further subsequence $(r_k)$ possibly depending on $\om$, such that $a_{r_k}(\om)$ converges to $0$.
Hence $a_n \to 0$ for almost all $\om$. This proves \eqref{eq:limitsub0rho0n}.

\vspace{0.1in}
\noindent
\emph{Part (ii).} By Proposition~\ref{thm:fluidlimitini0} we have for almost all $\om$ and every subsequence there exists a subsubsequence $(r_k)$ and a bounded, possibly stochastic process $w$ such that $\mathbbm{1}_{\{\bar{N}^{r_k}_{u-}<1\}}\stackrel{*}{\rightharpoonup}w$ in $L^\infty[0,T]$. Testing this against $\phi\la\in L^1[0,T]$ for $\phi\in L^\infty[0,T]$ gives
\begin{equation}\label{eq:indlaweak*}
\mathbbm{1}_{\{\bar{N}^{r_k}_{u-}<1\}}\la(u)\rightharpoonup w(u)\la(u) \quad \text{in } L^1[0,T].
\end{equation}
By Theorem~\ref{thm:solVol} we know that this $w\la$ is unique and hence deterministic. 
For each fixed $\phi\in L^\infty[0,T]$, every subsequence of the real sequence $\int_0^T\phi\,\mathbbm{1}_{\{\bar{N}^n_{u-}<1\}}\la\,du$ has a further subsequence converging to $\int_0^T\phi w\la$, and hence the sequence converges. Therefore the weak convergence in \eqref{eq:indlaweak*} holds for the entire sequence.

\vspace{0.1in}
\noindent
{\emph{Part (iii).}
This follows directly from Proposition~\ref{thm:fluidlimitini0}.(iii) and Theorems~\ref{thm:solVol}-\ref{thm:solVolw}.
}
\end{proof}

Note that the probability that an incoming arrival at time $t$ will be accepted to the system is given by $\mathbb{P}(\rho_{t-}^n < 1)$. The following corollary provides asymptotics for this acceptance probability.

\begin{corollary}\label{cor:acceptprob}
The acceptance probability in the $n$-th $M_t/G/n/n$ model $\mathbb{P}(\bar{N}^n_{t-}<1)$ satisfies
\begin{equation*}
\mathbb{P}\left(\bar{N}^{n}_{t-}<1\right) \la(t)\rightharpoonup  w(t)\la(t) ,\quad 
\text{ in } L^{1}[0,T]
\end{equation*}
where $w$ is defined in Theorem~\ref{thm:fluidlimitini0uni}. Moreover, 
$$
\mathbb{P}\bigl(\bar N^n_{t-}<1\bigr) \to 1
$$
locally uniformly on $\{t\in[0,T]:\rho_t<1\}$.

\end{corollary}
\begin{proof}
The first part of the statement follows directly from \eqref{eq:rhoweak*wrho0n}, dominated convergence and Fubini's theorem:
for $\phi\in L^\infty[0,T]$, $\int_0^T\phi\,\mathbb{P}(\bar{N}^n_{t-}<1)\la\,dt=\E\int_0^T\phi\,\mathbbm{1}_{\{\bar{N}^n_{t-}<1\}}\la\,dt$ since $|\phi|\la\in L^1[0,T]$, and the integrand on the right is bounded by $\|\phi\|_\infty\int_0^T\la$. 
For the second part, since $\rho$ is continuous, for every compact $K\subset\{\rho_t<1\}$ there exists 
$$
\varepsilon_K:=\min_{t\in K}(1-\rho_t)>0.
$$
By \eqref{eq:limitsub0rho0n} we have
$$
\lim_{n\to \infty} \mathbb{P}(\sup_{s\in[0,T]}|\bar{N}^n_s-\rho_s|\geq\varepsilon_K)=0.
$$
Observe that 
$$\left\{\bar{N}_{t-}^n \geq 1\right\} \subseteq\left\{\sup _{s \in[0, T]}\left|\bar{N}_s^n-\rho_s\right| \geq 1-\rho_t\right\}.$$ 
Therefore,
$$
\lim_{n\to\infty}\sup_{t\in K}\left|\mathbb{P}(\bar N^n_{t-}<1)-1\right|=\lim_{n\to\infty}\sup_{t\in K}\mathbb{P}(\bar N^n_{t-}\geq1)\leq \lim_{n\to\infty}\mathbb{P}(\sup_{s\in [0,T]}|\bar{N}^n_s-\rho_s|\geq\varepsilon_K)=0,
$$
which is the desired result.
\end{proof}
\begin{corollary}\label{cor:accum-block}
Let Assumption~\ref{asm:initial0} hold and $w$ be as in Theorem~\ref{thm:fluidlimitini0uni}.
Let 
$$
\bar{L}_t^n := \frac{1}{n} \sum_{i=1}^{\infty} \mathbbm{1}_{\{T_i \leq t\}} \mathbbm{1}_{\{\bar{N}_{T_i -}^n = 1\}}.
$$
Then we have
$$
\lim_{n \to \infty} \sup_{t \in [0,T]} \left| \bar{L}_t^n - \int_0^t (1-(w(u)))\la(u) du \right| = 0, \text{almost surely.}
$$
\end{corollary}

\begin{proof}
Denote
\begin{align}\label{eq:bar-L-*}
\bar{L}_t^{n,*}:&=\int_0^t\int_\R \frac{1}{n}\mathbbm{1}_{\{\bar{N}_{u -}^n = 1\}}\mathcal{M}^{n}(du,dx)-\int_0^t\int_\R \frac{1}{n}\mathbbm{1}_{\{\bar{N}_{u -}^n = 1\}}\mathcal{M}_c^{n}(du,dx)\nonumber \\
&=\int_0^t\int_\R \frac{1}{n}\mathbbm{1}_{\{\bar{N}_{u -}^n = 1\}}\mathcal{M}_*^{n}(du,dx)=\bar{L}_t^n -\int_0^t\mathbbm{1}_{\{\bar{N}_{u -}^n = 1\}}\lambda(u)du.
\end{align}
Since the integrand here does not depend on $t$, $\bar{L}^{n, *}$ is itself a martingale. Thus,
similar to the argument of \eqref{eq:Z_t_1}-\eqref{eq:Xto0U}, by Burkholder--Davis--Gundy inequality there exist constants $C,C'>0$ such that
\begin{equation}
    \E\left(\sup_{t\in[0,T]}|\bar{L}_t^{n,*}|\right)^4\leq \frac{C}{n^4}\E([\bar{L}^{n,*}]_T^2)\leq \frac{C'}{n^2}.
\end{equation}
and thus by Markov's inequality and Borel--Cantelli Lemma we obtain
\begin{equation}\label{eq:L*-a.s.}
    \bar{L}_t^{n,*}\rightarrow0,\quad a.s.,
\end{equation}
in the uniform topology. 
From \eqref{eq:bar-L-*}, it remains to identify the limit of $\int_0^t \1_{\{\bar{N}_{u-}^n=1\}} \la(u) du$. Since $\bar{N}^n \leq 1$ we have $\1_{\{\bar{N}_{u-}^n=1\}} = 1-\1_{\{\bar{N}_{u-}^n<1\}}$, and hence for each fixed $t$ Theorem~\ref{thm:fluidlimitini0uni}.(ii) applied to the test function $\phi = \1_{[0,t]} \in L^\infty[0,T]$ gives
\beq\label{eq:cor-3-2-a}
\lim_{n \to \infty} \int_0^t \1_{\{\bar{N}_{u-}^n = 1\}} \la(u) du = \int_0^t (1-w(u)) \la(u) du, \quad \text{ almost surely}. 
\eeq
Notice that both $t \mapsto \int_0^t \1_{\{\bar{N}_{u-}^n=1\}}\la(u) du$ and $t \mapsto \int_0^t (1-w(u))\la(u) du$ are non-decreasing, vanish at $0$, and the second has increments from $s$ to $t$ bounded by the deterministic $\int_s^t \la(u) du$. Applying Lemma~\ref{lem:to-be-written} on a partition where the increments of $\int_0^t (1-w(u))\la(u) du$ are at most $1/m$ we have using the convergence \eqref{eq:cor-3-2-a} at the finitely many grid points and then taking limit over $m$ that 
\beq\label{eq:cor-3-2-b}
\lim_{n \to \infty}\sup_{t\in[0,T]}\left|\int_0^t(\mathbbm{1}_{\{\bar{N}_{u -}^n = 1\}}-(1-w(u))) \lambda(u)du \right| =0, \quad \text{almost surely.}
\eeq
By triangle inequality, \eqref{eq:L*-a.s.} and \eqref{eq:cor-3-2-b} we have
\begin{align*}
&\lim_{n \to \infty} \sup_{t \in [0,T]} \left| \bar{L}_t^n - \int_0^t (1-w(u))\la(u) du \right|\\
&\qquad\leq \lim_{n \to \infty}\sup_{t\in[0,T]}|\bar{L}_t^{n,*}|+\lim_{n \to \infty}\sup_{t\in[0,T]} \left|\int_0^t(\mathbbm{1}_{\{\bar{N}_{u -}^n = 1\}}-(1-w(u))) \lambda(u)du \right|=0,\, \text{almost surely.}
\end{align*}
This completes the proof.
\end{proof}

\begin{remark}\label{rk:wdiscontinuous}
As noticed in \eqref{eq:wt}, the function $w(t)$ can be discontinuous even when $\lambda$ is continuous. This means the limit of the blocking/acceptance probability is discontinuous. This property is further reflected in the numerics below.
\end{remark}

\section{Fluid Limit for Loss System with Buffer}\label{sec:finite-buffer}

\subsection{Setup.}
In this section, we introduce a time-varying many-server loss queueing model with buffer. We work with a sequence of queueing systems indexed by $n$, subject to the following assumptions.

\begin{assumption}\label{asm:queue}
{We} consider a $M_t / G / n / n + b_n$ loss queueing system; namely, a queueing model with
\begin{enumerate}[label=\roman*.]
	\item a nonhomogeneous Poisson arrival process $A^n$ with rate or intensity function $n \lambda(\cdot)$, where $\lambda$ is locally integrable;
	\item general customer service times sampled {independently} from a distribution $G$ with density $g$; 
	\item the system has $n$ servers and $b_n$ buffer spaces or waiting spaces. When the system is full, new incoming customer arrivals are lost.  Additionally, $b_n \geq 1$ and $b_n$ satisfies
\begin{equation}\label{eq:anbn}
\lim_{n\rightarrow\infty}\frac{b_n}{n}\rightarrow\beta \in (0, \infty).
\end{equation} 
\end{enumerate}

\end{assumption}

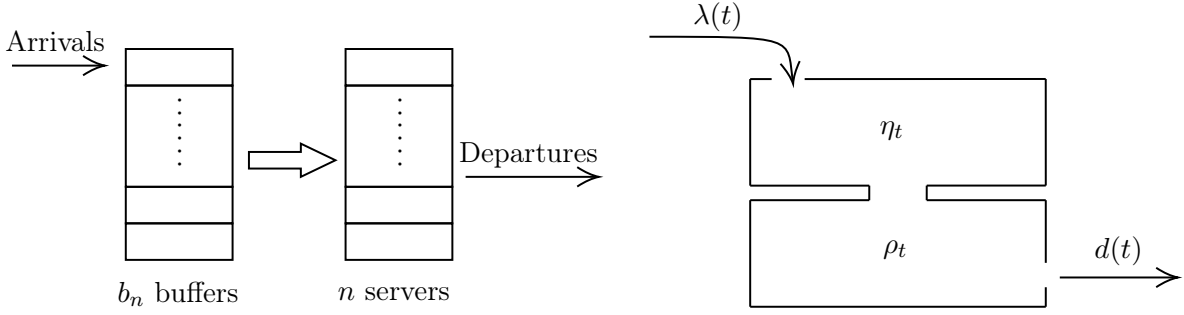
\begin{figure}[h]
\centering
\tikzset{every picture/.style={line width=0.75pt}} 
{%
\begin{tikzpicture}[x=0.75pt,y=0.75pt,yscale=-1,xscale=1]

\draw   (74.39,59.9) -- (128.25,59.9) -- (128.25,78.34) -- (74.39,78.34) -- cycle ;
\draw   (74.39,78.34) -- (128.25,78.34) -- (128.25,129.46) -- (74.39,129.46) -- cycle ;
\draw   (74.39,147.9) -- (128.25,147.9) -- (128.25,166.34) -- (74.39,166.34) -- cycle ;
\draw   (74.39,129.46) -- (128.25,129.46) -- (128.25,147.9) -- (74.39,147.9) -- cycle ;
\draw   (185.24,59.9) -- (239.1,59.9) -- (239.1,78.34) -- (185.24,78.34) -- cycle ;
\draw   (185.24,78.34) -- (239.1,78.34) -- (239.1,129.46) -- (185.24,129.46) -- cycle ;
\draw   (185.24,147.9) -- (239.1,147.9) -- (239.1,166.34) -- (185.24,166.34) -- cycle ;
\draw   (185.24,129.46) -- (239.1,129.46) -- (239.1,147.9) -- (185.24,147.9) -- cycle ;
\draw   (136.5,111.86) -- (162.68,111.86) -- (162.68,107.67) -- (180.13,116.05) -- (162.68,124.43) -- (162.68,120.24) -- (136.5,120.24) -- cycle ;
\draw    (17,68.28) -- (61.78,68.28) ;
\draw [shift={(63.78,68.28)}, rotate = 180] [color={rgb, 255:red, 0; green, 0; blue, 0 }  ][line width=0.75]    (10.93,-4.9) .. controls (6.95,-2.3) and (3.31,-0.67) .. (0,0) .. controls (3.31,0.67) and (6.95,2.3) .. (10.93,4.9)   ;
\draw    (245.78,124.43) -- (311,124.43) ;
\draw [shift={(313,124.43)}, rotate = 180] [color={rgb, 255:red, 0; green, 0; blue, 0 }  ][line width=0.75]    (10.93,-4.9) .. controls (6.95,-2.3) and (3.31,-0.67) .. (0,0) .. controls (3.31,0.67) and (6.95,2.3) .. (10.93,4.9)   ;
\draw    (478.33,136.21) -- (538.38,136.21) ;
\draw    (389.3,136.21) -- (449.35,136.21) ;
\draw    (389.3,136.21) -- (389.3,190) ;
\draw    (538.38,136.21) -- (538.38,167.83) ;
\draw    (389.3,190) -- (538.38,190) ;
\draw    (449.36,128.86) -- (389.31,128.87) ;
\draw    (538.38,128.82) -- (478.34,128.84) ;
\draw    (538.38,128.82) -- (538.36,75.03) ;
\draw    (389.31,128.87) -- (389.29,75.08) ;
\draw    (538.36,75.03) -- (416.77,75.08) ;
\draw    (449.36,128.86) -- (449.35,136.21) ;
\draw    (478.34,128.84) -- (478.33,136.21) ;
\draw    (538.38,180.15) -- (538.38,190) ;
\draw    (389.29,75.08) -- (400.12,75.08) ;
\draw    (338.5,53.71) .. controls (413.2,53.71) and (407.95,53.71) .. (410.72,75.77) ;
\draw [shift={(410.95,77.52)}, rotate = 262.04] [color={rgb, 255:red, 0; green, 0; blue, 0 }  ][line width=0.75]    (10.93,-4.9) .. controls (6.95,-2.3) and (3.31,-0.67) .. (0,0) .. controls (3.31,0.67) and (6.95,2.3) .. (10.93,4.9)   ;
\draw    (545.45,175.22) -- (603,175.22) ;
\draw [shift={(605,175.22)}, rotate = 180] [color={rgb, 255:red, 0; green, 0; blue, 0 }  ][line width=0.75]    (10.93,-4.9) .. controls (6.95,-2.3) and (3.31,-0.67) .. (0,0) .. controls (3.31,0.67) and (6.95,2.3) .. (10.93,4.9)   ;

\draw (106,82) node [anchor=north west][inner sep=0.75pt]  [rotate=-90]  {$\cdots \cdots $};
\draw (68.69,175.78) node [anchor=north west][inner sep=0.75pt]   [align=left] {$\displaystyle b_{n}$ buffers};
\draw (217,82) node [anchor=north west][inner sep=0.75pt]  [rotate=-90]  {$\cdots \cdots $};
\draw (179.87,179) node [anchor=north west][inner sep=0.75pt]   [align=left] {$\displaystyle n$ servers};
\draw (12.23,49.3) node [anchor=north west][inner sep=0.75pt]   [align=left] {Arrivals};
\draw (241,105.46) node [anchor=north west][inner sep=0.75pt]   [align=left] {Departures};
\draw (358.98,34.61) node [anchor=north west][inner sep=0.75pt]    {$\lambda ( t)$};
\draw (453.26,95) node [anchor=north west][inner sep=0.75pt]    {$\eta _{t}$};
\draw (454.92,155) node [anchor=north west][inner sep=0.75pt]    {$\rho _{t}$};
\draw (561.52,152.84) node [anchor=north west][inner sep=0.75pt]    {$d( t)$};

\end{tikzpicture}
}
\caption{{Loss system with buffer and its fluid model}}
\end{figure}

\subsection{Characterization of Relevant Stochastic Processes.}
Let $S_t^{n}$ and $Q_t^{n}$ denote the number of customers in service and buffer, respectively, at time $t$. In addition, let $D_t^n$ denote the cumulative number of departures from the system by time $t$. For the scaled processes, we define 
$$
\bar{S}^n_t:= \frac{S_t^{n}}{n},\, \bar{Q}^n_t:= \frac{Q_t^{n}}{n}, \text{ and } \bar{D}^n_t:=\frac{D^n_t}{n},
$$
to be the $n-$scaled number in service, in buffer and of departures respectively. Let $T_i$, $V_i$, and $D_i$ represent respectively the arrival time, service time, and departure time of the $i$-th customer to the system. Also, let $\mathcal{F}_t^n$ be the filtration generated by the marked point process, $\cf_t^n := \si(\{(T_i, V_i): T_i \leq t\})$. Note that a customer who arrives to find at least one idle server has their arrival time coincide with their service start time. However, a customer who upon arrival finds all servers busy and must first enter the buffer to wait, has their service start time determined by the arrival and service times of prior customers. In addition, their service entry time coincides with the departure time of a prior customer. For this scenario, we let $V_{j_i}$ denote the service time of the customer who enters service at time $D_i$. For simplicity, we initially assume that the number of customers in the system at time $t=0$ is zero. This assumption will be relaxed in the sequel. 

\subsubsection{Busy servers {or customers in service}.}
Observe that the number of busy servers or the number in service at time $t$ consists of customers from two groups: 
\begin{enumerate}[wide, labelwidth=!, labelindent = 0pt]
\item[(a)] \emph{Customers admitted directly upon arrival.} This scenario is similar to the setup of Section~\ref{sec:zero-buffer}. 
Observe that the number of customers at time $t$, who were directly admitted upon arrival, consists of all arrivals to the system such that all of the following conditions are met:
\begin{enumerate}[label = (\roman*), wide, labelwidth=!, labelindent=0pt]
\item the customer arrival occurs {at or} prior to time $t$, 
\item the number of occupied servers upon the customer's arrival is less than $n$, and
\item the remaining service time of this customer at time $t$ is positive, that is, the customer is yet to depart the system. 
\end{enumerate}
For the $i-$th customer arriving to the system, these conditions correspond to $\{T_i \leq t\}$, $\{S_{T_i-}^{n} < n\}$ or $\{\bar{S}_{T_i-}^{n} < 1\}$, and $\{V_i > t-T_i\}$ respectively. Consequently, the number of customers at time $t$, who were directly admitted upon arrival equals:
\beq\label{eq:d-s}
\sum_{i=1}^{\infty} \mathbbm{1}_{\{T_{i} \leq t\} } \mathbbm{1}_{\{S_{T_{i}-}^{n}<n\}} \mathbbm{1}_{\{V_{i} >t-T_{i}\}}.
\eeq

\item[(b)] \emph{Customers promoted from the buffer.} The customers in this scenario start service at the departure time $D_i$ of some customer $i$. Observe that the number of customers at time $t$, who were promoted from the buffers, consists of all departures such that all of the following conditions are met:
\begin{enumerate}[label = (\roman*), wide, labelwidth=!, labelindent=0pt]
\item the service start time $D_i$ of this customer is {at or prior to} time $t$,
\item the buffer is non-empty at time $D_i-$, and
\item the remaining service time of this customer at time $t$ is positive, that is, the customer is yet to depart the system.
\end{enumerate}
For the customer promoted from buffer at time $D_i$, these conditions correspond to $\{D_i \leq t\}$, $\{Q_{D_i-}^{n} > 0\}$ or $\{\bar{Q}_{D_{i}-}^{n}>0\}$, and $\{V_{j_i} > t-D_i\}$ respectively. Consequently, the number of customers at time $t$, who were promoted from the buffer equals:
\beq\label{eq:b-s}
\sum_{i=1}^{\infty}\mathbbm{1}_{\{ D_{i} \leq t\}} \mathbbm{1}_{\{Q_{D_i-}^{n} > 0\}}  \mathbbm{1}_{\{ D_{i}+V_{j_i}> t\}}.
\eeq
\end{enumerate} 
Therefore, by combining the two groups of customers from \eqref{eq:d-s}-\eqref{eq:b-s}, we have the number of customers in service at time $t$ satisfies
\begin{equation}\label{eq:N^nbuf}
S_{t}^{n}=\sum_{i=1}^{\infty} \mathbbm{1}_{\{T_{i} \leq t\} } \mathbbm{1}_{\{S_{T_{i}-}^{n}<n\}} \mathbbm{1}_{\{V_{i} >t-T_{i}\}}+\sum_{i=1}^{\infty} \mathbbm{1}_{\{ D_{i} \leq t\}}\mathbbm{1}_{\{Q_{D_i-}^{n} > 0\}}  \mathbbm{1}_{\{ D_{i}+V_{j_i}> t\}} .
\end{equation}
On scaling \eqref{eq:N^nbuf} by $n$, we have in contrast to \eqref{eq:rho^n1} that the scaled number of busy servers satisfy
\begin{equation}\label{eq:rho^nDef}
\bar{S}^n_t=\frac{1}{n} \sum_{i=1}^{\infty} \mathbbm{1}_{\{T_{i} \leq t\} }\mathbbm{1}_{\{\bar{S}_{T_{i}-}^{n}<1\}} \mathbbm{1}_{\{V_{i} >t-T_{i}\}}+\frac{1}{n} \sum_{i=1}^{\infty} \mathbbm{1}_{\{ D_{i} \leq t\}}\mathbbm{1}_{\{\bar{Q}_{D_{i}-}^{n}>0\}}  \mathbbm{1}_{\{ D_{i}+V_{j_i}> t\}}.
\end{equation}

\subsubsection{Occupied buffers.} 
The number of occupied buffers equals the difference between two groups:
\begin{enumerate}[wide, labelwidth=!, labelindent=0pt]
\item[(a)] \emph{Customers that entered the buffer.} Observe that the total number of customers who entered the buffer by time $t$ consists of those individuals who satisfy all of the following conditions:
\begin{enumerate}[label = (\roman*), wide, labelwidth=!, labelindent=0pt]
\item the customer arrival occurs {at or} prior to time $t$, 
\item the number of occupied servers upon the customer's arrival is $n$, and
\item the buffer upon the customer's arrival is not full.
\end{enumerate}
For the $i-$th customer arriving to the system, these conditions correspond to $\{T_i \leq t\}$, $\{S_{T_i-}^{n} = n\}$ or $\{\bar{S}_{T_i-}^{n} = 1\}$, and $\{Q_{T_i-}^{n} < b_n\}$ or $\{\bar{Q}_{T_i-}^{n} < \frac{b_n}{n}\}$ respectively. Consequently, the number of customers who entered the buffer by time $t$ equals:
\begin{equation}\label{eq:e-b}
\sum_{i=1}^{\infty} \mathbbm{1}_{\{T_{i} \leq t\}} \mathbbm{1}_{\{S_{T_i-}^{n} = n\}} \mathbbm{1}_{\{Q_{T_i-}^{n} < b_n\}} .
\end{equation}
\item[(b)] \emph{Customers that exited the buffer.} The customers in this scenario start service at the departure time $D_i$ of some customer $i$. Observe that the total number of customers who departed from the buffer by time $t$ consists of those individuals who satisfy all of the following conditions:
\begin{enumerate}[label = (\roman*), wide, labelwidth=!, labelindent=0pt]
\item the service start time $D_i$ of this customer is {at or prior to} time $t$, and
\item the buffer is non-empty at time $D_i-$.
\end{enumerate}
For the customer departing from buffer at time $D_i$, these conditions correspond to $\{D_i \leq t\}$ and $\{Q_{D_i-}^{n} > 0\}$ or $\{\bar{Q}_{D_{i}-}^{n}>0\}$ respectively. Consequently, the number of customers that exited the buffer by time $t$ equals:
\begin{equation}\label{eq:l-b}
\sum_{i=1}^{\infty} \mathbbm{1}_{\{D_{i} \leq t\}} \mathbbm{1}_{\{Q_{D_i-}^{n} > 0\}}.
\end{equation}
\end{enumerate}
{Therefore, by taking} the difference between the two groups of customers from \eqref{eq:e-b}-\eqref{eq:l-b}, we have the number of customers in buffer at time $t$ satisfies
\begin{equation}\label{eq:N^yDef}
Q_{t}^{n}=\sum_{i=1}^{\infty} \mathbbm{1}_{\{T_{i} \leq t\}} \mathbbm{1}_{\{S_{T_i-}^{n} = n\}} \mathbbm{1}_{\{Q_{T_i-}^{n} < b_n\}}  - \sum_{i=1}^{\infty} \mathbbm{1}_{\{D_{i} \leq t\}} \mathbbm{1}_{\{Q_{D_i-}^{n} > 0\}}.
\end{equation}	
{On} scaling \eqref{eq:N^yDef} by $n$, this yields that the scaled number in buffer satisfy
\begin{equation}\label{eq:y^nDef}
\bar{Q}^n_t=  \frac{1}{n} \sum_{i=1}^{\infty} \mathbbm{1}_{\{T_{i} \leq t\}}\mathbbm{1}_{\{\bar{S}_{T_{i}-}^{n}=1\}} \mathbbm{1}_{\{\bar{Q}_{T_{i}-}^{n}<\frac{b_n}{n}\}} -\frac{1}{n} \sum_{i=1}^{\infty} \mathbbm{1}_{\{D_{i} \leq t\}} \mathbbm{1}_{\{\bar{Q}_{D_{i}-}^{n}>0\}}.
\end{equation}	
\subsubsection{Departures} Observe that the cumulative number of departures also include customers from two groups: 
\begin{enumerate}[wide, labelwidth=!, labelindent=0pt]
\item[(a)] \emph{Departure of customers admitted directly upon arrival.} Observe that the number of customers at time $t$, who were directly admitted upon arrival and then departed from the system, consists of those customers who satisfy all of the following conditions:
\begin{enumerate}[label = (\roman*), wide, labelwidth=!, labelindent=0pt]
\item the customer arrival occurs {at or} prior to time $t$, 
\item the number of occupied servers upon the customer's arrival is less than $n$, and
\item the customer has departed from the system {by} time $t$.
\end{enumerate}
For the $i-$th customer arriving to the system, these conditions correspond to $\{T_i \leq t\}$, $\{S_{T_i-}^{n} < n\}$ or $\{\bar{S}_{T_i-}^{n} < 1\}$, and $\{T_i+V_i\leq t\}$ respectively. Consequently, the number of customers at time $t$, who were admitted directly upon arrival and then departed equals:
\begin{equation}\label{eq:d-d}
\sum_{i=1}^{\infty} \mathbbm{1}_{\{T_{i} \leq t\} } \mathbbm{1}_{\{S_{T_{i}-}^{n}<n\}} \mathbbm{1}_{\{T_{i}+V_{i}\leq t\}}.
\end{equation}
\item[(b)] \emph{Departure of customers promoted from the buffer.} The customers in this scenario start service at the departure time $D_i$ of some customer $i$. Observe that the number of customers at time $t$, who were promoted from the buffers and then departed from the system, consists of those customers who satisfy all of the following conditions:
\begin{enumerate}[label = (\roman*), wide, labelwidth=!, labelindent=0pt]
\item the service start time $D_i$ of this customer is {at or prior to} time $t$,
\item the buffer is non-empty at time $D_i-$, and
\item the customer has departed from the system at time $t$.
\end{enumerate}
For the customer promoted from buffer at time $D_i$, these conditions correspond to $\{D_i \leq t\}$, $\{Q_{D_i-}^{n} > 0\}$ or $\{\bar{Q}_{D_{i}-}^{n}>0\}$, and $\{ D_i+V_{j_i}\leq t\}$ respectively. Consequently, the number of customers at time $t$, who were promoted from the buffer and departed equals:
\begin{equation}\label{eq:b-d}
\sum_{i=1}^{\infty} \mathbbm{1}_{\{ D_{i} \leq t\}} \mathbbm{1}_{\{Q_{D_i-}^{n} > 0\}}  \mathbbm{1}_{\{ D_{i}+V_{j_i}\leq t\}} .
\end{equation}
	\end{enumerate} 
{Therefore}, by combining the two groups of customers from \eqref{eq:d-d}-\eqref{eq:b-d} we have the cumulative departures at time $t$ satisfies
\begin{equation*}
D_{t}^{n}=\sum_{i=1}^{\infty} \mathbbm{1}_{\{T_{i} \leq t\} } \mathbbm{1}_{\{S_{T_{i}-}^{n}<n\}} \mathbbm{1}_{\{V_{i}+T_{i}\leq t\}}+\sum_{i=1}^{\infty} \mathbbm{1}_{\{Q_{D_i-}^{n} > 0\}} \mathbbm{1}_{\{ D_{i} \leq t\}} \mathbbm{1}_{\{ D_{i}+V_{j_i}\leq t\}} .
\end{equation*}
{On} scaling by $n$ we have
\begin{equation}\label{eq:D^nDef}
\bar{D}_{t}^{n}= \frac{1}{n}\sum_{i=1}^{\infty} \mathbbm{1}_{\{D_{i} \leq t\}}= \frac{1}{n}\sum_{i=1}^{\infty} \mathbbm{1}_{\{T_{i} \leq t\}} \mathbbm{1}_{\{\bar{S}_{T_{i}-}^{n}<1\}}  \mathbbm{1}_{\{V_{i}+T_{i} \leq t\}} + \frac{1}{n}\sum_{i=1}^{\infty} \mathbbm{1}_{\{\bar{Q}_{D_{i}-}^{n}>0\}} \mathbbm{1}_{\{D_{i}+V_{j_i} \leq t\}} .
\end{equation}    
\subsection{{Stochastic Integral Representation}.}
As in Section~\ref{sec:zero-buffer}, we will use random measures to obtain cleaner representations of the processes under consideration. To that effect, we define:
\begin{align*}
& W_n^{s,A}(t,u,x)=\frac{1}{n} \mathbbm{1}_{\{u \leq t\} } \mathbbm{1}_{\{\bar{S}_{u-}^{n}<1\}} \mathbbm{1}_{\{x >t-u\}} , 
&W_n^{s,D}(t,u,x)=\frac{1}{n}  \mathbbm{1}_{\{ u \leq t\}}\mathbbm{1}_{\{\bar{Q}_{u-}^{n}>0\}}  \mathbbm{1}_{\{ x> t-u\}},\\
& W_n^{q,A}(t,u,x)=\frac{1}{n}  \mathbbm{1}_{\{u \leq t\}}\mathbbm{1}_{\{\bar{S}_{u-}^{n}=1\}} \mathbbm{1}_{\{\bar{Q}_{u-}^{n}<\frac{b_n}{n}\}} ,
&W_n^{q,D}(t,u,x)=\frac{1}{n}  \mathbbm{1}_{\{u \leq t\}} \mathbbm{1}_{\{\bar{Q}_{u-}^{n}>0\}},\\
& W_n^{d, A}(t,u,x)=\frac{1}{n} \mathbbm{1}_{\{u \leq t\}}\mathbbm{1}_{\{\bar{S}_{u-}^{n}<1\}}  \mathbbm{1}_{\{x \leq t-u\}},
&W_n^{d,D}(t,u,x)=\frac{1}{n} \mathbbm{1}_{\{\bar{Q}_{u-}^{n}>0\}} \mathbbm{1}_{\{x \leq t-u\}}.
\end{align*}
Using these notations, the relations \eqref{eq:rho^nDef}, \eqref{eq:y^nDef} and \eqref{eq:D^nDef} can be expressed as stochastic integrals
\begin{equation}\label{eq:S-Q-D}
\bar{S}^n_t=\bar{S}^{n,A}_t+\bar{S}^{n,D}_t,\quad\bar{Q}^n_t=\bar{Q}^{n,A}_t-\bar{Q}^{n,D}_t,\quad\bar{D}^n_t=\bar{D}^{n,A}_t+\bar{D}^{n,D}_t,
\end{equation}
where
\begin{align}
&\bar{S}^{n,A}_t:=\int_0^t\int_\R W_n^{s,A}(t,u,x) \mathcal{M}^{n,A} (du,dx),\quad\bar{S}^{n,D}_t:=\int_0^t\int_\R W_n^{s,D}(t,u,x) \mathcal{M}^{n,D} (du,dx),\label{eq:rho^n}\\
&\bar{Q}^{n,A}_t:=\int_0^t\int_\R W_n^{q,A}(t,u,x) \mathcal{M}^{n,A} (du,dx),\quad \bar{Q}^{n,D}_t:=\int_0^t\int_\R W_n^{q,D}(t,u,x) \mathcal{M}^{n,D} (du,dx),\label{eq:y^n}\\
&\bar{D}^{n,A}_t:=\int_0^t\int_\R W_n^{d, A}(t,u,x) \mathcal{M}^{n,A} (du,dx),\quad \bar{D}^{n,D}_t:=\int_0^t\int_\R W_n^{d,D}(t,u,x) \mathcal{M}^{n,D} (du,dx).\label{eq:D^n}
\end{align}
Here $\mathcal{M}^{n,A}$ is the counting measure associated with the marked point process of the arrival and service time pairs $(T_i,V_i)$, and $\mathcal{M}^{n,D}$ is the counting measure associated with the marked point process of the departure and service time pairs $(D_i,V_{j_i})$. Since the number of cumulative departures in $[0,t]$ is bounded by the number of arrivals in the same interval, the departure process is a locally finite point process.

Similar to Proposition~\ref{prop:weaklimits0}, in order to obtain the limit of the stochastic processes $(\bar{S}^n, \bar{Q}^n, \bar{D}^n)$ as $n$ goes to infinity, we again begin with a result proving convergence along a subsequence.
\begin{proposition}\label{prop:weaklimits}
	Let Assumption~\ref{asm:queue} hold. Assume that the system starts empty, that is, the number of customers at time $0$ is zero. Then
\begin{enumerate}[label = (\roman*), wide, labelwidth = !, labelindent = 0pt]
\item For any $T > 0$, almost all $\om\in\Omega$ and for any subsequence, there exists a further subsequence $(r_k)$, possibly depending on $\om$, and continuous, possibly stochastic processes $\rho,\eta,D$ such that
	\begin{equation}\label{eq:limitsub}
\bar{S}^{r_k}_t\rightarrow\rho_t,\quad \bar{Q}^{r_k}_t\rightarrow\eta_t, \quad\bar{D}_t^{r_k}\rightarrow D_t,
	\end{equation} 
in the uniform topology.
\item {Moreover, given $(r_k)$, there exist bounded, possibly stochastic processes $w^{1}, w^{2}, w^{3}$ such that }
\begin{equation}\label{eq:rhoydweak*w}
\mathbbm{1}_{\{\bar{S}^{r_k}_{t-}<1\}}\stackrel{*}{\rightharpoonup}w^1(t), \quad 
\mathbbm{1}_{\{\bar{Q}^{r_k}_{t-}<\frac{b_{r_k}}{r_k}\}}\stackrel{*}{\rightharpoonup}w^3(t), \quad \text{in } L^\infty[0,T],
\end{equation}
and
\begin{equation*}
\widehat D_t^{r_k}:=\int_0^t\mathbbm{1}_{\{\bar Q^{r_k}_{u-}>0\}}\,d\bar D_u^{r_k}\to \int_0^tw^2(u)dD_u, \quad \text{uniformly in }[0,T].
\end{equation*}
\item {Furthermore, almost surely, $(\rho,\eta,D,w^1,w^2,w^3)$ defined in \eqref{eq:limitsub}-\eqref{eq:rhoydweak*w} satisfy}
		\begin{align}
		&\rho_{t}=\int_{0}^{t} w^{1}(u) \bar{G}(t-u) \lambda(u) d u  +\int_{0}^{t} w^{2}(u) \bar{G}(t-u) d(u) d u, \label{eq:rhow}\\
		& \eta_{t}=\int_{0}^{t} (w^{3}(u)-w^1(u)) \lambda(u) d u  -\int_{0}^{t} w^{2}(u) d(u) d u, \label{eq:yw}\\
		& D_t=\int_{0}^{t}w^{1}(u) G(t-u) \lambda(u) d u  +\int_{0}^{t} w^{2}(u) G(t-u) d(u) d u,\label{eq:dw}	
\end{align}
and for almost every $t\in[0,T]$
\begin{equation*}
\mathbbm{1}_{\{\rho_t<1\}} \leq w^1(t) \leq 1, \quad
\mathbbm{1}_{\left\{\eta_{t}>0\right\}} \leq w^2(t) \leq 1, \quad
\mathbbm{1}_{\{\eta_t<\beta\}}\leq w^3(t) \leq 1.
\end{equation*}
Moreover, the limits obey the state constraints
\begin{equation}\label{eq:solbdn0}
0\leq\rho_t\leq1,\qquad 0\leq\eta_t\leq\beta,\qquad \eta_t\lp1-\rho_t\rp=0,\qquad t\in[0,T].
\end{equation}
That is, for almost all $\om\in\Omega$, $(\rho{(\om)},\eta{(\om)},D,w^1{(\om)},w^2{(\om)},w^3{(\om)})$ as in \eqref{eq:rhow}-\eqref{eq:dw} is a solution, interpreted according to Definition~\ref{def:sol_Vol}, to the following non-linear discontinuous Volterra integral equation
\begin{align}\label{eq:rhoVoln}
		&\rho_{t}=\int_{0}^{t} \mathbbm{1}_{\{\rho_{u-}<1\}} \bar{G}(t-u) \lambda(u) d u  +\int_{0}^{t} \mathbbm{1}_{\{\eta_{u-}>0\}} \bar{G}(t-u) d(u) d u, \nonumber\\
		& \eta_{t}=\int_{0}^{t} ( \mathbbm{1}_{\{\eta_{u-}<\beta\}}-\mathbbm{1}_{\{\rho_{u-}<1\}}) \lambda(u) d u  -\int_{0}^{t} \mathbbm{1}_{\{\eta_{u-}>0\}} d(u) d u, \nonumber\\
		& D_t=\int_{0}^{t} \mathbbm{1}_{\{\rho_{u-}<1\}} G(t-u) \lambda(u) d u  +\int_{0}^{t} \mathbbm{1}_{\{\eta_{u-}>0\}} G(t-u) d(u) d u.
	\end{align}
Here $D$ is absolutely continuous and has density $d$ satisfying
$$
d(t)=\int_0^tg(t-u)w^1(u)\lambda(u)du+\int_0^tg(t-u)w^2(u)d(u)du \quad \text{a.e.}
$$
\end{enumerate}
\end{proposition}
\begin{proof}
For simplicity we will consider the initial subsequence to be $(n)$, but the arguments below go through for any initial subsequence.

\vspace{0.1in}
\noindent
\emph{Part (i).}
\noindent
\emph{Arrivals affecting number in system}: We first investigate the stochastic integrals with respect to the random measure $\mathcal{M}^{n,A}$. Similar to \eqref{eq:M-comp}, we have
\begin{equation}\label{eq:M-comp2}
    \mathcal{M}^{n,A}_*=\mathcal{M}^{n,A}-\mathcal{M}^{n,A}_c,
\end{equation}
where $\mathcal{M}^{n,A}_c(du,dx)=\E[\mathcal{M}^{n,A}(du,dx)]=n\la(u)g(x)dudx$. By the decomposition \eqref{eq:M-comp2}
, the first term in \eqref{eq:rho^n} becomes:
\begin{equation}\label{eq:rhonA}
\bar{S}^{n,A}_t:=\int_0^t\int_\R W_n^{s,A}(t,u,x) \mathcal{M}^{n,A} (du,dx)=X_t^{s,n,A}+Y_t^{s,n,A},
	\end{equation}
where
\begin{equation*}
X_t^{s,n,A}:=\int_0^t\int_\R W_n^{s,A}(t,u,x) \mathcal{M}^{n,A}_* (du,dx),\quad\text{and} \quad Y_t^{s,n,A}:=\int_{0}^{t} \mathbbm{1}_{\{\bar{S}^n_{u-}<1\}} \bar{G}(t-u) \lambda(u) d u.
\end{equation*}
We can follow the same argument as in the proof of Proposition~\ref{prop:weaklimits0} to conclude similar to how we obtained \eqref{eq:rho0weakstar} that for almost all $\om\in\Omega$ and any subsequence $(l_k(\om))$, there exists a subsubsequence $(r_k(\om)) \subset(l_k(\om))$ and $w^1(\om) \in L^\infty[0,T]$, such that for any $\phi \in L^1[0,T]$ 
	\begin{equation}\label{eq:rhoweakstar}
		\lim _{k \rightarrow \infty} \int_0^t \phi(u) \mathbbm{1}_{\{\bar{S}^{r_k}_{u-}<1\}} d u=\int_0^t \phi(u) w^1(u) d u.
	\end{equation}
Furthermore, using similar arguments to how we obtained \eqref{eq:rho} we get for almost all $\om\in \Omega$ and any subsequence, there exists a further subsequence $(r_k)$ such that
	\begin{equation}\label{eq:rhoA}
		\bar{S}_t^{r_k,A}=X_t^{s,r_k, A}+Y_t^{s,r_k,A}\rightarrow \int_0^t w^{1}(u) \bar{G}(t-u) \lambda(u) d u:=Y^{s,A}.
	\end{equation}
in the uniform topology.

Similarly, for almost all $\om\in\Omega$ and any subsequence, there exists a further subsequence $(r_k(\om))$ such that the first term in \eqref{eq:y^n} and \eqref{eq:D^n} have the following convergence in the uniform topology:
\begin{align}
    &\bar{Q}^{r_k,A}_t:=\int_0^t\int_\R W_{r_k}^{q,A}(t,u,x) \mathcal{M}^{r_k,A} (du,dx)\to\int_{0}^{t} (w^3(u)-w^1(u)) \lambda(u) d u, \label{eq:QA}\\
    &\bar{D}^{r_k,A}_t:=\int_0^t\int_\R W_{r_k}^{d,A}(t,u,x) \mathcal{M}^{r_k,A} (du,dx)\to \int_{0}^{t}w^{1}(u) G(t-u) \lambda(u) d u ,\label{eq:DA}
\end{align}
where $w^3$ is a bounded function possibly depending on $\om$. Here in \eqref{eq:QA}, we used the pointwise identity below that follows from $\{\bar{S}^n_{u-}<1\}\subseteq\{\bar{Q}^n_{u-}=0\}\subseteq\{\bar{Q}^n_{u-}<b_n/n\}$ (a customer waits only when every server is busy)
$$
W_n^{q,A}(t,u,x)=\frac{1}{n}  \mathbbm{1}_{\{u \leq t\}}\left(1-\mathbbm{1}_{\{\bar{S}_{u-}^{n}<1\}}\right)\mathbbm{1}_{\{\bar{Q}_{u-}^{n}<\frac{b_n}{n}\}} =\frac{1}{n}  \mathbbm{1}_{\{u \leq t\}}\left( \mathbbm{1}_{\{\bar{Q}_{u-}^{n}<\frac{b_n}{n}\}}-\mathbbm{1}_{\{\bar{S}_{u-}^{n}<1\}}\right).
$$

\vspace{0.1in}
In this part we look at the stochastic integrals with respect to the random measure $\mathcal{M}^{n,D}$. Unlike the measure $\mathcal{M}^{n,A}$, the measure $\mathcal{M}^{n,D}$ does not have a deterministic compensator. Instead, we consider the scaled cumulative number of buffer-to-service transfers
$$
\widehat D_t^n:=\int_0^t\mathbbm{1}_{\{\bar Q^n_{u-}>0\}}\,d\bar D_u^n=\frac1n\sum_{i\ge1}\mathbbm{1}_{\{\bar Q^n_{D_i-}>0\}}\mathbbm{1}_{\{D_i\le t\}}.
$$
The process $\widehat D^n$ is non-decreasing and
$\widehat D_0^n=0$.

\vspace{0.1in}
\noindent
\emph{The departure process}: From \eqref{eq:S-Q-D}, the departure process has the decomposition:
\begin{equation}\label{eq:Dn-dcom}
    \bar{D}^n_t=\bar{D}_t^{n,A}+\bar{D}_t^{n,D}=\bar{D}_t^{n,A}+Y^{n,D}_t+X^{n,D}_t,
\end{equation}
where
$$
X^{n,D}_t:=\frac{1}{n} \sum_{i\geq1}\mathbbm{1}_{\{\bar{Q}^n_{D_i-}>0\}}\mathbbm{1}_{\{D_i\leq t\}}\left[\mathbbm{1}_{\left\{V_{j_i}\leq t-D_i\right\}}-G\left(t-D_i\right)\right],\quad Y_t^{n,D}:=\int_0^tG(t-u)\,d\widehat D_u^n.
$$
The first term is treated in \eqref{eq:DA}. Now let us look at the second term $Y^{n,D}_t$. Since $
A^n_T/n\to \Lambda(T)=\int_0^T\la(u)du$ almost surely, we can consider the probability one set $\Omega_A:=\{A_T^n/n \to\Lambda(T)\}$ where $\Lambda(T)=\int_0^T\lambda(u)du<\infty$. Fix $\om\in \Omega_A$, there exist $C(\om)$ such that
\begin{equation}\label{eq:muD-bound}
\sup_n\widehat D_T^n(\omega)\leq\sup_n\bar D_T^n(\omega)\leq\sup_n\frac{A_T^n(\omega)}n\leq C(\omega).
\end{equation}
Since $\widehat D^n_t $ are non-decreasing and uniformly bounded, by Helly's selection theorem, for any subsequence we may extract a further subsequence $(n_k(\om))$ such that
\begin{equation}\label{eq:muD-weak}
\widehat D_t^{n_k}\to\widehat D_t,\quad \text{for every }t\in[0,T],
\end{equation}
Since $G(0)=\widehat D_0^n=0$ and $G$ has density $g$, by integration by parts and dominated convergence theorem for every $t\in[0,T]$,
\begin{equation}\label{eq:YDid}
Y_t^{n_k,D}=\int_0^t g(t-u)\widehat D_u^{n_k}\,du\to Y_t^D:=\int_0^t g(t-u)\widehat D_u\,du=\int_0^tG(t-u)d\widehat D_u .
\end{equation}
Moreover, since $G$ is absolutely continuous, non-decreasing and bounded by $1$, using \eqref{eq:muD-bound} we have $\sup_n\sup_tY_t^{n,D}\leq \sup_n\widehat D_T^n(\omega) \leq C(\om),$ and 
\begin{align}\label{eq:YnD-cont}
\left|Y_t^{n,D}-Y_s^{n,D}\right|&=\int_s^tG(t-u)d\widehat D_u^n+\int_0^s(G(t-u)-G(s-u))d\widehat D_u^n \nonumber\\
&\leq G(t-s)(\widehat D_t^n-\widehat D_s^n)
+w_G(t-s)\widehat D_s^n \leq 2C(\omega)\om_G(t-s).
\end{align}
Here the last inequality follows from $G(t-s)=|G(t-s)-G(0)|\leq \om_G(t-s)$ where $\om_G(\der)$ is defined after \eqref{eq:tildemod}. Since $G$ is uniform continuous on $[0,T]$, it follows that $Y^{n,D}$ are uniformly equicontinuous for each $\om\in \Omega_A $. Therefore, by Arzel\`a-Ascoli \cite[Thm.\ 3.12]{bressan2012lecture}, we have that for almost all $\om \in \oom$, there exists a subsequence $(l_k(\om))\subset (n_k(\om))$ such that
\begin{equation}\label{eq:YtoYD}
    Y^{l_k,D} {\rightarrow}Y^D, \quad \text{uniformly in }[0,T],
\end{equation}
where $Y^D$ is identified by \eqref{eq:YDid}.

Let us look at the third term $X^{n,D}$ in \eqref{eq:Dn-dcom}. Similar to the proof of Proposition~\ref{prop:weaklimits0}, $X^{n,D}_t$ itself is not a martingale. However, we can use a similar strategy to consider martingales when $t$ is frozen. To this effect, we use a filtration in which each service time is revealed when the corresponding customer enters service. Let $E_j^n$ denote the service-entry time of customer $j$, with $E_j^n=\infty$ if that customer never enters service, and define
$$\mathcal{H}_t^n:=\sigma\left(T_i:T_i\leq t\right)\vee \sigma\left((E_j^n, V_j): E_j^n \leq t\right).$$
Then $\mathcal{H}_{D_i-}^n$ contains all the information available at the departure time $D_i$ but the service time of the entering customer $V_{j_i}$ is unrevealed. Then $V_{j_i}$ has conditional distribution $G$. Thus,
\begin{equation}\label{eq:cond-Vji}
\E\left[\mathbbm{1}_{\left\{V_{j_i}\leq t-D_i\right\}}-G\left(t-D_i\right) \mid \mathcal{H}_{D_i-}^n \right]=0.
\end{equation}
Now fix time $t\in[0,T]$ and define
$$
M_r^{n, t}=\frac{1}{n} \sum_{i\geq1}\mathbbm{1}_{\{\bar{Q}^n_{D_i-}>0\}}\mathbbm{1}_{\{D_i\leq r\}}\left[\mathbbm{1}_{\left\{V_{j_i}\leq t-D_i\right\}}-G\left(t-D_i\right)\right], \quad 0 \leq r \leq t .
$$
Take $0 \leq s<r \leq t$ and $A \in \mathcal{H}_s^n$. Since $\mathbbm{1}_A \mathbbm{1}_{\left\{s<D_i\right\}}$ is known before $D_i$, so it is $\mathcal{H}_{D_i-}^n$-measurable. Moreover, $\mathbbm{1}_{\{\bar{Q}^n_{D_i-}>0\}}$ is included in $\mathcal{H}_{D_i-}^n$. Therefore, $\mathbbm{1}_A \mathbbm{1}_{\{\bar{Q}^n_{D_i-}>0\}} \mathbbm{1}_{\left\{s<D_i \leq r\right\}}$ is $\mathcal{H}_{D_i-}^n$-measurable. Hence, by the tower property of conditional expectation and \eqref{eq:cond-Vji} we have
\begin{align*}
\E&\left[\mathbbm{1}_A(M_r^{n, t}-M_s^{n, t})\right]=\frac{1}{n} \sum_{i\geq1}\E\left[\E\left[\mathbbm{1}_A \mathbbm{1}_{\{\bar{Q}^n_{D_i-}>0\}} \mathbbm{1}_{\left\{s<D_i \leq r\right\}}\left(\mathbbm{1}_{\left\{V_{j_i}\leq t-D_i\right\}}-G\left(t-D_i\right)\right)\mid \mathcal{H}_{D_i-}^n\right]\right]\\
&\qquad\qquad\qquad\qquad=\frac{1}{n} \sum_{i\geq1}\E\left[\mathbbm{1}_A \mathbbm{1}_{\{\bar{Q}^n_{D_i-}>0\}} \mathbbm{1}_{\left\{s<D_i \leq r\right\}}\E\left[\mathbbm{1}_{\left\{V_{j_i}\leq t-D_i\right\}}-G\left(t-D_i\right) \mid \mathcal{H}_{D_i-}^n \right]\right]=0.
\end{align*}
This implies $M_r^{n,t}$ is a $\mathcal{H}^n$-martingale. Observe that $|\mathbbm{1}_{\{V_{j_i}\leq t-D_i\}}-G(t-D_i)|\leq 1 $, the quadratic variation of $M_r^{n,t}$:
\begin{align*}
[M^{n,t}]_{r}&=\frac{1}{n^2}\sum_{i\geq 1}\left(\mathbbm{1}_{\{\bar{Q}^n_{D_i-}>0\}}\mathbbm{1}_{\{D_i\leq r\}}\right)^2\left[\mathbbm{1}_{\left\{V_{j_i}\leq t-D_i\right\}}-G\left(t-D_i\right)\right]^2\\
&\leq \frac{1}{n^2}\sum_{i\geq 1}\mathbbm{1}_{\{\bar{Q}^n_{D_i-}>0\}}\mathbbm{1}_{\{D_i\leq r\}}\leq \frac{1}{n^2}D_T^n\leq \frac{1}{n^2}A^n_T.
\end{align*}
Since $M^{n,t}_{t}=X^{n,D}_{t}$, by the Burkholder--Davis--Gundy inequality there exists a constant $c>0$ such that
\begin{equation*}
\E[|X^{n,D}_{t}|^4]=\E[|M^{n,t}_{t}|^4]\leq \E\left(\sup_r |M^{n,t}_r|\right)^4\leq c\E\left([M^{n,t}]_{t}^2\right)\leq \frac{c}{n^4}\E\left[(A^n_T)^2\right]\leq\frac{c'}{n^2},
\end{equation*}
where $c'>0$ is a finite constant and the last inequality follows from $\mathbb{E}[(A_T^n)^2]=n \Lambda(T)+n^2 \Lambda(T)^2$. Thus, by Markov's inequality and Borel--Cantelli Lemma we have for any fixed $t\in[0,T]$
\begin{equation}\label{eq:XDto0-fix-t}
 X_{t}^{n,D}\to 0 \quad\text{ a.s.} 
\end{equation}
Furthermore, 
$$
X_t^{n,D}=\bar{D}^{n,D}_t-Y_t^{n,D}.
$$
Recall that $\bar{D}^{n,D}_t= \frac{1}{n}\sum_{i\geq 1} \mathbbm{1}_{\{\bar{Q}_{D_{i}-}^{n}>0\}} \mathbbm{1}_{\{D_{i}+V_{j_i} \leq t\}}$. Since $\bar{D}^{n,D},Y^{n,D}$ are non-decreasing, choosing a deterministic grid $0=s_0<\cdots<s_{k}=T$ with mesh size tends to zero as $k$ grows and applying Lemma~\ref{lem:to-be-written} we have
$$
\sup_t|X^{n,D}_t|\leq \max_{0\leq j \leq k}|\bar{D}^{n,D}_{s_j}-Y_{s_j}^{n,D}|+\max_{0\leq j \leq k}(Y^{n,D}_{s_j}-Y^{n,D}_{s_{j-1}}).
$$
Letting $n\to\infty$ and then $k\to\infty$, by \eqref{eq:YnD-cont} and \eqref{eq:XDto0-fix-t} the right hand side of the above inequality goes to 0. Excluding the union of the corresponding null sets over $j$ and over $k\in\N$ we obtain
\begin{equation}\label{eq:XDto0}
    \lim_{n\to \infty}\sup_{t\in[0,T]}|X^{n,D}_t|=0 \quad \text{ almost surely}.
\end{equation}
Therefore, using \eqref{eq:DA}, \eqref{eq:YtoYD} and \eqref{eq:XDto0} in \eqref{eq:Dn-dcom} we can conclude that for any subsequence, there exists a further subsequence $(r_k(\om))\subset (l_k(\om))$ such that almost surely
\begin{equation}\label{eq:DtoDU}
    \bar{D}^{r_k}\to D \quad\text{uniformly in }[0,T] ,
\end{equation}
where
\begin{equation}\label{eq:D-pre}
    D_t=\int_{0}^{t}w^{1}(u) G(t-u) \lambda(u) d u+\int_0^tG(t-u)d\widehat D_u .
\end{equation}
Moreover, using $G(x)=\int_0^x g(s) d s$, by Fubini's theorem we have
\begin{align*}
    D_t&=\int_{0}^{t}\int_u^t g(s-u)ds \,w^{1}(u)\lambda(u) d u+\int_0^t\int_u^t g(s-u)ds\,d\widehat D_u \\
    &=\int_0^t\left[\int_0^sg(s-u)w^1(u)\lambda(u)du\right]ds+\int_0^t\left[\int_0^sg(s-u)d\widehat D_u\right]ds,
\end{align*}
and it follows that $D$ is absolutely continuous and has density 
$$
d(t)=\int_0^tg(t-u)w^1(u)\lambda(u)du+\int_0^tg(t-u)d\widehat D_u \quad \text{a.e.}
$$
On the other hand, by definition of $\widehat D^n$, for $0\le s<t\le T$, we have $0\le\widehat D_t^n-\widehat D_s^n\le\bar D_t^n-\bar D_s^n.$
Passing to the limit using \eqref{eq:muD-weak}
and \eqref{eq:DtoDU} gives
$$
0\le\widehat D_t-\widehat D_s\le D_t-D_s.
$$
Therefore, $d\widehat D$ is absolutely continuous with respect to $dD$. By Radon–Nikodym theorem there exist $0\leq w^2\leq 1$ $dD$-a.e. such that
\begin{equation}\label{eq:muD-w2}
d\widehat D_u=w^2(u)dD_u=w^2(u)d(u)du.
\end{equation}
Plugging this back to \eqref{eq:D-pre} we get \eqref{eq:dw}.

\vspace{0.1in}
\noindent
\emph{Departures affecting number in service/buffer}: From \eqref{eq:rho^n}, we can have a similar decomposition to \eqref{eq:Dn-dcom} for the scaled number in service: \begin{equation}\label{eq:Sn-dcom}
    \bar{S}^n_t=\bar{S}_t^{n,A}+\bar{S}_t^{n,D}=\bar{S}_t^{n,A}+Y^{S,n,D}_t+X^{S,n,D}_t,
\end{equation}
where
$$
X^{S,n,D}_t:=\frac{1}{n} \sum_{i\geq1}\mathbbm{1}_{\{\bar{Q}^n_{D_i-}>0\}}\mathbbm{1}_{\{D_i\leq t\}}\left[\mathbbm{1}_{\left\{V_{j_i}> t-D_i\right\}}-\bar{G}\left(t-D_i\right)\right],\quad Y_t^{S,n,D}:=\int_0^t\bar G(t-u)\,d\widehat D_u^n.
$$
Since
\[
\mathbbm{1}_{\{V_{j_i}>t-D_i\}}-\bar G(t-D_i)
=
-\left(
\mathbbm{1}_{\{V_{j_i}\le t-D_i\}}-G(t-D_i)
\right),
\]
we have $X^{S,n,D}=-X^{n,D}$. Therefore, \eqref{eq:XDto0} immediately gives
\begin{equation}\label{eq:XsDto0}
\lim_{n\to\infty}\sup_{t\in[0,T]}|X^{S,n,D}_t|=0\quad \text{almost surely}.
\end{equation}
For the term $Y^{S,n,D}_t$, using $\bar G=1-G$ and $\widehat D_0^n=0$, we have
$$
Y_t^{S,n,D}=\widehat D_t^n-Y_t^{n,D}.
$$
Similar to \eqref{eq:XDto0}, applying Lemma~\ref{lem:to-be-written} to $\widehat D^{r_k}$ and $\widehat D$, and then letting the grid mesh tend to zero, we obtain
\begin{equation}\label{eq:BktoB}
\lim_{k\to\infty}\sup_{t\in[0,T]} |\widehat D_t^{r_k}-\widehat D_t|=0.
\end{equation}
Therefore, \eqref{eq:YtoYD} and \eqref{eq:BktoB} imply
\begin{equation}\label{eq:YstoYsD}
\lim_{k\to\infty}\sup_{t\in[0,T]}
|Y_t^{S,r_k,D}-Y_t^{S,D}|=0,
\end{equation}
where, by \eqref{eq:muD-w2},
$$
Y_t^{S,D}:=\widehat D_t-Y_t^D=\int_{(0,t]}\bar G(t-u)\,d\widehat D_u=\int_0^t\bar G(t-u)w^2(u)d(u)\,du.
$$
Combining \eqref{eq:rhoA}, \eqref{eq:XsDto0}, and \eqref{eq:YstoYsD} with \eqref{eq:Sn-dcom}, we obtain uniform convergence of $\bar S^{r_k}$ to $\rho$ and the representation \eqref{eq:rhow}.
For the buffer, since $W^{q,D}_n$ carries no mark, $\bar{Q}^{n,D}=\widehat{D}^n$. Then \eqref{eq:S-Q-D}, \eqref{eq:QA} and \eqref{eq:BktoB} give $\bar{Q}^{r_k}\rightarrow\eta$ uniformly, with $\eta$ as in \eqref{eq:yw} by \eqref{eq:muD-w2}.

\vspace{0.1in}
\noindent
\emph{Part (ii).} The weak* convergence of $\mathbbm{1}_{\{\bar{S}^{r_k}_{u-}<1\}}$ has already been shown above in \eqref{eq:rhoweakstar}, and the counterpart of 
$\mathbbm{1}_{\{\bar{Q}^{r_k}_{u-}<b_{r_k}/{r_k}\}}$ is similar, along the same extracted subsequence. Uniform convergence of $\widehat D^{r_k}$ is given by \eqref{eq:BktoB}.

\vspace{0.1in}
\noindent
\emph{Part (iii).}
The functions $\rho, \eta, D$ have been identified in the proof of \emph{Part (i)}. It remains to show the constraints for functions $w^1,w^2,w^3$. We now observe that the set $\left\{\bar{S}_{u-}^n<1\right\}$ is identical to the set $\left\{\bar{S}_{u-}^n \leq 1-\frac{1}{n}\right\}$. This is because $\bar{S}^n$ only takes values in $\left\{\frac{i}{n}: i=0, 1, \ldots, n\right\}$. Therefore, \eqref{eq:rhonA} {can be rewritten as}
	\begin{equation*}
		\bar{S}^{n,A}_t=X_t^{s,n,A}+\int_{0}^{t} \mathbbm{1}_{\{\bar{S}^n_{u-}\leq1-\frac{1}{n}\}} \bar{G}(t-u) \lambda(u) d u.
	\end{equation*}
Next, we try to find out $w^1$. {Recall} the convergence stated in \eqref{eq:limitsub} in the uniform topology. Consequently fix $\varepsilon>0$ and choose $N(\om)$ large enough such that for all $k>N(\om)$ we have $r_k>\frac{3}{\varepsilon}$, $\left\|\bar{S}^{r_k}-\rho\right\|_T<\frac{\varepsilon}{3}$ {almost surely}. Then it is readily checked that
\begin{equation}\label{eq:leqchain}
	\mathbbm{1}_{\{\rho_{u-}\leq1-\varepsilon\}} \leq \mathbbm{1}_{\{\bar{S}^{r_k}_{u-}\leq 1-1/r_k\}} \leq \mathbbm{1}_{\{\rho_{u-}<1+\varepsilon\}}.
\end{equation} 
Therefore for any such that $\phi\in L^1[0, T]$ we have
	\begin{equation*}
		\int_0^t \phi(u) \mathbbm{1}_{\{\rho_{u-}\leq 1-\varepsilon\}} d u \leq \int_0^t \phi(u) \mathbbm{1}_{\{\bar{S}^{r_k}_{u-}\leq 1-1/r_k\}} d u \leq \int_0^t \phi(u) \mathbbm{1}_{\left\{\rho_{u-}<1+\varepsilon\right\}} d u, \quad \text{for all }t \in [0,T].
	\end{equation*}
	Note that $\lim _{\varepsilon \downarrow 0} \mathbbm{1}_{\{\rho_{u-}<1-\varepsilon\}}=\mathbbm{1}_{\{\rho_{u-}<1\}}$ and $\lim _{\varepsilon \downarrow 0} \mathbbm{1}_{\{\rho_{u-}<1+\varepsilon\}}=\mathbbm{1}_{\{\rho_{u-}\leq 1\}}=1$. Consequently taking $k \rightarrow \infty$ and then $\varepsilon \downarrow 0$ we have by the dominated convergence theorem and \eqref{eq:rhoweakstar} that:
	\begin{equation*}
		\int_0^t \phi(u) \mathbbm{1}_{\left\{\rho_{u-}<1\right\}} d u \leq \int_0^t w^1(u) \phi(u) d u \leq \int_0^t \phi(u)  d u, \quad \text{for all }t \in [0,T].
	\end{equation*}
	Since $\phi$ is arbitrary we have {almost surely}
	\begin{equation}\label{eq:w1leq}
		\mathbbm{1}_{\left\{\rho_{u-}<1\right\}} \leq w^1(u) \leq 1,\quad \text{a.e. in }[0,T].
	\end{equation}
Observe that one can also replace $\{\bar{Q}^{r_k}_{u-}<b_{r_k}/r_k\}$ by $\{\bar{Q}^{r_k}_{u-}\leq b_{r_k}/r_k-1/r_k\}$. Similar to \eqref{eq:leqchain}, for any $\varepsilon >0$ one can choose large enough $N(\om)$ such that for all $k>N(\om)$ we have $r_k>\frac{3}{\varepsilon}$, $\left\|\bar{Q}^{r_k}-\eta\right\|_T<\frac{\varepsilon}{3}$ and $|b_{r_k}/r_k-\beta|<\frac{\varepsilon}{3}$ {almost surely}. Then it is readily checked that {almost surely}
\begin{align}
	&\mathbbm{1}_{\{\eta_{u-}\leq\beta-\varepsilon\}} \leq \mathbbm{1}_{\{\bar{Q}^{r_k}_{u-}\leq b_{r_k}/r_k-1/r_k\}} \leq \mathbbm{1}_{\{\eta_{u-}<\beta+\varepsilon\}},\label{eq:y<1}
\end{align} 
Similar to the argument leading to \eqref{eq:w1leq}, from \eqref{eq:y<1} we can conclude that for any non-negative $\phi\in L^1[0,T]$
\begin{equation*}
		\int_0^t\mathbbm{1}_{\{\eta_{u-}<\beta\}}\phi(u)du \leq \int_0^t w^3(u) \phi(u) d u \leq \int_0^t \phi(u) d u,
	\end{equation*}
and almost surely
\begin{equation*}
		\mathbbm{1}_{\{\eta_{u-}<\beta\}}\leq w^3(u) \leq 1,\quad \text{a.e. in }[0,T].
	\end{equation*}
For the condition on $w^2$, let us work along the extracted subsequence $(r_k)$ and suppress its index. Since $\widehat D$ and $D$ are nondecreasing and uniformly bounded, the convergences in \emph{Part (i)} gives
$$
\sup_t|\bar{Q}^n-\eta| \rightarrow 0, \quad d \bar{D}^n \Rightarrow d D, \quad d\widehat D^n \Rightarrow w^2 d D.
$$
We know that $\eta$ is continuous and $0 \leq w^2 \leq 1, d D$-a.e. Let us define
$$
\nu_n(d u):=d \bar{D}^n_u-d\widehat D^n_u=\mathbbm{1}_{\left\{\bar{Q}_{u-}^n=0\right\}} d \bar{D}^n_u.
$$
Then
\begin{equation}\label{eq:nu-weak}
\nu_n \Rightarrow\left(1-w^2\right) d D.
\end{equation}
Since $\eta$ is continuous and non-negative, by the convergence of $\bar{Q}^n$ we have
\begin{equation*}
0\leq\int_0^T \eta_u\nu_n(du)=\int_0^T(\eta_u-\bar{Q}^n_{u-})\nu_n(du) +\int_0^T\bar{Q}^n_{u-}\mathbbm{1}_{\left\{\bar{Q}_{u-}^n=0\right\}} d \bar{D}^n_u\rightarrow 0.
\end{equation*}
On the other hand, passing the limit of the integral through \eqref{eq:nu-weak} we have
$$
\int_0^T \eta_u\nu_n(du)\rightarrow\int_0^T\eta_u(1-w^2(u))dD_u.
$$
Consequently we have $\int_0^T\eta_u(1-w^2(u))dD_u=0$. The integrand is nonnegative, so
$$
\eta_u(1-w^2(u))=0 \quad dD\text{-a.e.}
$$
Therefore,
$$
\mathbbm{1}_{\{\eta_{u}>0\}}\leq w^2(u)\leq 1 \quad d D \text {-a.e.}
$$
The Radon--Nikodym derivative $w^2$ is determined only $dD$-almost everywhere. Without loss of generality, we fix the version with $w^2:=1$ on $\{d=0\}$. Thus $\mathbbm{1}_{\{\eta_t>0\}}\leq w^2(t)\leq1$ holds for Lebesgue almost every $t$, as stated.

It remains to record the state constraints \eqref{eq:solbdn0}, on which the uniqueness argument of Theorem~\ref{thm:solVoln} relies. Each holds identically in the prelimit and passes to the uniform limit \eqref{eq:limitsub}. $0\leq\bar{S}^n_t\leq1$ gives $0\leq\rho_t\leq1$, and $0\leq\bar{Q}^n_t\leq b_n/n$ together with \eqref{eq:anbn} gives $0\leq\eta_t\leq\beta$. Since a customer waits only when every server is busy, $\bar{Q}^n_t\lp1-\bar{S}^n_t\rp=0$ for every $t$ and $n$, whence $\eta_t\lp1-\rho_t\rp=0$.

Therefore, by \eqref{eq:rhow}-\eqref{eq:dw} and Definition~\ref{def:sol_Vol} we conclude that $(\rho,\eta,d,w^1,w^2,w^3)$ is the solution to the discontinuous Volterra equation \eqref{eq:rhoVoln}.

\end{proof}

We have established a fluid limit for $(\bar{S}^n, \bar{Q}^n, \bar{D}^n)$ along a subsequence when the system starts empty. Now, we extend our considerations to a more general case.
\begin{assumption}\label{asm:initial}
	Let the conditions under Assumption~\ref{asm:queue} hold. In addition, let the number of customers in the system at time $0$: $N_0^n$, satisfy the following asymptotic result:
	
\begin{equation*}
		\lim _{n \rightarrow \infty}\left|\frac{N_0^n}{n}-r_0\right|=0, \quad \text{ almost surely, }
\end{equation*}
	where $r_0 \in[0,1+\beta]$. Moreover, assume that the remaining service times of each of the initially occupied servers are independently and identically distributed with the distribution $F^n$, independent of $N^n_0$, satisfying
	\begin{equation*}
		\lim_{n \rightarrow \infty} \sup_{t}\left|F^n(t)-F(t)\right|=0, 
	\end{equation*}
	for some distribution $F$ with density $f$. The customers initially in the buffer have i.i.d. service requirements with distribution $G$, independent of everything else.
\end{assumption}
\begin{proposition}\label{prop:weaklimitsinif}
	Let Assumption~\ref{asm:initial} hold. Then
\begin{enumerate}[label = (\roman*), wide, labelwidth = !, labelindent = 0pt]
\item For any $T > 0$, almost all $\om\in\Omega$ and for any subsequence of $(n)$, there exists a further subsequence $r_k$, possibly depending on $\om$ and real-valued continuous, possibly stochastic processes $\rho,\eta,D$ such that
	\begin{equation}\label{eq:limitsubninibuf}
\bar{S}^{r_k}_t\rightarrow\rho_t,\quad \bar{Q}^{r_k}_t\rightarrow\eta_t, \quad\bar{D}_t^{r_k}\rightarrow D_t,
	\end{equation} 
in the uniform topology.
\item {Moreover, given $(r_k)$, almost surely there exist bounded, possibly stochastic processes $w^{1}, w^{2}, w^{3}$ such that }
\begin{equation}\label{eq:rhoydweak*wninibuf}
\mathbbm{1}_{\{\bar{S}^{r_k}_{t-}<1\}}\stackrel{*}{\rightharpoonup}w^1(t), \quad 
\mathbbm{1}_{\{\bar{Q}^{r_k}_{t-}<{b_{r_k}}/{r_k}\}}\stackrel{*}{\rightharpoonup}w^3(t), \quad \text{in } L^\infty[0,T].
\end{equation}
and
\begin{equation*}
\widehat D_t^{r_k}:=\int_0^t\mathbbm{1}_{\{\bar Q^{r_k}_{u-}>0\}}\,d\bar D_u^{r_k}\to \int_0^tw^2(u)dD_u, \quad \text{uniformly in }[0,T].
\end{equation*}
\item {Furthermore, almost surely, $(\rho,\eta,D,w^1,w^2,w^3)$ defined in \eqref{eq:limitsubninibuf}-\eqref{eq:rhoydweak*wninibuf} satisfy}
\begin{align}
		&\rho_{t}=\min\{r_0,1\}\bar{F}(t)+\int_{0}^{t} w^{1}(u) \bar{G}(t-u) \lambda(u) d u  +\int_{0}^{t} w^{2}(u) \bar{G}(t-u) d(u) d u, \label{eq:rhow2}\\
		& \eta_{t}=\max\{r_0-1,0\}+\int_{0}^{t}(w^3(u)-w^1(u))\lambda(u) d u  -\int_{0}^{t} w^{2}(u) d(u) d u, \label{eq:yw2}\\
		& D_t=\min\{r_0,1\}F(t)+\int_{0}^{t}w^{1}(u) G(t-u) \lambda(u) d u  +\int_{0}^{t} w^{2}(u) G(t-u) d(u) d u,\label{eq:dw2}	
\end{align}
and for almost every $t\in[0,T]$
\begin{equation*}
\mathbbm{1}_{\{\rho_t<1\}} \leq w^1(t) \leq 1, \quad
\mathbbm{1}_{\left\{\eta_{t}>0\right\}} \leq w^2(t) \leq 1, \quad
\mathbbm{1}_{\{\eta_t<\beta\}}\leq w^3(t) \leq 1.
\end{equation*}
Moreover, as in \eqref{eq:solbdn0}, the limits obey the state constraints
\begin{equation}\label{eq:solbdnini}
0\leq\rho_t\leq1,\qquad 0\leq\eta_t\leq\beta,\qquad \eta_t\lp1-\rho_t\rp=0,\qquad t\in[0,T].
\end{equation}
That is, {for almost all $\om\in\Omega$} $(\rho{(\om)},\eta{(\om)},D,w^1{(\om)},w^2{(\om)},w^3{(\om)})$ as in \eqref{eq:rhow2}-\eqref{eq:dw2} is a solution, interpreted according to Definition~\ref{def:sol_Vol}, to the following non-linear discontinuous Volterra integral equation
\begin{align}\label{eq:rhoVoln-init}
		&\rho_{t}=\min\{r_0,1\}\bar{F}(t)+\int_{0}^{t} \mathbbm{1}_{\{\rho_{u-}<1\}} \bar{G}(t-u) \lambda(u) d u  +\int_{0}^{t} \mathbbm{1}_{\{\eta_{u-}>0\}} \bar{G}(t-u) d(u) d u, \nonumber \\
		& \eta_{t}=\max\{r_0-1,0\}+\int_{0}^{t} (\mathbbm{1}_{\{\eta_{u-}<\beta\}}-\mathbbm{1}_{\{\rho_{u-}<1\}} ) \lambda(u) d u  -\int_{0}^{t} \mathbbm{1}_{\{\eta_{u-}>0\}} d(u) d u, \nonumber \\
		& D_t=\min\{r_0,1\}F(t)+\int_{0}^{t} \mathbbm{1}_{\{\rho_{u-}<1\}} G(t-u) \lambda(u) d u  +\int_{0}^{t} \mathbbm{1}_{\{\eta_{u-}>0\}} G(t-u) d(u) d u.
	\end{align}
Here $D$ is absolutely continuous and has density $d$ satisfying
$$
d(t)=\min\{r_0,1\}f(t)+\int_0^tg(t-u)w^1(u)\lambda(u)du+\int_0^tg(t-u)w^2(u)d(u)du \quad \text{a.e.}
$$
\end{enumerate}
\end{proposition}
\begin{proof}
When $r_0=0$, this is trivial: the three initial terms below are bounded by $N_0^n/n \to 0$. We consider $r_0>0$. At time $0$, the number of customers in service is $\min\{N_0^n,n\}$ and the number of customers in buffer is $\max\{N_0^n-n,0\}$. Let the remaining service times for the customers in service to be $\left(V_i^{0,n}\right)_{1 \leq i \leq \min\{N_0^n,n\}}$. Then, similar to \eqref{eq:rho^n}-\eqref{eq:D^n} we have:
\begin{align*}
	&\bar{S}^n_t=\frac{1}{n}\sum_{i=1}^{\min\{N_0^n,n\}}\mathbbm{1}_{\{V_i^{0,n}>t\}}+\int_0^t\int_\R W_n^{s,A}(t,u,x) \mathcal{M}^{n,A} (du,dx)+\int_0^t\int_\R W_n^{s,D}(t,u,x) \mathcal{M}^{n,D} (du,dx),\\
	&\bar{Q}^n_t=\frac{1}{n}\max\{N_0^n-n,0\}+\int_0^t\int_\R W_n^{q,A}(t,u,x) \mathcal{M}^{n,A} (du,dx)-\int_0^t\int_\R W_n^{q,D}(t,u,x) \mathcal{M}^{n,D} (du,dx),\\
	&\bar{D}_{t}^{n}=\frac{1}{n}\sum_{i=1}^{\min\{N_0^n,n\}}\mathbbm{1}_{\{V_i^{0,n}\leq t\}}+\int_0^t\int_\R W_n^{d, A}(t,u,x) \mathcal{M}^{n,A} (du,dx)+\int_0^t\int_\R W_n^{d,D}(t,u,x) \mathcal{M}^{n,D} (du,dx).
\end{align*}	
Observing that 
\begin{equation*}
	\frac{1}{n}\sum_{i=1}^{\min\{N_0^n,n\}}\mathbbm{1}_{\{V_i^{0,n}>t\}}=\frac{\min\{N_0^n,n\}}{n}\sum_{i=1}^{\min\{N_0^n,n\}}\frac{\mathbbm{1}_{\{V_i^{0,n}>t\}}}{\min\{N_0^n,n\}}.
\end{equation*}
Similar to the argument leading to \eqref{eq:iniD0}, by Assumption~\ref{asm:initial} and \eqref{eq:anbn} we have that
\begin{equation}\label{eq:inirho}
	\lim_{n \rightarrow \infty} \sup_t \left|\frac{1}{n}\sum_{i=1}^{\min\{N_0^n,n\}}\mathbbm{1}_{\{V_i^{0,n}>t\}}- \min\{r_0,1\}\bar{F}(t)\right|=0, \quad \text{almost surely.}
\end{equation}
And
\begin{equation}\label{eq:iniD}
	\lim_{n \rightarrow \infty} \sup_t \left|\frac{1}{n}\sum_{i=1}^{\min\{N_0^n,n\}}\mathbbm{1}_{\{V_i^{0,n}\leq t\}}- \min\{r_0,1\}F(t)\right|=0, \quad \text{almost surely.}
\end{equation}
Moreover,
\begin{equation}\label{eq:iniy}
	\lim_{n \rightarrow \infty}\left|\frac{1}{n}\max\{N_0^n-n,0\}-\max\{r_0-1,0\} \right|=0 , \quad \text{almost surely.}
\end{equation}
By a similar analysis of integration with respect to $\mathcal{M}^{n,A}$ and $\mathcal{M}^{n,D}$ in Proposition~\ref{prop:weaklimits}, combining with \eqref{eq:inirho}-\eqref{eq:iniy} we get the desired results.
\end{proof}
Now, we establish the existence of a unique $(\rho, \eta, D)$ that satisfies \eqref{eq:rhoVoln-init} in the sense of Definition~\ref{def:sol_Vol}.
Consequently, we obtain a unique fluid limit of the fraction of busy servers, fraction of occupied buffers and the $n-$scaled cumulative departure rate.

\begin{theorem}\label{thm:solVoln}
	Let Assumption~\ref{asm:initial} hold. Then there exists a unique solution to the discontinuous Volterra integral equation \eqref{eq:rhoVoln-init}, in the following sense: there exist unique functions $(\rho,\eta,d)$ such that for $t\in [0,T]$
	\begin{align}
&\rho_{t}=\min\{r_0,1\}\bar{F}(t)+\int_{0}^{t} z^1(u) \bar{G}(t-u) \lambda(u) d u  +\int_{0}^{t} z^2(u) \bar{G}(t-u) d(u) d u, \label{eq:rhown}\\
		& \eta_{t}=\max\{r_0-1,0\}+\int_{0}^{t} (z^3(u)-z^1(u)) \lambda(u) d u  -\int_{0}^{t} z^2(u) d(u) d u, \label{eq:ywn}\\
		& D_t=\min\{r_0,1\}F(t)+\int_{0}^{t}z^1(u) G(t-u) \lambda(u) d u  +\int_{0}^{t} z^2(u) G(t-u) d(u) d u,\label{eq:dwn}
	\end{align}
{for some} $(z^1,z^2,z^3)$ that satisfies for almost every $t\in[0,T]$
\begin{align}
&\mathbbm{1}_{\{\rho_t<1\}} \leq z^1(t) \leq 1,\quad  \mathbbm{1}_{\left\{\eta_{t}>0\right\}} \leq z^2(t) \leq 1,\quad \mathbbm{1}_{\{\eta_t<\beta\}}\leq z^3(t) \leq 1, \quad and \label{eq:w123tbd}\\
&0 \leq \rho_t\leq 1,\quad 0 \leq \eta_t\leq \beta,\quad \eta_t(1-\rho_t)=0. \label{eq:solbdn}
\end{align}
In addition, the functions $z^1\la+z^2d$ and $z^3\la$ are unique almost everywhere.
\end{theorem}
\begin{proof}
The existence of the solution follows directly from Proposition~\ref{prop:weaklimitsinif}. It remains to show the uniqueness. By adding \eqref{eq:rhown}-\eqref{eq:dwn} we get
\begin{equation}\label{eq:rho+eta+D}
    \rho_t+\eta_t+D_t=r_0+\int_0^t z^3(u)\la(u)du.
\end{equation}
Similarly, adding \eqref{eq:rhown} and \eqref{eq:dwn} we get
\begin{equation}\label{eq:rho+D}
    \rho_t+D_t=\min\{r_0,1\}+\int_0^t \left(z^1(u)\la(u)+z^2(u)d(u)\right)du.
\end{equation}
Suppose there are two solutions $(\rho^i,\eta^i,D^i,z^1_i,z^2_i,z^3_i)$, $i=1,2$. For $i=1,2$, let us define
$$
X^i:=\rho^i+\eta^i,
$$
the system external acceptance rate
$$
p_i(t):=z^3_i(t)\la(t),
$$
and the service acceptance rate
$$
q_i(t):=z^1_i(t)\la(t)+z^2_i(t)d_i(t).
$$
Let us denote
$$
Y:=X^1-X^2,\quad U:=\rho^1-\rho^2,\quad O:=D^1-D^2.
$$
Since $\lambda,d_i$ are integrable and $z^1_i,z^2_i,z^3_i$ are bounded, $p_i,q_i$ are integrable. Thus we can define
$$
P(t):=\int_0^t(p_1(u)-p_2(u))du,\quad Q(t):=\int_0^t(q_1(u)-q_2(u))du.
$$
By Lebesgue differentiation theorem (see \cite[Chap.\ 3, Thm.\ 3.11]{Stein2005}) we know that $P,Q$ are absolute continuous and $P^{\prime}(t)=p_1(t)-p_2(t),\,Q^{\prime}(t)=q_1(t)-q_2(t)$ almost everywhere. From \eqref{eq:w123tbd}-\eqref{eq:solbdn} we know that
$$
\rho^i=X^i \wedge 1, \quad \eta^i=\left(X^i-1\right)^{+}, \quad 0 \leq X^i \leq 1+\beta ,
$$
and
\begin{equation}\label{eq:p=la}
p_i(t)=\lambda(t) \quad \text {for a.e. }t\in\left\{X^i<1+\beta\right\},
\end{equation}
where the last identity follows since $\eta^i = (X^i - 1)^+ < \beta$ and hence $z_i^3 =1$ almost everywhere on that set by \eqref{eq:w123tbd}. Furthermore, the bound $z^3_i\leq1$ gives $p_i\leq\la$ everywhere.
Therefore, if $Y(t)>0$, then $X^2(t)<X^1(t)\leq 1+\beta$, and hence $X^2(t)<1+\beta$. From \eqref{eq:p=la} this forces $p_2(t)=\lambda(t)$. Since by construction $p_1(t)\leq \lambda(t)$, we have
$$
P^{\prime}(t)=p_1(t)-p_2(t)\leq0.
$$
Similarly, $Y(t)<0$ implies $P^{\prime}(t)\geq0$. Consequently
\begin{equation}\label{eq:Y-P'}
    Y(t)P^{\prime}(t)\leq 0 \quad \text{a.e.}
\end{equation}
From \eqref{eq:rho+eta+D} we have
\begin{equation}\label{eq:Y=P-V}
Y=P-O.
\end{equation}
From \eqref{eq:rho+D} we have
\begin{equation}\label{eq:U=Q-V}
    U=Q-O.
\end{equation}
Notice from \eqref{eq:dwn} that
$$
O(t)=\int_0^t\left(q_1(u)-q_2(u)\right)G(t-u)du.
$$
Since $G(0)=0$, by integration by part we have
\begin{equation*}
    O(t)=\int_0^t Q(u)g(t-u)du.
\end{equation*}
Since $g$ is integrable on $[0,T]$, we can choose $\delta>0$ such that
$$c_\delta:=\int_0^\delta g(s)ds<\frac{1}{3}.$$
Denote $M:=\sup_{t\in[0,\delta]}|Q(t)|$, then for $t\in[0,\delta]$
\begin{equation}\label{eq:V-leq-cM}
    \left|O(t)\right|=\left|\int_0^t Q(u)g(t-u)du\right|\leq M\int_0^\delta g(s)ds=c_\delta M.
\end{equation}
For $t\in[0,\delta]$, using \eqref{eq:Y-P'},\eqref{eq:Y=P-V} and \eqref{eq:V-leq-cM}, by a similar argument leading to \eqref{eq:K-leq-cM} we can conclude that
\begin{equation*}
    \left|P(t)\right|\leq c_\delta M.
\end{equation*}
Hence, by \eqref{eq:Y=P-V} we have
\begin{equation*}
    \left|Y(t)\right|\leq\left|P(t)\right|+\left|O(t)\right|\leq 2 c_\delta M.
\end{equation*}
Moreover, since $\rho^i=X^i \wedge 1$, we have
\begin{equation}\label{eq:U-leq-cM}
    \left|U(t)\right|\leq\left|Y(t)\right|\leq 2 c_\delta M.
\end{equation}
Using \eqref{eq:V-leq-cM} and \eqref{eq:U-leq-cM} in \eqref{eq:U=Q-V}, we have
$$
\left|Q(t)\right|\leq\left|U(t)\right|+\left|O(t)\right|\leq 3 c_\delta M.
$$
Taking supremum over $t\in[0,\delta]$ we get $M\leq 3c_\delta M$. However, $c_\delta<\frac{1}{3}$. Consequently, $M=0$, and hence
$$
Y=P=Q=U=O=0, \quad \text{for } t\in[0,\delta].
$$
Similar to the proof of Theorem~\ref{thm:solVol}, by iterating this argument we can conclude that
$$
Y=P=Q=U=O=0, \quad \text{for } t\in[0,T].
$$
From $U=0$ and $O=0$ we obtain $\rho^1=\rho^2$ and $D^1=D^2$ on $[0,T]$. Since $\eta^i=(X^i-1)^+$, $Y=0$ yields $\eta^1=\eta^2$. Differentiating $D^1=D^2$ gives $d_1=d_2=:d$ almost everywhere. It remains to identify the integrands. Since $P\equiv Q\equiv0$, we have $P^{\prime}=Q^{\prime}=0$ almost everywhere. Recalling $P^{\prime}=p_1-p_2$ and $Q^{\prime}=q_1-q_2$ almost everywhere, we conclude
$$
z^3_1\la=z^3_2\la,\qquad z^1_1\la+z^2_1d=z^1_2\la+z^2_2d,\quad\text{a.e. in }[0,T].
$$
\end{proof}
Similar to Theorem~\ref{thm:solVolw}, we now provide possible solutions to the auxiliary functions $(z^1, z^2, z^3)$.
\begin{theorem}\label{thm:solVolnw}
	Under the setting of Theorem~\ref{thm:solVoln}, the solution to \eqref{eq:rhown}-\eqref{eq:solbdn} satisfies
\begin{align}\label{eq:z123}
	\left\{\begin{array}{ll}
		z^1(t)=1,z^2(t)d(t)=0,z^3(t)=1,&\rho_t<1,\eta_t=0\\
		z^1(t)\la(t)+z^2(t)d(t)=\la(t),z^3(t)=1,&\rho_t=1,\eta_t=0\\
z^1(t)\la(t)=0,z^2(t)=1,z^3(t)=1,&\rho_t=1,0<\eta_t<\beta\\
z^1(t)\la(t)=0,z^2(t)=1,z^3(t)\la(t)=d(t),\,&\rho_t=1,\eta_t=\beta
	\end{array}
	\right. \quad \text{a.e. in } [0,T],
\end{align}
and
\begin{equation}\label{eq:d-sol}
   d(t)=\min\{r_0,1\}f(t)+\int_{0}^{t}\left(z^1(u)\lambda(u)+z^2(u) d(u)\right) g(t-u)  d u , \quad \text{a.e. in }[0,T]. 
\end{equation}
In particular, the tuple $(\rho,\eta,d,z^1,z^2,z^3)$ satisfying \eqref{eq:rhown}-\eqref{eq:dwn} and
\begin{align}\label{eq:z123-sp}
	\left\{\begin{array}{ll}
		z^1(t)=1,z^2(t)=0,z^3(t)=1,&\rho_t<1,\eta_t=0\\
		z^1(t)=1,z^2(t)=0,z^3(t)=1,&\rho_t=1,\eta_t=0\\
z^1(t)=0,z^2(t)=1,z^3(t)=1,&\rho_t=1,0<\eta_t<\beta\\
z^1(t)=0,z^2(t)=1,z^3(t)=d(t)/\lambda(t)\wedge1,\,&\rho_t=1,\eta_t=\beta
	\end{array}
	\right.
\end{align}
and \eqref{eq:d-sol} is a solution to the system \eqref{eq:rhown}-\eqref{eq:solbdn}.
Here we choose that in State 4 of \eqref{eq:z123-sp}, $z^3(t):=1$ whenever $\la(t)=0$. In addition, note that \eqref{eq:z123} leaves $z^1$ and $z^2$ individually undetermined in State 2, and $z^3$ undetermined where $\la = 0$. The particular choice in \eqref{eq:z123-sp} satisfies \eqref{eq:w123tbd} and provides the same $z^1\la + z^2 d$, $z^3 \la$ and therefore by Theorem~\ref{thm:solVoln} the same $(\rho, \eta, D)$.
\end{theorem}
\begin{proof}
By Theorem~\ref{thm:solVoln} there exist a unique solution $(\rho,\eta,D)$ that solves \eqref{eq:rhown}-\eqref{eq:solbdn} and $z^1\la+z^2d$, $z^3 \la$ are  unique almost surely. We will find out $(z^1,z^2,z^3)$. Since $\int_0^td(u)du\leq\int_0^t\lambda(u)du\leq\infty$ for $t\in[0,T]$, $d\in L^1[0,T]$. For bounded $x(t)$, since $d(t),\lambda(t),g(t) \in L^1[0,T]$, by Young's convolution inequality we have 
\begin{align*}
\frac{\partial}{\partial t} x(u)G(t-u)\lambda(u)=x(u)\lambda(u)g(t-u)\in L^1([0,T]\times[0,T]),\\
\frac{\partial}{\partial t} x(u)G(t-u)d(u)=x(u)d(u)g(t-u)\in L^1([0,T]\times[0,T]).
\end{align*}
Therefore, by \cite[Thm.\ 2.7]{vsremr2012differentiation} we can differentiate both sides of \eqref{eq:dwn} for $t\in(0,T)$ 
\begin{equation}\label{eq:dwdiff}
d(t)=\min\{r_0,1\}f(t)+\int_{0}^{t}\left(z^1(u)\lambda(u)+z^2(u) d(u)\right) g(t-u)  d u , \quad \text{a.e. in }[0,T].
\end{equation}
Differentiating both sides of \eqref{eq:rhown} and plugging in \eqref{eq:dwdiff} we obtain 
\begin{align}\label{eq:rhowdiff}
\rho_t^{\prime}&=-\min\{r_0,1\}f(t)+\left(z^1(t)\lambda(t)+z^2(t)d(t) \right)-\int_{0}^{t}\left( z^1(u)\lambda(u) +z^2(u)d(u)\right) g(t-u)  d u \nonumber\\
&=\left(z^1(t)\lambda(t)+z^2(t)d(t) \right)-d(t), \quad \text{a.e. in }[0,T].
\end{align}
Differentiating both sides of \eqref{eq:ywn} we obtain
\begin{equation}\label{eq:etawdiff}
    \eta_t^{\prime}=(z^3(t)-z^1(t))\la(t)-z^2(t)d(t),\quad \text{a.e. in }[0,T].
\end{equation}
By \eqref{eq:solbdn} we know that when $\eta_t>0$, $\rho_t=1$. Hence, we can divide the situations into four states: 
\begin{enumerate}
\item $\rho_t<1,\eta_t=0$. 
\item $\rho_t=1,\eta_t=0$. 
\item $\rho_t=1, 0<\eta_t<\beta$.
\item $\rho_t=1,\eta_t=\beta$.
\end{enumerate}
Similar to the argument leading to \eqref{eq:A=cM}, we can conclude that 
\begin{equation}\label{eq:rho'=0}
\rho_t^{\prime}=0 \quad \text{for a.e. } t\in\{\rho_t=1\},
\end{equation}
and
\begin{equation}\label{eq:eta'=0}
\eta_t^{\prime}=0\quad \text{for a.e. } t\in\{\eta_t=0\}\cup\{\eta_t=\beta\}.
\end{equation}
In the following we will omit almost everywhere for convenience.

\noindent
\emph{State 1}: $\rho_t<1,\eta_t=0$. By \eqref{eq:w123tbd} we have $z^1=1,\,z^3=1$. Plugging this and \eqref{eq:eta'=0} to \eqref{eq:etawdiff} we have
$$z^2(t)d(t)=0 .$$
\noindent
\emph{State 2}: $\rho_t=1,\eta_t=0$. By \eqref{eq:w123tbd} we have $z^3=1$. Using \eqref{eq:rho'=0} and \eqref{eq:eta'=0} in \eqref{eq:rhowdiff} and \eqref{eq:etawdiff} respectively, we get
\begin{equation}\label{eq:z1la+z2d}
    z^1(t)\la(t)+z^2(t)d(t)=d(t) ,
\end{equation}
and
\begin{equation}\label{eq:z2d}
    (1-z^1(t))\lambda(t)=z^2(t)d(t) .
\end{equation}
Plugging \eqref{eq:z2d} in \eqref{eq:z1la+z2d} we obtain $\la(t)=d(t)$. Plugging this back to \eqref{eq:z1la+z2d} we have
$$
z^1(t)\la(t)+z^2(t)d(t)=\la(t).
$$
\noindent
\emph{State 3}: $\rho_t=1,0<\eta_t<\beta$. By \eqref{eq:w123tbd} we have $z^2=1,z^3=1$. Using \eqref{eq:rho'=0} in \eqref{eq:rhowdiff} we have
$$z^1(t)\la(t)+z^2(t)d(t)=d(t) .$$
However, since $z^2=1$, the above equation implies $z^1\la=0$.

\noindent
\emph{State 4}: $\rho_t=1,\eta_t=\beta$. By \eqref{eq:w123tbd} we have $z^2=1$. And again by \eqref{eq:rho'=0} and \eqref{eq:rhowdiff} we have $z^1\la=0$. Using \eqref{eq:eta'=0} in \eqref{eq:etawdiff} we have
$$
(z^3(t)-z^1(t))\lambda(t)=z^2(t)d(t) .
$$
Since $z^2=1$ and $z^1\la=0$, combining with the above equation we have
$$
z^3(t)\la(t)=d(t).
$$
This completes the proof.
\end{proof}
\begin{remark}
Recall from \eqref{eq:rho+eta+D} that the system external acceptance rate
\begin{equation}\label{eq:p=z3*la}
p(t)=z^3(t)\lambda(t).
\end{equation}
Similar to Remark~\ref{rk:exit-time}, from \eqref{eq:rhowdiff}-\eqref{eq:etawdiff} and \eqref{eq:p=z3*la}, the entire system has a flow-balance relation:
$$
\rho_t'+\eta_t'=z^3(t)\la(t)-d(t),\quad \text{a.e. in }[0,T].
$$
The system can remain at {\it State 4} only if departures can be replaced by arrivals. Thus,
$$
z^3(t)\la(t)=d(t)\leq\lambda(t).
$$
Consequently, the system leaves capacity when the arrivals are insufficient to replace the departures. More precisely, a point $t$ in {\it State 4} is a right exit point if and only if, in every right neighborhood of $t$, $d(s)>\lambda(s)$ on a set of positive Lebesgue measure.
\end{remark}

\begin{theorem}\label{thm:fluidlimitbuf}
	Let Assumption~\ref{asm:initial} hold. Then
\begin{enumerate}[label = (\roman*), wide, labelwidth = !, labelindent = 0pt]
\item For any $T > 0$, there exist real-valued continuous deterministic processes $\rho,\eta,D$ such that almost surely
	\begin{equation}\label{eq:limitunif}
		\lim_{n\rightarrow\infty}\sup_{t\in[0,T]}\left|\bar{S}^n_t-\rho_t\right|=0,\quad \lim_{n\rightarrow\infty}\sup_{t\in[0,T]}\left|\bar{Q}^n_t-\eta_t\right|=0,\quad \lim_{n\rightarrow\infty}\sup_{t\in[0,T]}\left|\bar{D}_t^{n}-D_t\right|=0.
	\end{equation} 
\item 
{Moreover, there exist bounded functions $w^{1}, w^{2}, w^{3}$ such that almost surely}
\begin{equation}\label{eq:rhoydweak*wbuf}
\mathbbm{1}_{\{\bar{Q}^{n}_{t-}<{b_{n}}/{n}\}}\la(t)\rightharpoonup w^3(t)\la(t), \quad \text{in } L^1[0,T],
\end{equation}
where $w^1,w^2,w^3$ satisfy \eqref{eq:z123}. 
\item {Furthermore, $(\rho,\eta,D,w^1,w^2,w^3)$ defined in \eqref{eq:limitunif}-\eqref{eq:rhoydweak*wbuf} satisfy}
		\begin{align}
		&\rho_{t}=\min\{r_0,1\}\bar{F}(t)+\int_{0}^{t} w^{1}(u) \bar{G}(t-u) \lambda(u) d u  +\int_{0}^{t} w^{2}(u) \bar{G}(t-u) d(u) d u, \label{eq:rhow3}\\
		& \eta_{t}=\max\{r_0-1,0\}+\int_{0}^{t} (w^3(u)-w^1(u))\lambda(u) d u  -\int_{0}^{t} w^{2}(u) d(u) d u, \label{eq:yw3}\\
		& D_t=\min\{r_0,1\}F(t)+\int_{0}^{t}w^{1}(u) G(t-u) \lambda(u) d u  +\int_{0}^{t} w^{2}(u) G(t-u) d(u) d u,\label{eq:dw3}	
\end{align}
and for almost every $t\in[0,T]$
\begin{equation*}
\mathbbm{1}_{\{\rho_t<1\}} \leq w^1(t) \leq 1,\quad \mathbbm{1}_{\left\{\eta_{t}>0\right\}} \leq w^2(t) \leq 1,\quad\mathbbm{1}_{\{\eta_t<\beta\}}\leq w^3(t) \leq 1.
\end{equation*}
Moreover, $(\rho,\eta)$ obeys the state constraints \eqref{eq:solbdnini}.
That is, $(\rho,\eta,d,w^1,w^2,w^3)$ as in \eqref{eq:rhow3}-\eqref{eq:dw3} is a solution, interpreted according to Definition~\ref{def:sol_Vol}, to the following non-linear discontinuous Volterra integral equation
\begin{align}
		&\rho_{t}=\min\{r_0,1\}\bar{F}(t)+\int_{0}^{t} \mathbbm{1}_{\{\rho_{u-}<1\}} \bar{G}(t-u) \lambda(u) d u  +\int_{0}^{t} \mathbbm{1}_{\{\eta_{u-}>0\}} \bar{G}(t-u) d(u) d u, \label{eq:rhoVolnbuf}\\
		& \eta_{t}=\max\{r_0-1,0\}+\int_{0}^{t} (\mathbbm{1}_{\{\eta_{u-}<\beta\}} - \mathbbm{1}_{\{\rho_{u-}<1\}}) \lambda(u) d u  -\int_{0}^{t} \mathbbm{1}_{\{\eta_{u-}>0\}} d(u) d u, \label{eq:etaVolnbuf}\\
		& D_t=\min\{r_0,1\}F(t)+\int_{0}^{t} \mathbbm{1}_{\{\rho_{u-}<1\}} G(t-u) \lambda(u) d u  +\int_{0}^{t} \mathbbm{1}_{\{\eta_{u-}>0\}} G(t-u) d(u) d u.\label{eq:DVolnbuf}
	\end{align}
Here $D$ is absolutely continuous and has density $d$ satisfying
$$
d(t)=\min\{r_0,1\}f(t)+\int_0^tg(t-u)w^1(u)\lambda(u)du+\int_0^tg(t-u)w^2(u)d(u)du \quad \text{a.e.}
$$
\end{enumerate}
\end{theorem}
\begin{proof}
\emph{Part (i)}. 
Define the sequence  
$$a_n^b(\om) := \sup_{t \in [0,T]} \left|\bar{S}_t^n(\om) - \rho_t\right|+\sup_{t\in[0,T]}\left|\bar{Q}^n_t(\om)-\eta_t\right|+\sup_{t\in[0,T]}\left|\bar{D}_t^{n}(\om)-D_t\right|.$$
By Theorem~\ref{thm:solVoln}, $\rho,\eta,d$ are unique. Consequently $\rho,\eta,d$ are deterministic functions, which are the same for all $\om$. From Proposition~\ref{prop:weaklimitsinif}, for almost all $\om$ and every subsequence of $a_n^b(\om)$ there exists a further subsequence $(r_k)$ possibly depending on $\om$, such that $a_{r_k}^b(\om)$ converges to $0$.
Hence $a_n^b \to 0$ for almost all $\om$. This proves \eqref{eq:limitunif}.

\vspace{0.1in}
\noindent
\emph{Part (ii)}. By Proposition~\ref{prop:weaklimitsinif} we have for every subsequence there exists a subsubsequence $(r_k)$ and bounded, possibly stochastic process $w^3$ such that almost surely
\begin{equation*}
\mathbbm{1}_{\{\bar{Q}^{r_k}_{t-}<{b_{r_k}}/{r_k}\}}\stackrel{*}{\rightharpoonup}w^3(t), \quad \text{in } L^\infty[0,T].
\end{equation*}
Testing against $\phi\la\in L^1[0,T]$ for $\phi\in L^\infty[0,T]$ gives
\begin{equation}\label{eq:rydwk*wninibf}
\mathbbm{1}_{\{\bar{Q}^{r_k}_{t-}<{b_{r_k}}/{r_k}\}}\la(t)\rightharpoonup w^3(t)\la(t), \quad \text{in } L^1[0,T].
\end{equation}
By Theorem~\ref{thm:solVoln} we know that $w^3\la$ is unique. Therefore the weak convergence in \eqref{eq:rydwk*wninibf} holds for the entire sequence, by the subsequence argument in the proof of Theorem~\ref{thm:fluidlimitini0uni}.

\vspace{0.1in}
\noindent
\emph{Part (iii)}. This follows directly from Proposition~\ref{prop:weaklimitsinif}.(iii) and Theorem~\ref{thm:solVoln}-\ref{thm:solVolnw}.

\end{proof}

Similar to Corollary~\ref{cor:acceptprob}-\ref{cor:accum-block}, an asymptotic result of the acceptance probability can be obtained. We state the following results without proof.

\begin{corollary}\label{cor:acceptprobn}
The acceptance probability in the $n$-th $M_t/G/n/n+b_n$ model $\mathbb{P}(\bar{Q}^n_{t-}<\frac{b_n}{n})$ satisfies 
\begin{equation*}
\mathbb{P}\left(\bar{Q}^n_{t-}<\frac{b_n}{n}\right) \la(t)\rightharpoonup  w^3(t) \la(t),\quad  \text{in }L^1[0,T],
\end{equation*}
where $w^3$ is defined in Theorem~\ref{thm:fluidlimitbuf}. Moreover, 
$$
\mathbb{P}\left(\bar{Q}^n_{t-}<\frac{b_n}{n}\right) \to 1
$$
locally uniformly on $\{t\in[0,T]:\eta_t<\beta\}$.
\end{corollary}
\begin{proof}
    The second statement follows as in Corollary~\ref{cor:acceptprob}. Setting $\ep_K:=\min_{t\in K}(\beta-\eta_t)>0$, \eqref{eq:anbn} gives $b_n/n-\eta_t\geq\ep_K/2$ on $K$ for all large $n$. Therefore $\{\bar{Q}^n_{t-}\geq b_n/n\}\subseteq\{\sup_s\lln\bar{Q}^n_s-\eta_s\rrn\geq\ep_K/2\}$ and \eqref{eq:limitunif} applies.
\end{proof}

\begin{corollary}\label{cor:accum-block-n}
Let Assumption~\ref{asm:initial} hold and $w^3$ be as in Theorem~\ref{thm:fluidlimitbuf}.
Let 
$$
\bar{L}_t^{n,b} := \frac{1}{n} \sum_{i=1}^{\infty} \mathbbm{1}_{\{T_i \leq t\}} \mathbbm{1}_{\{\bar{Q}_{T_i-}^n = \frac{b_n}{n}\}}.
$$
Then we have
$$
\lim_{n \to \infty} \sup_{t \in [0,T]} \left| \bar{L}_t^{n,b} - \int_0^t (1-w^3(u))\la(u) du \right| = 0, \quad \text{almost surely.}
$$
\end{corollary}

\begin{remark}\label{rk:w3discontinuous}
Similar to Remark~\ref{rk:wdiscontinuous}, the function $w^3$ can be discontinuous even when $\lambda$ is continuous.
\end{remark}

\section{Numerics and Operational Perspectives}\label{sec:numerics}

\subsection{{Numerical Methods for Discontinuous VIE}} 

This section outlines a simple procedure to numerically solve the discontinuous VIEs \eqref{eq:rhoVolinin} and \eqref{eq:rhoVolnbuf}-\eqref{eq:DVolnbuf}, using an explicit Euler discretization. The main computational challenge lies in updating the auxiliary functions $z$ or {$(z^1,z^2,z^3)$} in tandem with the solution trajectories $\rho$ or $(\rho,\eta,d)$ at each iteration. Algorithm~\ref{alg:1} details the steps to solve \eqref{eq:rhoVolinin} following the solution framework described in Theorems~\ref{thm:solVol}-\ref{thm:solVolw}, while Algorithm~\ref{alg:2} extends this to the coupled system \eqref{eq:rhoVolnbuf}-\eqref{eq:DVolnbuf} using Theorems~\ref{thm:solVoln} - \ref{thm:solVolnw}. In the following we set: $0=t_0<t_1<\cdots<t_N=T$ and $\Delta_j:=t_j-t_{j-1}$ for $j=1,\ldots,N$.

\begin{breakablealgorithm}
\caption{VIE for Zero-Buffer Loss System} 
\label{alg:1}
\begin{algorithmic}[1] 
\State \textbf{Input:} Initial value $\rho_{t_0}$, functions $f, g, \bar{F}, \bar{G}, \lambda$, error tolerance $\ep$. 
\For{$i = 0$ to $N-1$} 
\Statex \Comment{Determine $z$ values for time $t_{i}$}
\If{$\rho_{t_i} < 1-\ep$} 
\State $z_{t_{i}} \gets 1$ 
\ElsIf{$\rho_{t_i} \geq 1-\ep$} 
\State $z_{t_{i}} \gets \min\left(\frac{1}{\lambda(t_i)}\left(\rho_0 f(t_i) + \sum_{j=0}^{i-1} z_{t_j}\lambda(t_j)g(t_{i} - t_j)\Delta_{j+1}\right), 1\right)$ 
\EndIf 
\Statex \Comment{Update state for time $t_{i+1}$ by discretizing integral equations}
\State $\rho_{t_{i+1}} \gets \rho_0\bar{F}(t_{i+1}) + \sum_{j=0}^{i} z_{t_j}\lambda(t_j)\bar{G}(t_{i+1}-t_j)\Delta_{j+1}$ 
\EndFor 
\State \textbf{Output:} The sequence of values for $\rho$ and $z$. 
\end{algorithmic} 
\end{breakablealgorithm} 

\begin{breakablealgorithm}
\caption{VIE for Loss System with Buffer} 
\label{alg:2} 
\begin{algorithmic}[1] 
\State \textbf{Input:} Initial value $r_0$ or $(\rho_{t_0}, \eta_{t_0})$, threshold $\beta$, functions $f, g, \bar{F}, \bar{G}, \lambda$, error tolerance $\ep$. 
\State \textbf{Initialization:} Set $d(t_0)=\rho_{t_0}f(0)$. 
\For{$i = 0$ to $N-1$} 
\Statex \Comment{Determine $z$ values for time $t_{i}$} 
\If{$\rho_{t_i} < 1-\ep$ and $\eta_{t_i} < \ep$} 
\State $(z^1_{t_{i}}, z^2_{t_{i}}, z^3_{t_{i}}) \gets (1, 0, 1)$ 
\ElsIf{$\rho_{t_i} \geq 1-\ep$ and $\eta_{t_i} < \ep$} 
\State $(z^1_{t_{i}}, z^2_{t_{i}}, z^3_{t_{i}}) \gets (1, 0, 1)$ 
\ElsIf{$\rho_{t_i} \geq 1-\ep$ and $\ep \leq \eta_{t_i} < \beta-\ep$} 
\State $(z^1_{t_{i}}, z^2_{t_{i}}, z^3_{t_{i}}) \gets (0, 1, 1)$ 
\ElsIf{$\rho_{t_i} \geq 1-\ep$ and $\eta_{t_i} \geq \beta-\ep$} 
\State $(z^1_{t_{i}}, z^2_{t_{i}}) \gets (0, 1)$ 
\State $z^3_{t_{i}} \gets \min(d(t_i) / \lambda(t_i), 1)$ 
\EndIf \Statex \Comment{Update state for time $t_{i+1}$ by discretizing integral equations} \State $d(t_{i+1}) \gets \min(r_0, 1)f(t_{i+1}) + \sum_{j=0}^{i} \left(z^1_{t_j}\lambda(t_j)+ z^2_{t_j}d(t_j)\right)g(t_{i+1}-t_j)\Delta_{j+1}$ 
\State $\rho_{t_{i+1}} \gets \min(r_0, 1)\bar{F}(t_{i+1}) + \sum_{j=0}^{i} \left(z^1_{t_j}\lambda(t_j) + z^2_{t_j}d(t_j)\right)\bar{G}(t_{i+1}-t_j)\Delta_{j+1}$ 
\State $\eta_{t_{i+1}} \gets \max(r_0-1, 0) + \sum_{j=0}^{i} \left[(z^3_{t_j}-z^1_{t_j})\lambda(t_j) -  z^2_{t_j}d(t_j)\right]\Delta_{j+1}$ 
\EndFor \State \textbf{Output:} The sequences of values for $(\rho, \eta, d)$ and $(z^1, z^2, z^3)$. 
\end{algorithmic} 
\end{breakablealgorithm} 

\subsubsection{Example 1: Zero-buffer Loss System}
We first solve the {VIE \eqref{eq:rhoVolini} using Algorithm~\ref{alg:1}} with error tolerance $\ep=10^{-3}$. The simulated system has $n = 150$ servers and Lognormal$(-0.5, 2)$ service times. Two types of arrival rates are used: \emph{periodic} with  $\lambda(t)=\sfrac{2}{3}(1+\sin(\sfrac{2\pi t}{10}))$ as in Figure~\ref{fig:periodic-zero-1}, and \emph{episodic} with $\la(t) = 0.005 \cdot t(T-t)$ as in Figure~\ref{fig:episodic-zero-1}. The simulated trajectory {$\bar{N}^n$} closely matches the fluid-limit solution $\rho$, as in Figures~\ref{fig:periodic-zero-2} and \ref{fig:episodic-zero-2}, confirming the convergence in Theorem~\ref{thm:fluidlimitini0uni}. Repeating the simulation over $R = 200$ {replications} shows that the empirical blocking probability $\bar{B}^n(t)$ aligns well with the theoretical $1 - w(t)$, supporting Corollary~\ref{cor:acceptprob}. Figures~\ref{fig:periodic-zero-3}-\ref{fig:periodic-zero-4} and \ref{fig:episodic-zero-3}-\ref{fig:episodic-zero-4} illustrate this relationship for $n=150$ and $n=5000$ for both arrival types.

\begin{figure}[h!]
\centering
\begin{subfigure}{0.45\textwidth}
\centering
\includegraphics[width=\linewidth]{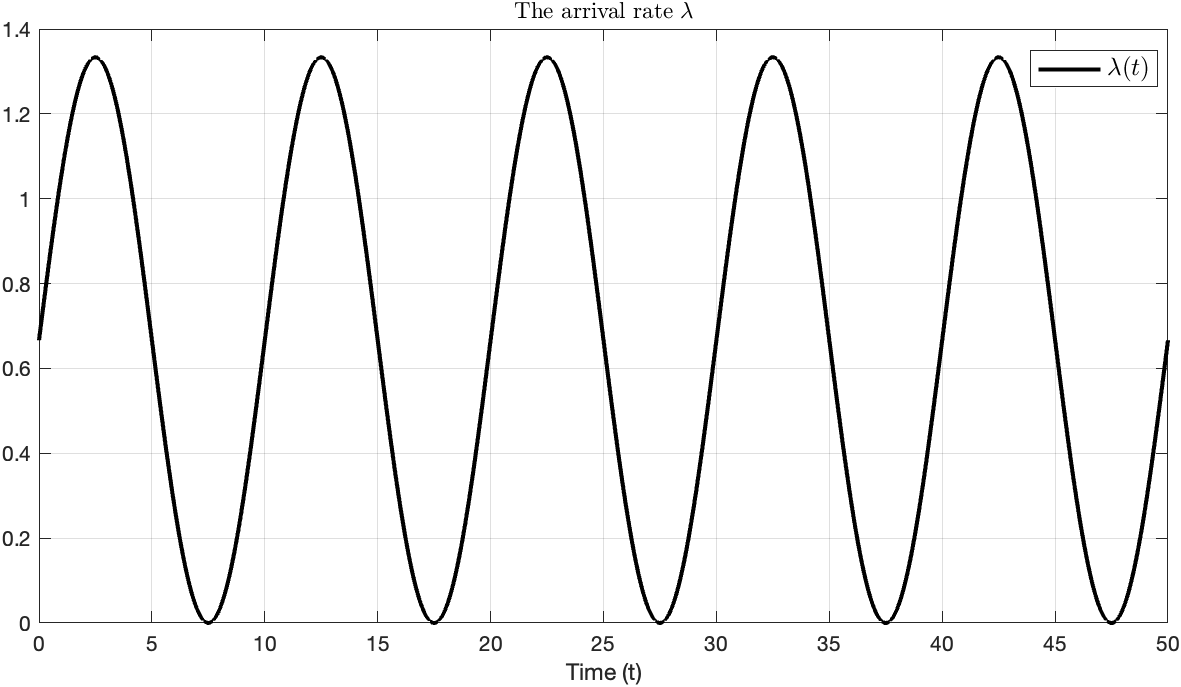}
\caption{}
\label{fig:periodic-zero-1}
\end{subfigure}%
\begin{subfigure}{0.45\textwidth}
\centering
\includegraphics[width=\linewidth]{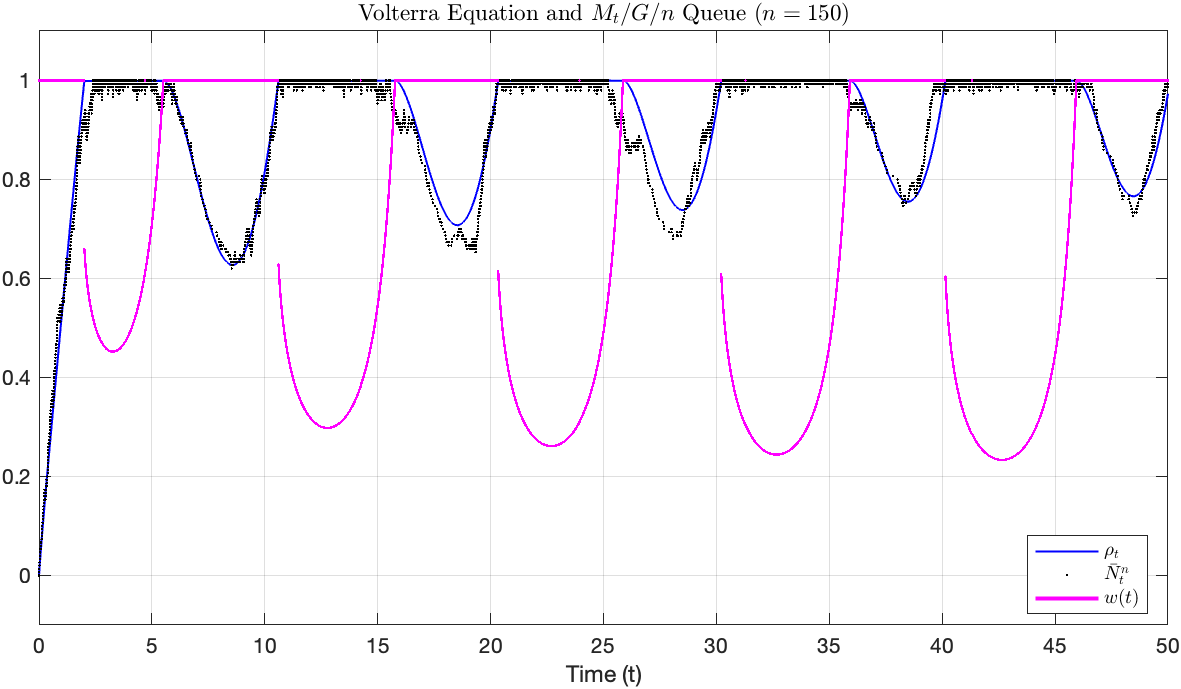}
\caption{}
\label{fig:periodic-zero-2}
\end{subfigure}
\begin{subfigure}{0.45\textwidth}
\centering
\includegraphics[width=\linewidth]{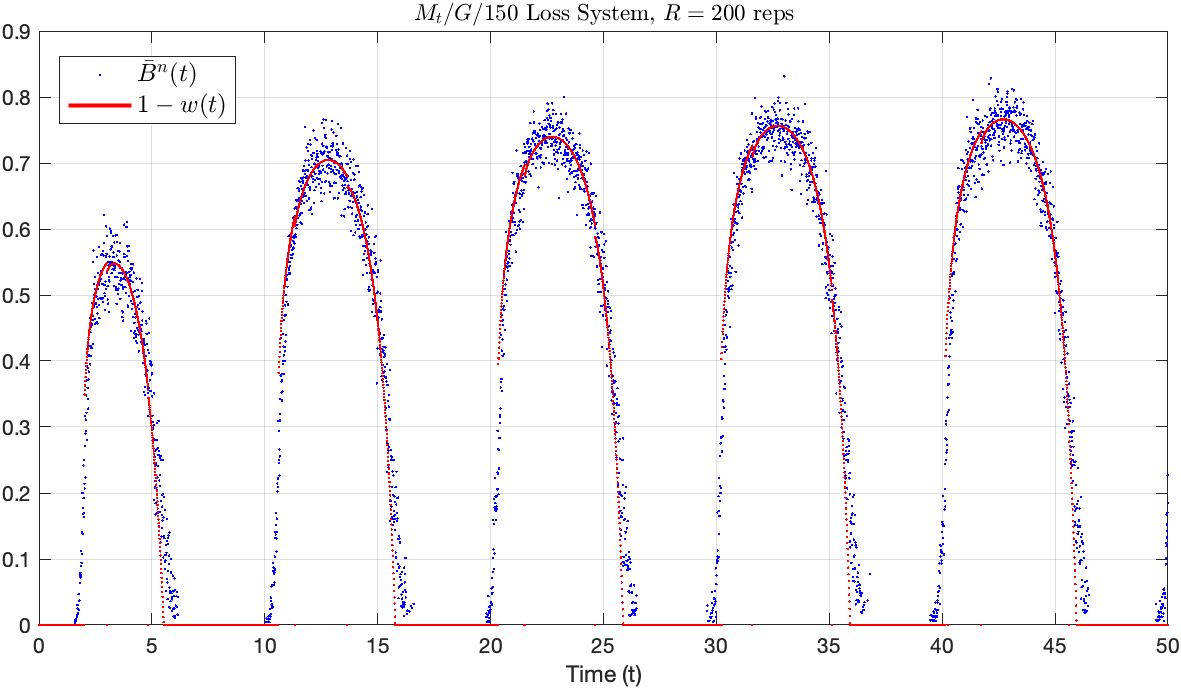}
\caption{}
\label{fig:periodic-zero-3}
\end{subfigure}%
\begin{subfigure}{0.45\textwidth}
\centering
\includegraphics[width=\linewidth]{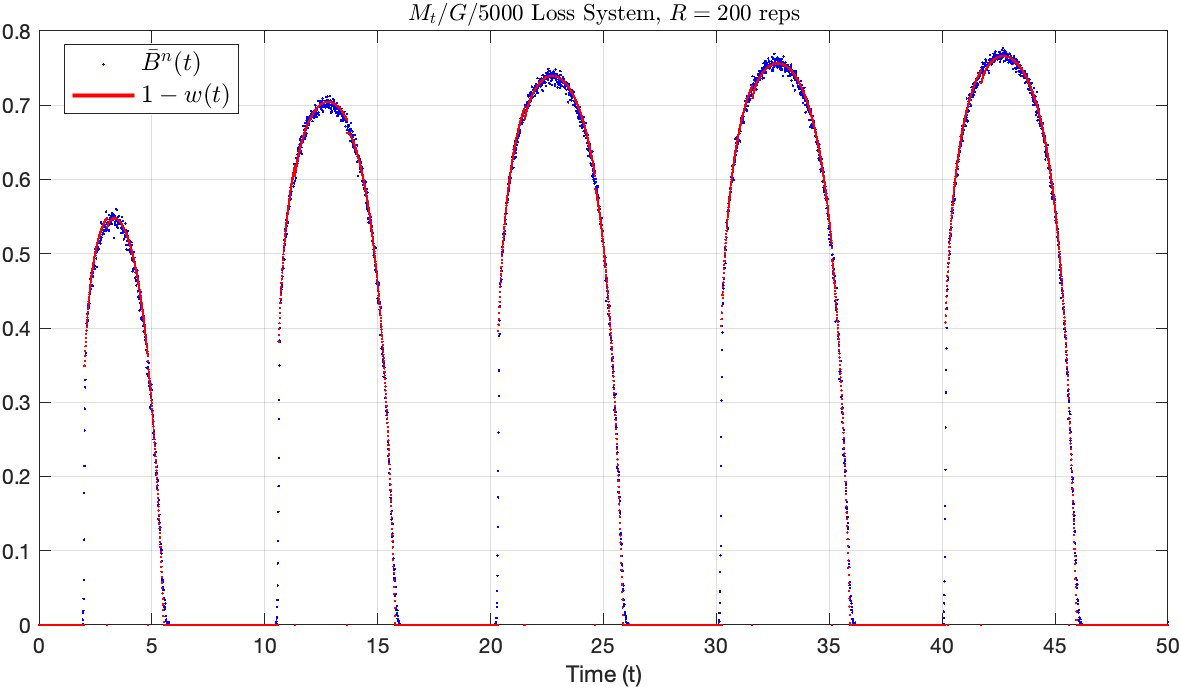}
\caption{}
\label{fig:periodic-zero-4}
\end{subfigure}
\caption{Zero-Buffer Loss Queue with Periodic Arrival Rate}
\label{fig:periodic-zero}
\end{figure}

\begin{figure}
\centering
\begin{subfigure}{0.45\textwidth}
\centering
\includegraphics[width=\linewidth]{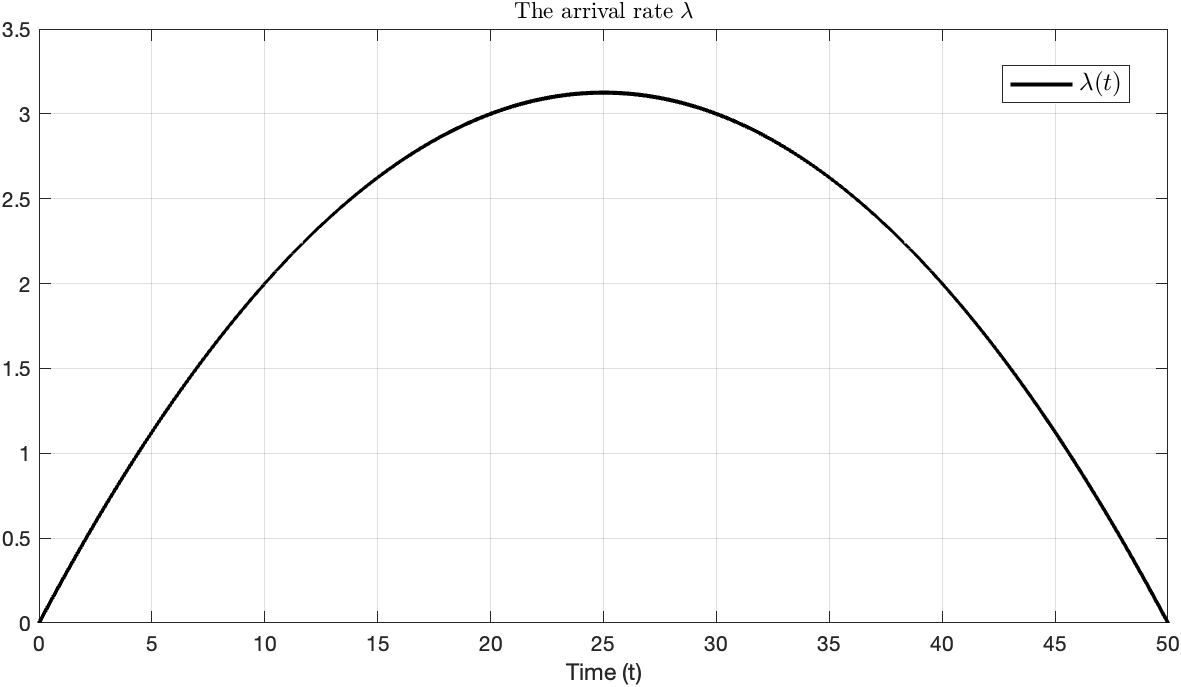}
\caption{}
\label{fig:episodic-zero-1}
\end{subfigure}%
\begin{subfigure}{0.45\textwidth}
\centering
\includegraphics[width=\linewidth]{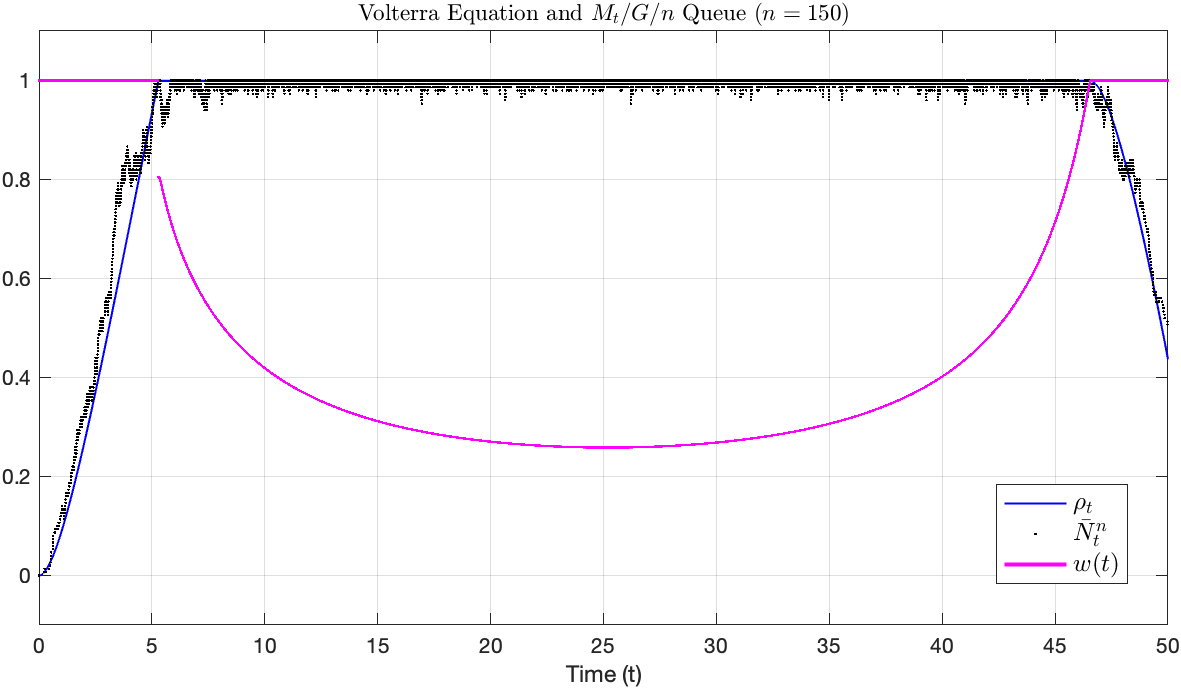}
\caption{}
\label{fig:episodic-zero-2}
\end{subfigure}
\begin{subfigure}{0.45\textwidth}
\centering
\includegraphics[width=\linewidth]{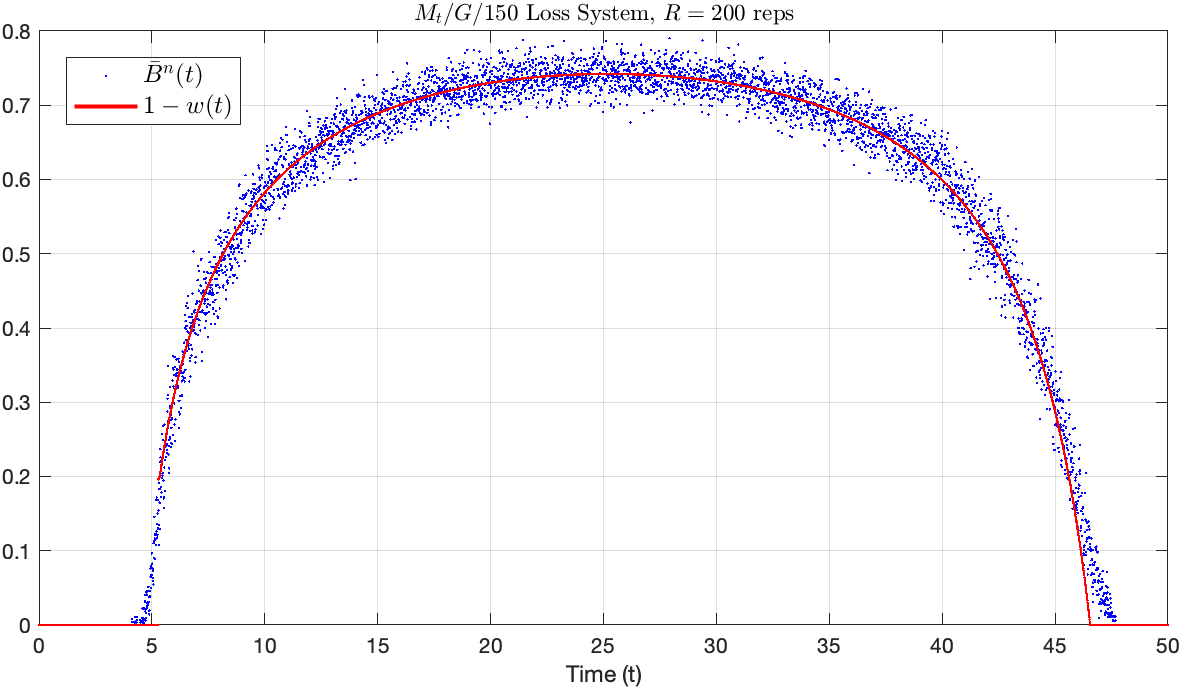}
\caption{}
\label{fig:episodic-zero-3}
\end{subfigure}
\begin{subfigure}{0.45\textwidth}
\centering
\includegraphics[width=\linewidth]{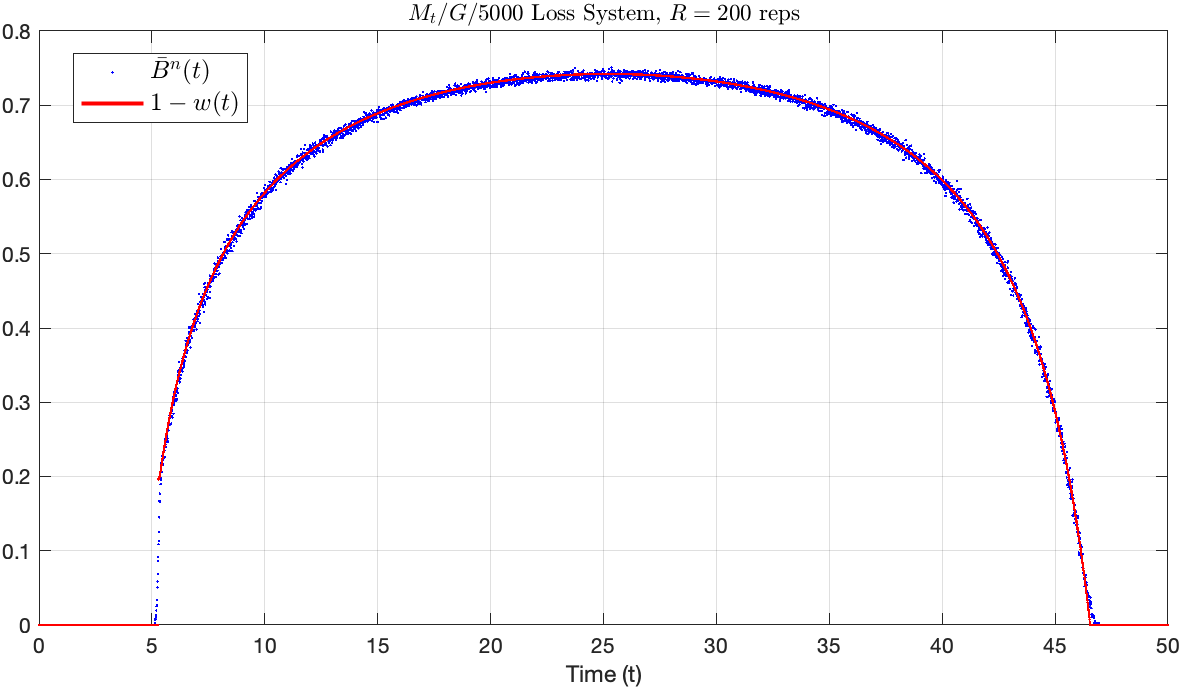}
\caption{}
\label{fig:episodic-zero-4}
\end{subfigure}
\caption{Zero-Buffer Loss Queue with Episodic Arrival Rate}
\label{fig:episodic-zero}
\end{figure}

\subsubsection{Example 2: Finite-buffer Loss System}
Next, we solve the VIE system \eqref{eq:rhoVoln-init} using Algorithm~\ref{alg:2} with error tolerance $\ep=10^{-3}$. Again, the system has $n = 150$ servers and Lognormal$(-0.5, 1.2)$ service times, with two types of arrival rates: \emph{periodic} with $\lambda(t)=\frac{2}{3}(1.5+\sin(\tfrac{2\pi t}{10}))$ as in Figure~\ref{fig:periodic-finite-1} and \emph{episodic} with $\la(t) = 0.005 \cdot t(T-t)$ as in Figure~\ref{fig:episodic-finite-1}. The simulated trajectories $\bar{S}^n$ and $\bar{Q}^n$ track the theoretical $(\rho,\eta)$ closely, as in Figures~\ref{fig:periodic-finite-2} and \ref{fig:episodic-finite-2}, confirming the convergence in Theorem~\ref{thm:fluidlimitbuf}. Similarly, the blocking probability $\bar{B}^n(t)$ aligns with $1 - w^3(t)$, validating the finite-buffer fluid approximation in Corollary~\ref{cor:acceptprobn}. This is illustrated in Figures~\ref{fig:periodic-finite-3}-\ref{fig:periodic-finite-4} and \ref{fig:episodic-finite-3}-\ref{fig:episodic-finite-4} for $n=150$ and $n=5000$ both arrival types.

\begin{figure}
\centering
\begin{subfigure}{0.45\textwidth}
\centering
\includegraphics[width=\linewidth]{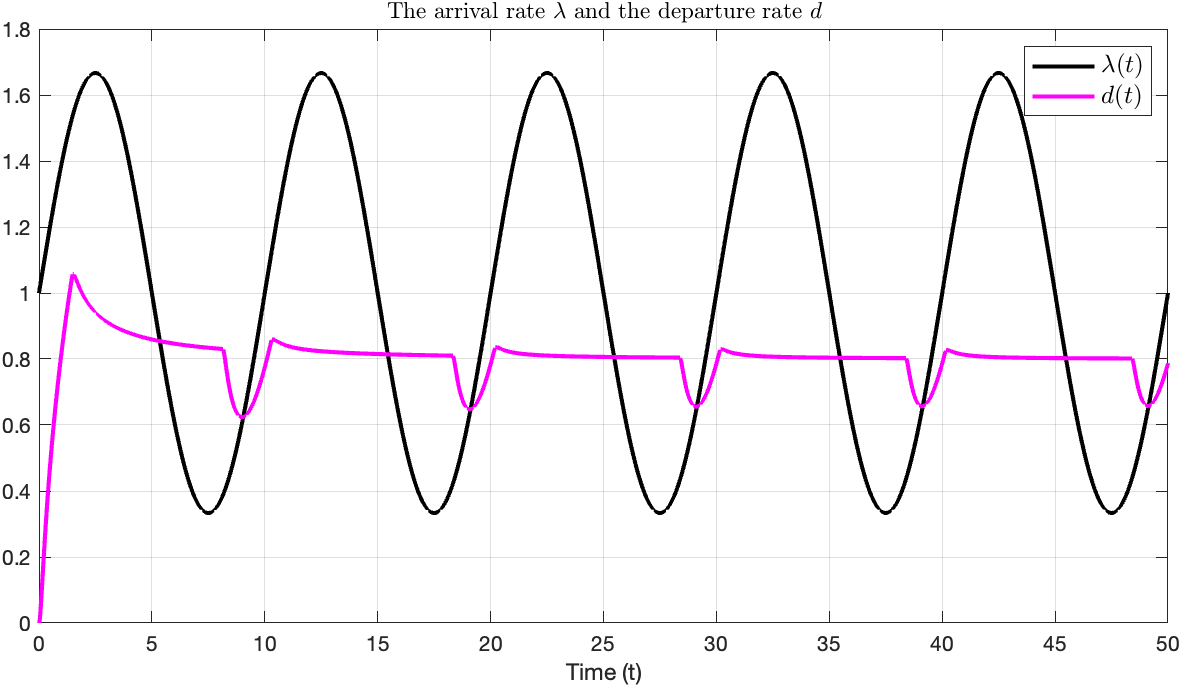}
\caption{}
\label{fig:periodic-finite-1}
\end{subfigure}%
\begin{subfigure}{0.45\textwidth}
\centering
\includegraphics[width=\linewidth]{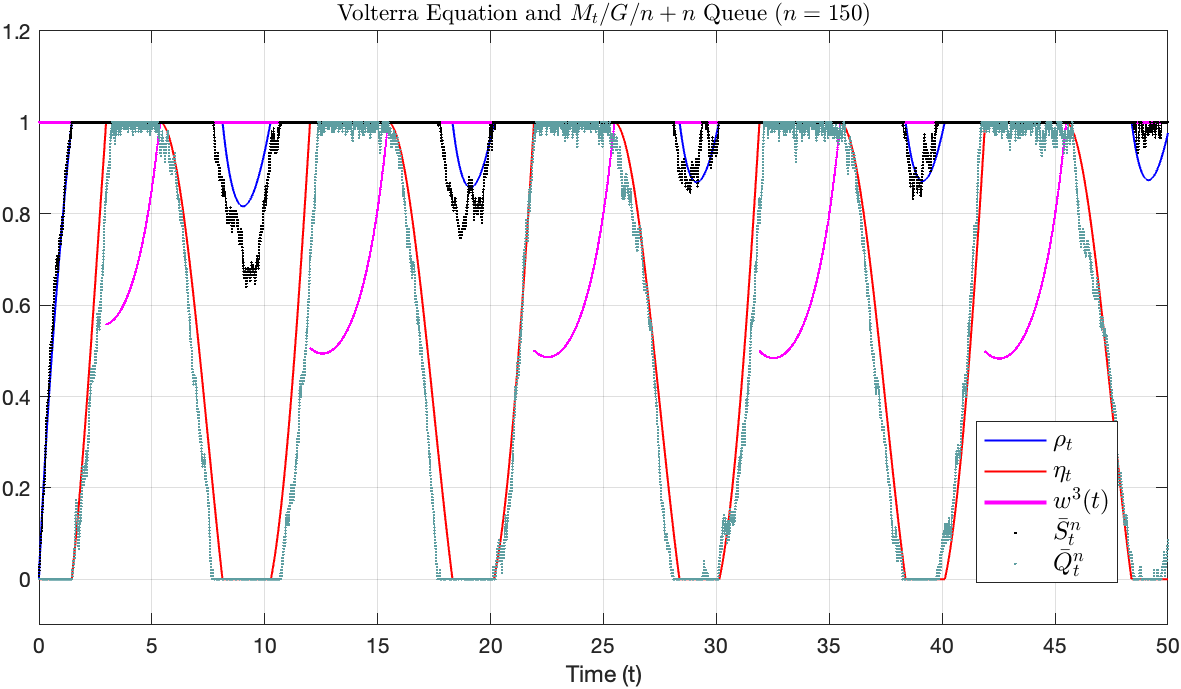}
\caption{}
\label{fig:periodic-finite-2}
\end{subfigure}
\begin{subfigure}{0.45\textwidth}
\centering
\includegraphics[width=\linewidth]{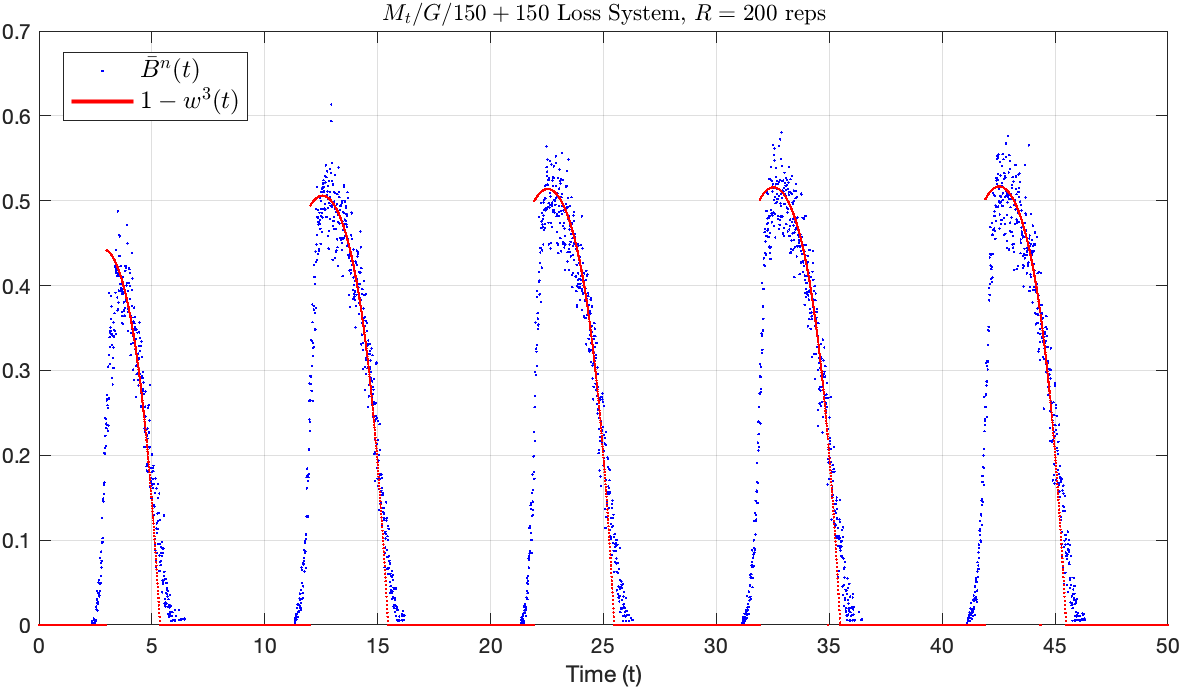}
\caption{}
\label{fig:periodic-finite-3}
\end{subfigure}
\begin{subfigure}{0.45\textwidth}
\centering
\includegraphics[width=\linewidth]{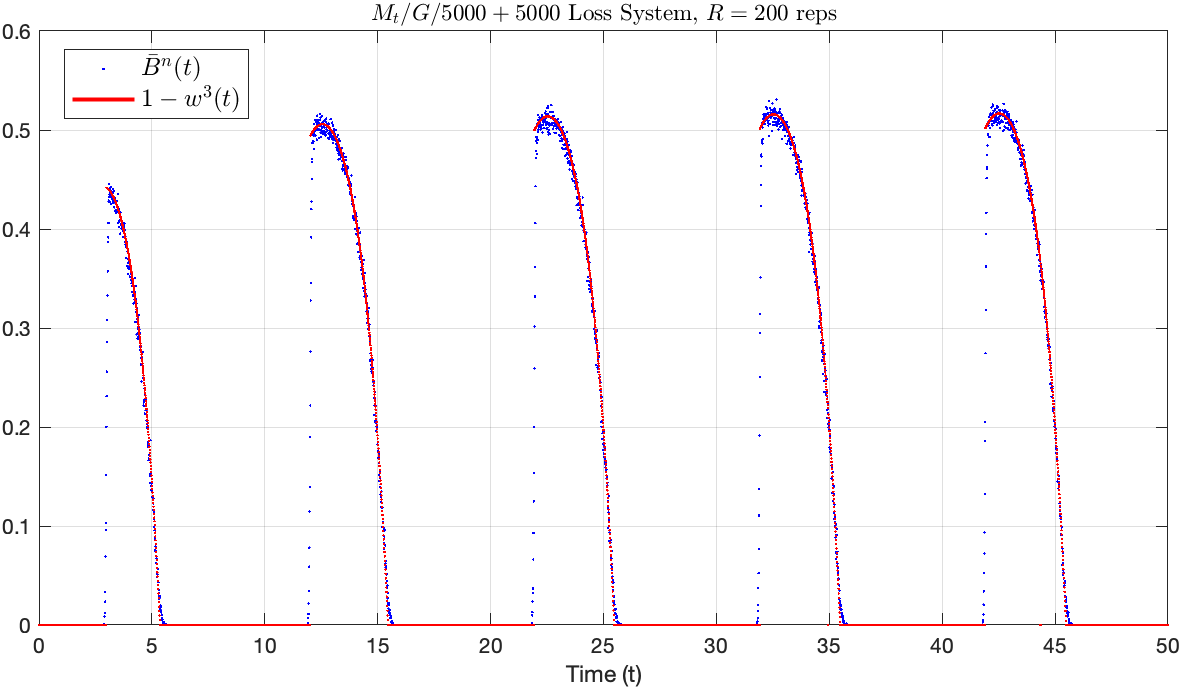}
\caption{}
\label{fig:periodic-finite-4}
\end{subfigure}
\caption{Finite Buffer Loss Queue with Periodic Arrival Rate}
\label{fig:periodic-finite}
\end{figure}

\begin{figure}
\centering
\begin{subfigure}{0.45\textwidth}
\centering
\includegraphics[width=\linewidth]{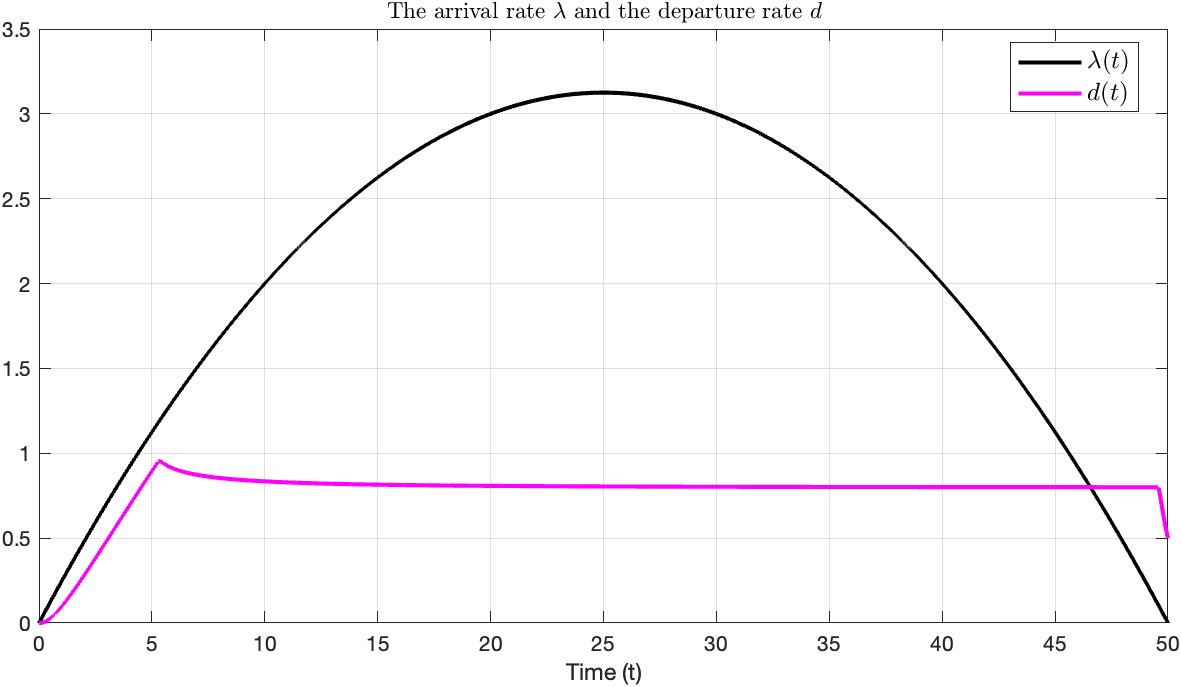}
\caption{}
\label{fig:episodic-finite-1}
\end{subfigure}%
\begin{subfigure}{0.45\textwidth}
\centering
\includegraphics[width=\linewidth]{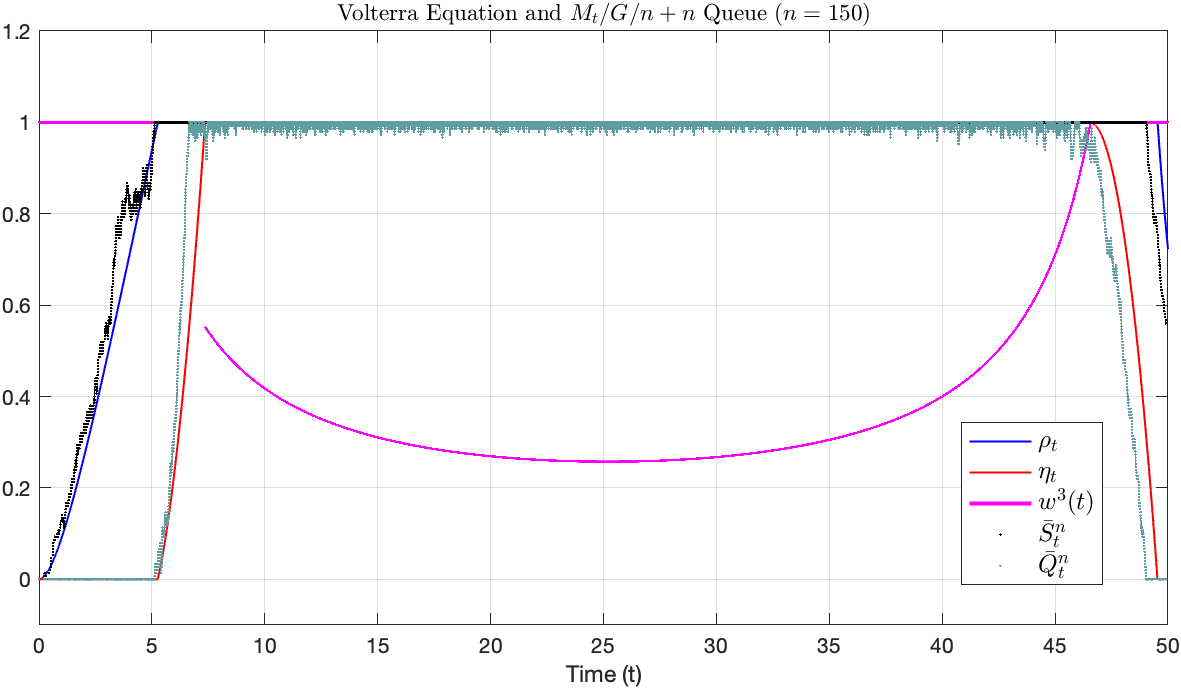}
\caption{}
\label{fig:episodic-finite-2}
\end{subfigure}
\begin{subfigure}{0.45\textwidth}
\centering
\includegraphics[width=\linewidth]{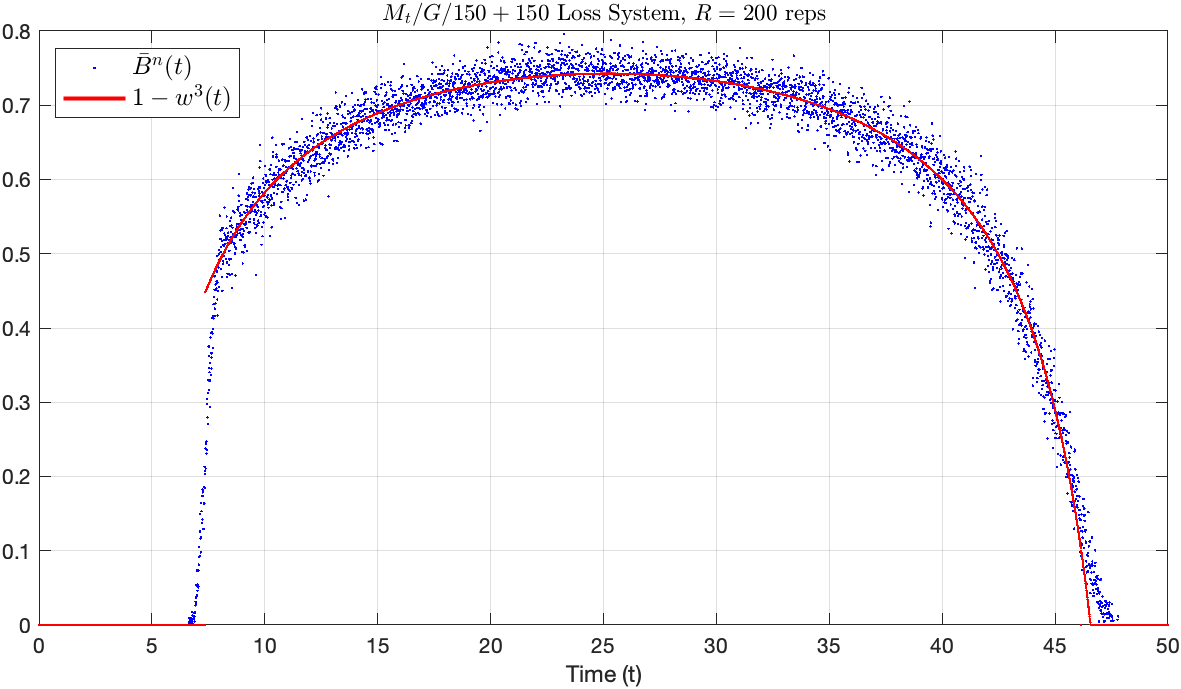}
\caption{}
\label{fig:episodic-finite-3}
\end{subfigure}
\begin{subfigure}{0.45\textwidth}
\centering
\includegraphics[width=\linewidth]{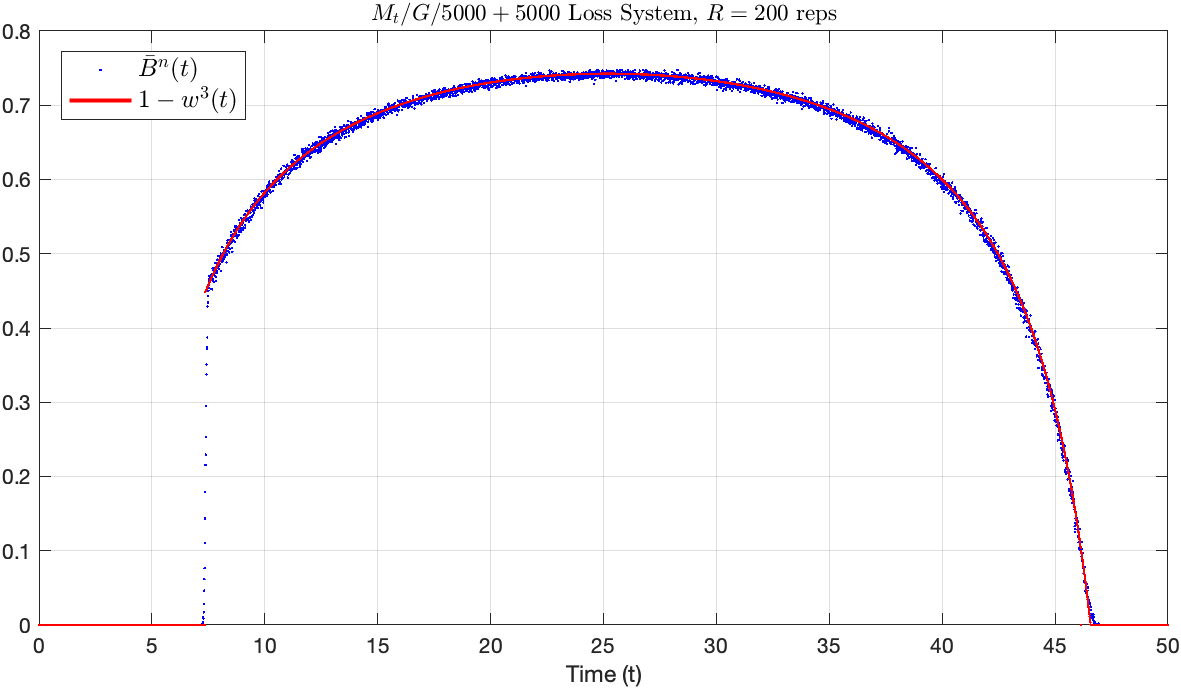}
\caption{}
\label{fig:episodic-finite-4}
\end{subfigure}
\caption{Finite Buffer Loss Queue with Episodic Arrival Rate}
\label{fig:episodic-finite}
\end{figure}

\begin{remark}
As noted in Remark~\ref{rk:wdiscontinuous} and \ref{rk:w3discontinuous}, the auxiliary functions $w$ and $w^3$ may exhibit discontinuities. In the numerical results, this discontinuity becomes evident as system size $n$ increases (e.g., $n=5000$). For smaller systems, the blocking probability appears smoother, but the underlying discontinuity emerges clearly in the large-system limit. 
\end{remark}

\subsection{Operational Perspectives.}
Finite capacity is a defining characteristic of many real-world service systems, such as call centers, emergency departments, and cloud resource pools. Such systems exhibit non-stationary queueing dynamics due to time-varying arrivals and general service times, making accurate transient analysis crucial for operational insights.

Our fluid limits for the zero-buffer $M_t/G/n/n$ and finite-buffer $M_t/G/n/n+b_n$ provide first-order approximations of the system occupancy, queue length, departure process, and acceptance probabilities as $n \to \infty$. These deterministic approximations offer a tractable foundation for operational optimization: they allow one to compute time-varying blocking probabilities directly and to optimize server and buffer capacities against transient performance constraints.

\subsubsection{Staffing Optimization in Zero-Buffer Systems.} 

\sloppy
Consider a sequence of non-stationary $M_t/G/c_n/c_n$ loss systems, with $c_n = \lfloor n c \rfloor$ servers and arrivals satisfying Assumption~\ref{asm:initial0}. For simplicity, assume that the system starts empty. Let the scaled number in the system or the proportion of occupied servers in the $n-$th system be $\bar{N}_t^n$ and the $n-$scaled cumulative departure process be $\bar{D}_t^n$. Then,
repeating Section~\ref{sec:zero-buffer} with the capacity level $1$ replaced by $c$ throughout, so that $\mathbbm{1}_{\{\bar{N}^n_{T_i-}<1\}}$ becomes $\mathbbm{1}_{\{\bar{N}^n_{T_i-}<c\}}$ in \eqref{eq:rho^n1} and $\mathbbm{1}_{\{\rho_{u-}<1\}}$ becomes $\mathbbm{1}_{\{\rho_{u-}<c\}}$ in \eqref{eq:rhoVolini}, and noting that every argument in the proofs of Proposition~\ref{prop:weaklimits0} and Theorems~\ref{thm:solVol}--\ref{thm:fluidlimitini0uni} uses the capacity only through that level set, we obtain
$$
\lim_{n \to \infty} \sup_{t \in [0,T]} \lln \bar{N}_t^n - \rho_t \rrn  = 0, \quad \lim_{n \to \infty} \sup_{t \in [0,T]} \lln \bar{D}_t^n - D_t \rrn = 0,
$$
almost surely where $D(t)$ is the fluid cumulative departure rate given by $D_t = \int_0^t d(u) du$ and $d(\cdot)$ is the fluid instantaneous departure rate whose dynamics is presented below. In addition, $\rho_t$ solves the discontinuous non-linear VIE given by
\beq\label{eq:rho-c-0}
\rho_t = \int_0^t \mathbbm{1}_{\{ \rho_{u-}< c \}} \bar{G}(t-u)\la_u du.
\eeq
We note that the number of servers could be time-varying with $c_n(t) = \lfloor n c(t) \rfloor$ in which case our results will remain valid as long as the capacity constraint $\mathbbm{1}_{\{\rho_{u-}< c_u\}}$ is incorporated in \eqref{eq:rho-c-0}. However for simplicity we consider $c(\cdot)$ to be constant.
Under Definition~\ref{def:sol_Vol}, \eqref{eq:rho-c-0} can be equivalently expressed as
\beq\label{eq:rho-c}
\rho_t = \int_0^t w_c(u) \bar{G}(t-u)\la(u) du,
\eeq
where 
\beq\label{eq:w-c}
w_c(t) = 
\begin{cases}
1, \quad &\text{if }\rho_t < c,\\
\frac{d(t)}{\la(t)} \wedge 1, \quad &\text{if }\rho_t = c,
\end{cases}
\eeq
and
\beq\label{eq:d-c}
d(t) = \int_0^t w_c(u) g(t-u)\la(u)du.
\eeq
Furthermore the acceptance probability converges in the sense of Corollary~\ref{cor:acceptprob}:
\beq\label{eq:ap-conv-1}
 \mathbb{P}(\bar{N}_{t-}^n < c_n/n)\la(t) 
\rightharpoonup
w_c(t) \la(t)
\text{ in }
L^1[0,T],
\eeq
and, for every compact $K\subset\{t:\rho_t<c\}$, $\sup_{t\in K}\lln\mathbb{P}(\bar{N}^n_{t-}<c_n/n)-1\rrn\rightarrow0$.

Equations~\eqref{eq:rho-c}-\eqref{eq:d-c} and the convergence result~\eqref{eq:ap-conv-1} direct server capacity optimization while maintaining the transient blocking probability above a given threshold. This is particularly relevant in emergency departments or call centers where it is important from a managerial perspective to minimize the total number of blocked patients or customers. Summarizing, we solve the following problem in the fluid limit
$$
\min c, \text{ such that } \inf_{t \in [0,T]} w_c(t) \geq 1-\al,
$$
where $\al$ is the maximum allowable instantaneous blocking probability. {Since the infimum of the acceptance probability $w$ increases as the capacity $c$ increases, the problem admits a unique optimal solution to the above constrained optimization problem. Therefore we can apply standard root-finding techniques (e.g. the bisection method) to obtain the optimal server capacity $c^{\ast}$. For each value of $c$, we solve the discontinuous VIE by Algorithm~\ref{alg:1} to get $\inf_{t \in [0,T]} w_c(t)$, and stop searching once $\inf_{t \in [0,T]} w_{c^\ast}(t)=1-\al$.} Numerical results are presented in Figures~\ref{fig:per-c-ast-zero} and~\ref{fig:epi-c-ast-zero}, respectively, for periodic and episodic arrival rates, and for two choices of $\alpha$. 

\begin{figure}
\centering
\begin{subfigure}{.49\textwidth}
\centering
\includegraphics[width=\linewidth]{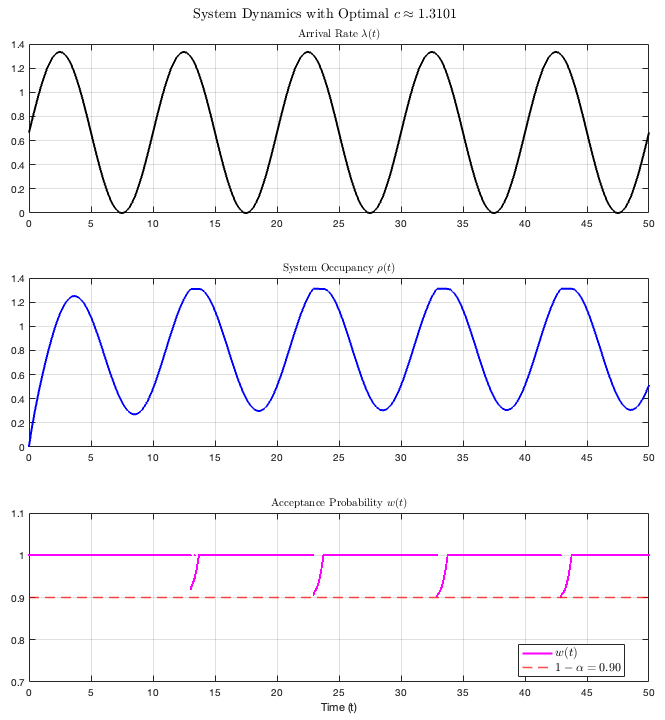}
\caption{Blocking probability ${\leq} \al = 0.1$}
\label{fig:per-c*-al=0.1}
\end{subfigure}%
\begin{subfigure}{.49\textwidth}
\centering
\includegraphics[width=\linewidth]{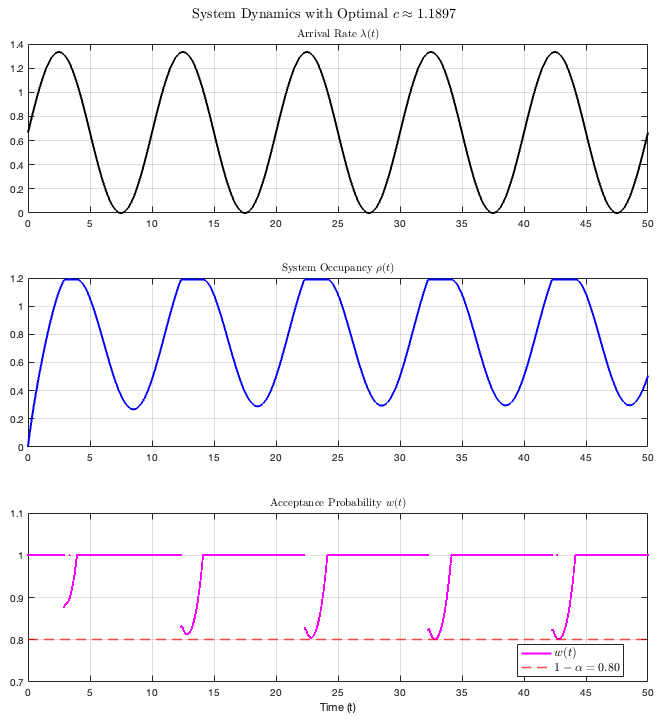}
\caption{Blocking probability ${\leq} \al = 0.2$}
\label{fig:per-c*-al=0.2}
\end{subfigure}
\caption{Optimal server capacity in zero-buffer loss queue with periodic arrival rate.}
\label{fig:per-c-ast-zero}
\end{figure}

\begin{figure}[h!]
\centering
\begin{subfigure}{.49\textwidth}
\centering
\includegraphics[width=\linewidth]{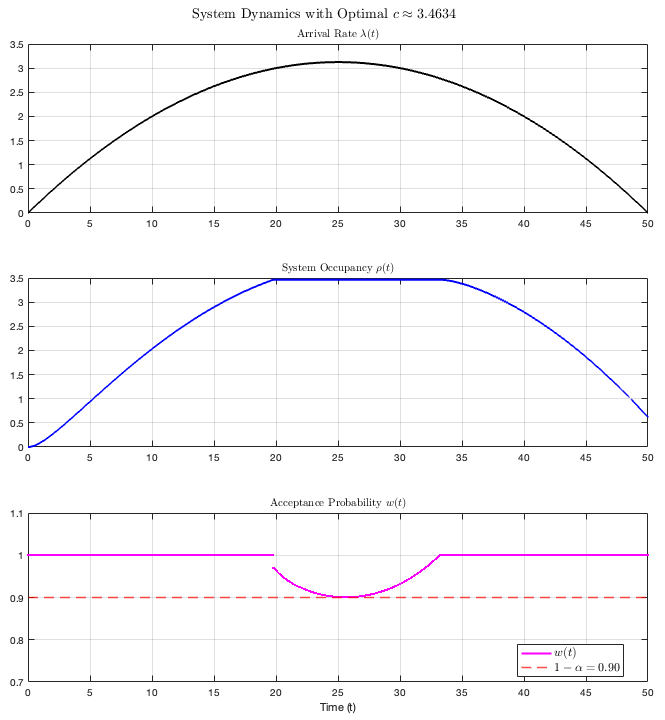}
\caption{Blocking probability ${\leq} \al = 0.1$}
\label{fig:epi-c*-al=0.1}
\end{subfigure}%
\begin{subfigure}{.49\textwidth}
\centering
\includegraphics[width=\linewidth]{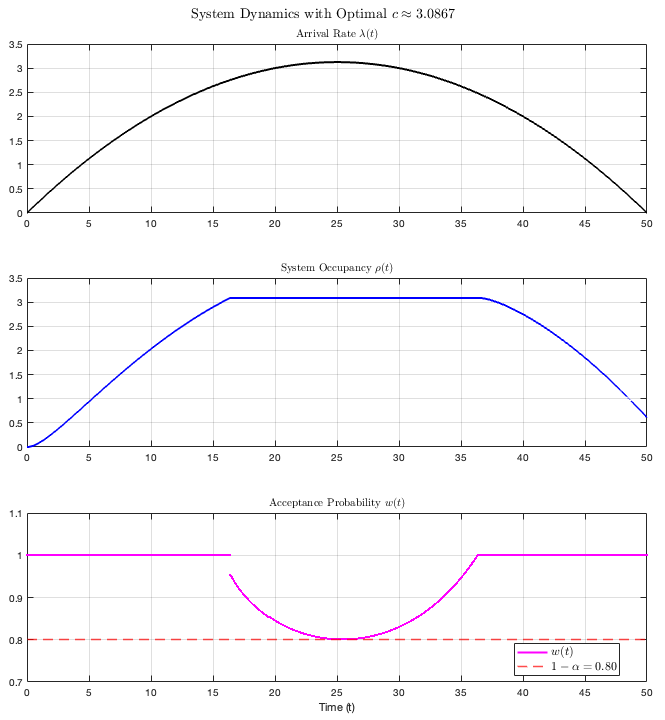}
\caption{Blocking probability ${\leq} \al = 0.2$}
\label{fig:epi-c*-al=0.2}
\end{subfigure}
\caption{Optimal server capacity in zero-buffer loss queue with episodic arrival rate.}
\label{fig:epi-c-ast-zero}
\end{figure}

\subsubsection{Joint Staffing and Buffer Capacity Optimization.} Now consider a sequence of non-stationary $M_t/G/c_n/c_n + b_n$ loss queueing systems, where $c_n = \lfloor n c \rfloor$, buffer size $b_n = \lfloor n \beta \rfloor$, and arrivals satisfy Assumption~\ref{asm:queue}. Similar to the previous case, server and buffer size could be time-varying with $c_n(t) = \lfloor n c(t) \rfloor$ and $b_n(t) = \lfloor n \beta(t)\rfloor$, and our results would still hold valid as long as $c(\cdot)$ and $\beta(\cdot)$ are piecewise constant. However, for simplicity, we do not consider those generalizations and also assume that the system starts empty. Let the scaled number being served or the proportion of occupied servers be $\bar{S}_t^n$, the scaled number waiting in buffer be $\bar{Q}_t^n$ and the $n-$scaled cumulative departure process be $\bar{D}_t^n$. Then similar to the treatise done in Section~\ref{sec:finite-buffer} and Theorem~\ref{thm:fluidlimitbuf}, we let the patient reader work out the details to conclude that
$$
\lim_{n \to \infty} \sup_{t \in [0,T]} \lln \bar{S}_t^n - \rho_t \rrn  = 0, \quad \lim_{n \to \infty} \sup_{t \in [0,T]} \lln \bar{Q}_t^n - \eta_t \rrn = 0 \quad \lim_{n \to \infty} \sup_{t \in [0,T]} \lln \bar{D}_t^n - D_t \rrn = 0, 
$$
almost surely where $D(t)$ is the fluid cumulative departure rate given by $D_t = \int_0^t d(u) du$ and $d(\cdot)$ is the fluid instantaneous departure rate whose dynamics is presented below. In addition, the fluid limits $(\rho, \eta, d)$ solves a coupled discontinuous nonlinear VIE system which interpreted according to Definition~\ref{def:sol_Vol} reads
\begin{align}\label{eq:rho-c-b}
&\rho_{t}=\int_{0}^{t} w_{c,\beta}^{1}(u) \bar{G}(t-u) \lambda(u) d u  +\int_{0}^{t} w_{c,\beta}^{2}(u) \bar{G}(t-u) d(u) d u, \nonumber \\
&\eta_{t}=\int_{0}^{t} (w_{c,\beta}^{3}(u)-w_{c,\beta}^1(u)) \lambda(u) d u  -\int_{0}^{t} w_{c,\beta}^{2}(u) d(u) d u, \nonumber\\
&D_t=\int_{0}^{t}w_{c,\beta}^{1}(u) G(t-u) \lambda(u) d u  +\int_{0}^{t} w_{c,\beta}^{2}(u) G(t-u) d(u) d u,
\end{align}
where the auxiliary functions $w_{c,\beta}^1$, $w_{c,\beta}^2$, $w_{c,\beta}^3$ are given, as in \eqref{eq:z123}, by:
\begin{align}\label{eq:w-c-b}
w_{c,\beta}^1(t)=1,w_{c,\beta}^2(t)=0,w_{c,\beta}^3(t)=1, &\qquad\rho_t<c,\eta_t=0 \nonumber\\
w_{c,\beta}^1(t)=1,w_{c,\beta}^2(t)=0,w_{c,\beta}^3(t)=1, &\qquad \rho_t=c,\eta_t=0 \nonumber\\
w_{c,\beta}^1(t)=0,w_{c,\beta}^2(t)=1,w_{c,\beta}^3(t)=1, &\qquad \rho_t=c,0<\eta_t<\beta \nonumber\\
w_{c,\beta}^1(t)=0,w_{c,\beta}^2(t)=1,w_{c,\beta}^3(t)=d(t)/\lambda(t)\wedge1, &\qquad \rho_t=c,\eta_t=\beta
\end{align}
The acceptance probability again satisfies, in the sense of Corollary~\ref{cor:acceptprobn}:
\beq\label{eq:ap-conv-2}
\mathbb{P}(\bar{Q}_{t-}^n < \frac{b_n}{n})\la(t) 
\rightharpoonup
w_{c,\beta}^3(t)\la(t)
\text{ in }
L^{1}[0,T].
\eeq
Using equations~\eqref{eq:rho-c-b}-\eqref{eq:w-c-b} and the convergence result~\eqref{eq:ap-conv-2}, we can formulate a joint staffing-buffer optimization problem in the fluid limit constrained to maintain the transient blocking probability above a threshold:
$$
\min v\cdot c+ (1-v) \cdot \beta, \text{ such that }  \inf_{t \in [0,T]} w_{c,\beta}^3(t) \geq 1-\al,
$$
where $v$ weights the relative cost of servers and buffer space, and $\al$ denotes the maximum allowable instantaneous blocking probability. For each $c$, the infimum of the acceptance probability $w^3$ increases with buffer size $\beta$. Thus, there exists a unique $\beta_c$ such that $\inf_{t \in [0,T]} w^3_{c,\beta_c}(t) = 1 - \al$. We can perform a grid search for $c$ and apply standard root-finding techniques (e.g., bisection method) to determine the optimal $\beta_c$ for each $c$. At each $c$ and $\beta$, we solve the discontinuous VIE system using Algorithm~\ref{alg:2} to obtain $\inf_{t \in [0,T]} w^3_{c,\beta}(t)$. The search terminates when $\inf_{t \in [0,T]} w^3_{c,\beta_c}(t) = 1 - \al$. The optimal solution is the capacity $c^\ast$ that minimizes $v \cdot c^\ast + (1 - v) \cdot \beta_{c^\ast}$ during the grid search. Numerical results are presented in Figure~\ref{fig:c*-b*-fin}. It is important to note that this solution provides a rudimentary approach to solving the constrained optimization problem. While there may be more effective optimization techniques available, the focus of this paper is not on exploring such alternative solutions.
\begin{figure}
\begin{subfigure}{.49\textwidth}
\includegraphics[width=\linewidth]{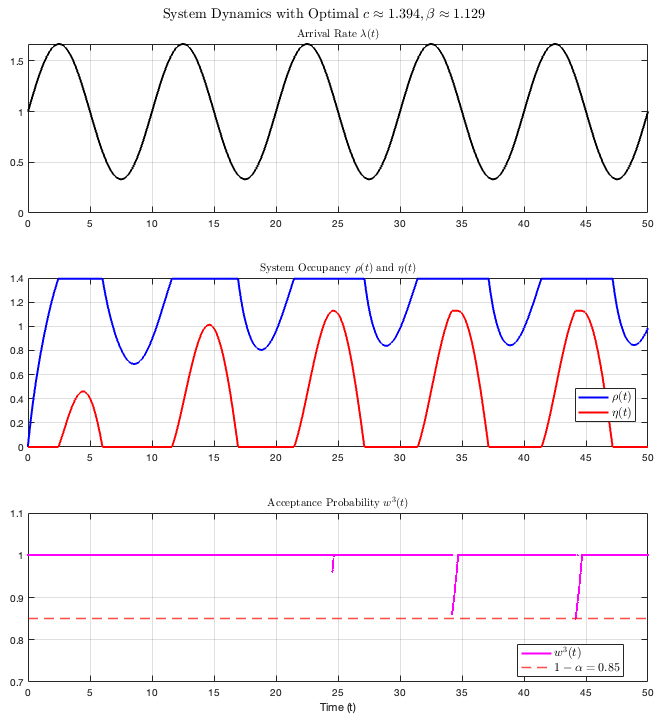}
\caption{}
\end{subfigure}%
\begin{subfigure}{.49\textwidth}
\includegraphics[width=\linewidth]{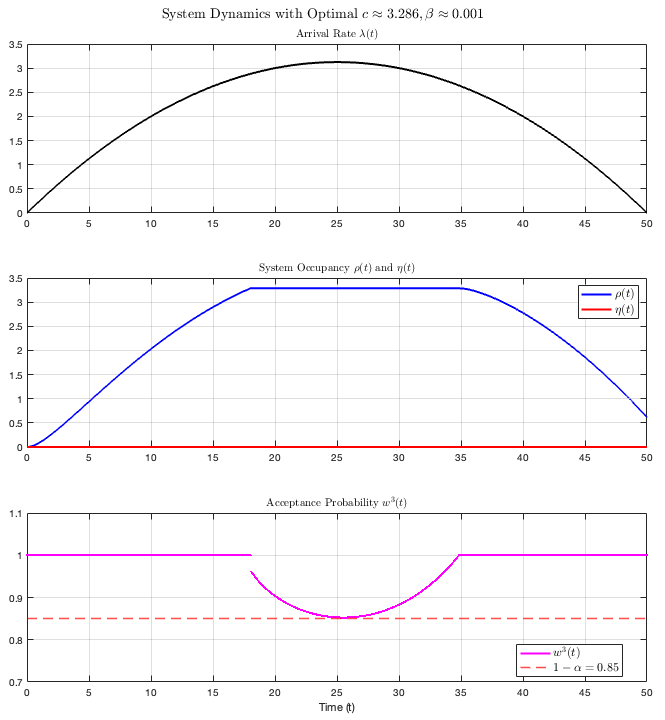}
\caption{}
\end{subfigure}%
\caption{Optimal server and buffer capacity with periodic and episodic arrival rates.}
\label{fig:c*-b*-fin}
\end{figure}

\section{Conclusion}\label{sec:conclusion}
This paper developed a unified fluid-limit framework for nonstationary many-server loss systems with general service-time distributions. In the first part, we established a functional strong law of large numbers for the zero-buffer $M_t/G/n/n$ model via a discontinuous Volterra integral equation representation. The second part extended this analysis to the finite-buffer $M_t/G/n/(n+b_n)$ model, showing that the joint dynamics of servers, buffer occupancy, and departures satisfy a coupled system of discontinuous Volterra equations. In both regimes, we proved existence and uniqueness of the limiting trajectories and convergence of the associated time-varying acceptance and blocking probabilities.

The results demonstrate that deterministic fluid models can accurately describe transient behavior in large-scale, non-Markovian, time-varying loss systems. The discontinuous Volterra structure captures admission control and boundary effects within a mathematically rigorous and computationally tractable framework, bridging the gap between asymptotic theory and operational approximation.

Beyond theoretical insight, the model provides a practical basis for performance evaluation and real-time decision-making. We showed how the fluid limit can be used for optimal staffing and buffer capacity design, and the same structure can naturally extend to dynamic control settings. Future work may pursue diffusion refinements, stochastic perturbation analysis, and optimization-based control formulations, further integrating transient queueing dynamics into the broader landscape of stochastic operations management.

\bigskip

\bibliographystyle{plain}

\bibliography{ref_corrected}

\end{document}